\documentclass[a4paper, 12pt]{article}

\usepackage{amscd}
\usepackage{amsmath,amsthm,amssymb,amsfonts,dsfont,amscd,amsrefs}
\usepackage{mathrsfs}
\usepackage{esint}

\usepackage{pgfplots}
\pgfplotsset{compat=1.18}

\author{Sylvie Monniaux}

\title{A lecture on Navier-Stokes equations}

\date{}

\theoremstyle{plain}
\newtheorem{proposition}{Proposition}[section]
\newtheorem{theorem}[proposition]{Theorem}

\newtheorem{lemma}[proposition]{Lemma}

\theoremstyle{definition}
\newtheorem{definition}[proposition]{Definition}
\newtheorem{example}[proposition]{Examples}
\newtheorem{exercise}{Exercise}
\newtheorem{notation}[proposition]{Notation}

\theoremstyle{remark}
\newtheorem{remark}[proposition]{Remark}

\numberwithin{equation}{section}

\begin{document}
 
\maketitle

\abstract{The content of the following pages was part of a special topics lecture given at the Australian National University during 
the first semester of 2026. That was a very nice expericence, I really enjoyed giving this 12 weeks (2 hours a week) lecture.
It is meant to be self contained, starting with results on Fourier transform and Sobolev spaces. As a toy model,
before treating the Navier-Stokes system, we focus on the non linear heat equation where the non linearity is polynomial. Ultimately,
we prove existence and uniqueness of mild solutions of the Navier-Stokes equations in critical spaces. This script contains some exercises
and the text of the mid-semester exam as well as the final exam.}

{\scriptsize{
\tableofcontents}
}

\section{Fourier transform, Sobolev spaces}

\subsection{Fourier transform}

On ${\mathds{R}}^n$ ($n \ge 1$), we consider the {\tt normalised Lebesgue measure} 
\[
{\rm d}{\rm m}_n = \frac{1}{(2\pi)^{\frac n2}} \,{\rm d}\lambda_n
\]
where $\lambda_n$ is the classical Lebesgue measure on ${\mathds{R}}^n$. 
The spaces $L^p$ ($1\le p < \infty$) that we consider here (normed by ${\rm dm}_n$) concern functions defined on 
${\mathds{R}}^n$, with complex values. For $f \in L^p({\mathds{R}}^n, {\mathds{C}})$:
\[
\|f\|_p = \left(\int_{{\mathds{R}}^n} |f|^p {\rm d}{\rm m}_n\right)^{\frac 1p}.
\]
The {\tt convolution} between two functions $f$ and $g$ (as long as the integrals exist) is given by
\[
f*g(x) = \int_{{\mathds{R}}^n} f(x-y) g(y)\, {\rm dm}_n(y) = \int_{{\mathds{R}}^n} f(y) g(x-y) \,{\rm dm}_n(y), \quad x \in {\mathds{R}}^n.
\]

\begin{proposition}
Let $1\le p<\infty$.
If $f \in L^1({\mathds{R}}^n)$ and $g \in L^p({\mathds{R}}^n)$, then $f*g \in L^p({\mathds{R}}^n)$ and it holds 
\[
\|f*g\|_p \le \|f\|_1 \|g\|_p.
\]
\end{proposition}

\begin{exercise}
What happens if $p=\infty$? {\sl Hint:} think of the space of bounded continuous functions.
\end{exercise}

\begin{definition}
For $f \in L^1({\mathds{R}}^n)$, the {\tt Fourier transform} of $f$ is the function $\hat{f}$ (also denoted by ${\mathscr{F}} f$) defined by
\begin{equation}
\label{eq:defFourier}
\hat{f} (\xi) = \int_{{\mathds{R}}^n} f e_{-\xi} \,{\rm dm}_n = \int_{{\mathds{R}}^n} f(x) e^{-i\xi\cdot x} \,{\rm dm}_n(x), \quad \xi \in {\mathds{R}}^n,
\end{equation}
where for all $t\in {\mathds{R}}^n$, $e_{t}:{\mathds{R}}^n\to {\mathds{C}}$ is defined by $e_t(x)=e^{it\cdot x}$, $x\in{\mathds{R}}^n$.
\end{definition}

\begin{remark}
For all $f \in L^1({\mathds{R}}^n)$ and all $\xi\in{\mathds{R}}^n$, the function $fe_{-\xi}$ belongs to $L^1({\mathds{R}}^n)$ since
$|fe_{-\xi}|=|f|$.
\end{remark}

\begin{lemma} \label{fourier-conv}
For $f \in L^1({\mathds{R}}^n)$ and $\xi\in{\mathds{R}}^n$, we have that 
\[
\hat{f} (\xi) = (f * e_\xi)(0).
\]
\end{lemma}

\begin{proof}
Just write it (being careful that all integrals exist!).
\end{proof}

For $a \in {\mathds{R}}^n$, we denote by $\tau_a$ the {\tt translation} by $a$: $\tau_a (f) (y) = f(y-a)$ for (almost) all $y \in {\mathds{R}}^n$.
Recall that for all $1\le p<\infty$ and all $f\in L^p({\mathds{R}}^n)$, 
\begin{equation}
\label{eq:translationLp}
\|\tau_a(f)-f\|_p\xrightarrow[|a|\to0]{}0.
\end{equation}

\begin{exercise}
Prove \eqref{eq:translationLp}. {\sl Hint:} use the fact that compactly supported continuous functions are dense in $L^p$.
\end{exercise}

A $n$ dimensional multi-index is a $n$-tuple $\alpha = (\alpha_1, ... , \alpha_n) \in {\mathds{N}}^n$. Its length is denoted by
$|\alpha|$ and is equal to $\alpha_1 + ... + \alpha_n$. The $\alpha$th derivative is: 
$D^\alpha = \partial_{x_1}^{\alpha_1} ...  \partial_{x_n}^{\alpha_n}$. To simplify future notations,
\[
D_\alpha = i^{-|\alpha|} D^\alpha = \left(\frac{1}{i^{\alpha_1}}\partial_{x_1}^{\alpha_1}\right) ... \left(\frac{1}{i^{\alpha_n}}\partial_{x_n}^{\alpha_n}\right).
\]

\begin{exercise}
Show that for all $\xi\in{\mathds{R}}^n$, we have that
 $D_\alpha e_\xi = \xi^\alpha e_\xi$, where $\xi^\alpha = \xi_1^{\alpha_1} ... \xi_n^{\alpha_n}$.
\end{exercise}

\begin{theorem} 
\label{thm:Fourier1}
Let $f,g \in L^1({\mathds{R}}^n)$, $x\in {\mathds{R}}^n$ and $\lambda >0$. The following identities hold
\begin{enumerate}
\item
${\mathscr{F}} (\tau_x f) = e_{-x} {\mathscr{F}} f$;
\item
${\mathscr{F}}(e_x f) = \tau_x ({\mathscr{F}} f)$; 
\item
${\mathscr{F}}(f*g) = ({\mathscr{F}} f) ({\mathscr{F}} g)$;
\item
denoting $f_\lambda(x) = f\bigl(\frac{1}{\lambda}x\bigr)$, $x \in {\mathds{R}}^n$, we have that 
${\mathscr{F}}( f_\lambda) (\xi) = \lambda^n ({\mathscr{F}} f)(\lambda \xi)$ for all $\xi \in {\mathds{R}}^n$.
\end{enumerate} 
\end{theorem}

\begin{proof}
Comes directly from the definition of the Fourier transform.
\end{proof}

Let ${\mathscr{C}}_0({\mathds{R}}^n)$ the space of bounded continuous functions on ${\mathds{R}}^n$ (with values in ${\mathds{C}}$)
with 0 limit at infinity.

\begin{exercise}
Show that, endowed with the $L^\infty$ norm, ${\mathscr{C}}_0({\mathds{R}}^n)$ is a Banach space.
\end{exercise}

\begin{theorem}[Riemann-Lebesgue Lemma]
\label{thm:RL}
The Fourier transform ${\mathscr{F}}:f \mapsto \hat{f}$ is linear, continuous from $L^1({\mathds{R}}^n)$ to ${\mathscr{C}}_0({\mathds{R}}^n)$.
\end{theorem}

\begin{proof}
Let $f\in L^1({\mathds{R}}^n)$. Since $e_{-x}:\xi\mapsto e^{-ix\cdot\xi}$ is continuous on ${\mathds{R}}^n$ 
for all $x\in{\mathds{R}}^n$ and
$|fe_{-x}|\le |f| \in L^1({\mathds{R}}^n)$, the dominated convergence theorem ensures that 
$\xi\mapsto \hat{f}(\xi)=\int_{{\mathds{R}}^n}f e_{-\xi}\,{\rm dm}_n$ is continuous. Linearity of ${\mathds{F}}$ 
is immediate from the definition.

\noindent
Now, to prove that $\hat{f}(\xi)\xrightarrow[|\xi|\to\infty]{}0$, let $\xi\in {\mathds{R}}^n$, $\xi\neq 0$ and write 
$-1=e^{-i\pi}=e^{-i\xi\cdot\frac{\pi\xi}{|\xi|^2}}$. We have then
\begin{align*}
\hat{f}(\xi)&= \int_{{\mathds{R}}^n} f e_{-i\xi}\,{\rm dm}_n\\
&=\int_{{\mathds{R}}^n} f(x) e^{-ix\cdot\xi}\,{\rm dm}_n(x)\\
&=-\int_{{\mathds{R}}^n} f(x) e^{-ix\cdot\xi}e^{i\xi\cdot\frac{\pi\xi}{|\xi|^2}}\,{\rm dm}_n(x)\\
&=-\int_{{\mathds{R}}^n} f(x) e^{-i(x+\frac{\pi\xi}{|\xi|^2})\cdot\xi}\,{\rm dm}_n(x)\\
&=-\int_{{\mathds{R}}^n} f(y-\tfrac{\pi\xi}{|\xi|^2}) e^{-iy\cdot\xi}\,{\rm dm}_n(y), \quad 
\mbox{\footnotesize with the change of variable $y=x+\tfrac{\pi\xi}{|\xi|^2}$}\\
&=-\int_{{\mathds{R}}^n} \tau_{\tfrac{\pi\xi}{|\xi|^2}}(f) e_{-\xi}\,{\rm dm}_n, \quad 
\mbox{\footnotesize where $\tau_a$ is the translation defined above}.
\end{align*}
Therefore, writing $\hat{f}(\xi)$ as the mean value of the first and the last lines:
\[
\hat{f}(\xi)=\frac{1}{2}\int_{{\mathds{R}}^n} \bigl(f-\tau_{\tfrac{\pi\xi}{|\xi|^2}}(f)\bigr)e_{-\xi}\,{\rm dm}_n,
\]
we obtain
\[
|\hat{f}(\xi)|\le \frac{1}{2}\bigl\|f-\tau_{\tfrac{\pi\xi}{|\xi|^2}}(f)\bigr\|_1\xrightarrow[|\xi|\to \infty]{}0
\]
thanks to \eqref{eq:translationLp}
\end{proof}

\subsection{The Schwartz space}

We denote by ${\mathscr{S}}({\mathds{R}}^n)$ the space of {\tt rapidly decreasing functions} on ${\mathds{R}}^n$,
also known as the {\tt Schwartz space}. It consists in smooth functions that decrease, as well as all their derivatives, faster than any polynomials.
More precisely:
\[
{\mathscr{S}}({\mathds{R}}^n):=\bigl\{f\in{\mathscr{C}}^\infty({\mathds{R}}^n); \forall\,\alpha\in{\mathds{N}}^n,\beta\in{\mathds{N}}^n:
\sup_{x\in{\mathds{R}}^n}|x^\alpha D_\beta f(x)|<\infty\bigr\}.
\]
Recall that for a multi-index $\alpha\in{\mathds{N}}^n$ and $x\in{\mathds{R}}^n$, $x^\alpha$ denotes $x_1^{\alpha_1}...x_n^{\alpha_n}$
and for a multi-index $\beta\in{\mathds{N}}^n$ and $f\in {\mathscr{C}}^\infty({\mathds{R}}^n)$, $D_\beta f$ denotes
$i^{-|\beta|}\partial_{x_1}^{\beta_1}...\partial_{x_n}^{\beta_n}f$. In other words, up to the factor $i^{-|\beta|}$ ($|\beta|=\beta_1+...+\beta_n$),
$D_\beta f$ is the $\beta_1$ derivative with respect to $x_1$, $\beta_2$ derivative with respect to $x_2$,..., $\beta_n$ derivative
with respect to $x_n$ of the function $f$.

\begin{example}
A prototype of a Schwartz function is the gaussian ${\mathds{R}}^n\ni x\mapsto e^{-|x|^2}$. It is also clear that 
${\mathscr{C}}^\infty_c({\mathds{R}}^n) \subset {\mathscr{S}}({\mathds{R}}^n)$.
\end{example}

On ${\mathscr{S}}({\mathds{R}}^n)$, we define the family of seminorms
\[
p_{\alpha,\beta}(f)=\sup_{x\in{\mathds{R}}^n}|x^\alpha D_\beta f(x)|, \quad \alpha,\beta\in{\mathds{R}}^n.
\]

\begin{exercise}
\begin{enumerate}
\item
Show that for all $\alpha,\beta\in{\mathds{N}}^n$, $p_{\alpha,\beta}$ are seminorms on ${\mathscr{S}}({\mathds{R}}^n)$
(verify all properties of a norm up to separation).
\item
Show that ${\mathscr{S}}({\mathds{R}}^n)$ with the family of seminorms $\{p_{\alpha,\beta}, \alpha,\beta\in{\mathds{N}}^n\}$
is a locally convex topological vector space.
\item
Show that if $|\alpha|,|\beta|\le N$, the following inequality is true:
\[
p_{\alpha,\beta}(f)\le \sup_{x\in{\mathds{R}}^n, |\gamma|\le N}(1+|x|)^N |D_\gamma f(x)| =: p_N(f), \quad f\in{\mathscr{S}}({\mathds{R}}^n).
\]
\item
Show that $d:{\mathscr{S}}({\mathds{R}}^n)\times {\mathscr{S}}({\mathds{R}}^n)\to [0,\infty)$ defined by
\[
d(f,g)=\sum_{N=0}^\infty 2^{-(N+1)}\frac{p_N(f-g)}{1+p_N(f-g)}
\]
is a distance which makes ${\mathscr{S}}({\mathds{R}}^n)$ a complete metric space.
\end{enumerate}
\end{exercise}

\begin{theorem}
The Schwartz space is not normable.
\end{theorem}

\begin{proof}
First, it is clear that the distance $d$ does not come from a norm. If it were the case, then the norm would be given by $\|f\|=d(f,0)$, 
$f\in {\mathscr{S}}({\mathds{R}}^n)$, but this doesn't satisfy $\|\lambda f\|=|\lambda|\|f\|$ if $\lambda \in{\mathds{R}}\setminus\{-1,0,1\}$.

\noindent
Now, a locally convex metric vector space is normable if (and only if) there is a bounded neighbourhood of 0. We will show that 
all neighbourhoods of 0 are unbounded. Bounded here means: bounded for every seminorm $p_{\alpha,\beta}$. The topology of
${\mathscr{S}}({\mathds{R}}^n)$ is given by the family of seminorms $\{p_{\alpha,\beta}, \alpha,\beta\in{\mathds{N}}^n\}$: it means that
a {\tt basis of neighbourhoods} of 0 is given by 
\[
U_{N,\varepsilon}:=\bigl\{f\in {\mathscr{S}}({\mathds{R}}^n); p_{\alpha,\beta}(f)<\varepsilon,\ \forall\ |\alpha|,|\beta|\le N\bigr\}.
\]
Choose your favourite non zero function in the Schwartz space, for instance $\varphi(x)=e^{-\tfrac{1}{2}\,|x|^2}$,
$x\in{\mathds{R}}^n$. For $k\ge 1$ and $M\in{\mathds{N}}$
to be chosen later, denote by $f_k$ the function defined by 
\[
f_k(x)=k^{-M}\varphi\bigl(\tfrac{x}{k}\bigr), \quad x\in{\mathds{R}}^n.
\]
It is easy to calculate (exercise!) $p_{\alpha,\beta}(f_k)$:
\[
p_{\alpha,\beta}(f_k)=k^{-M+|\alpha|-|\beta|}p_{\alpha,\beta}(\varphi).
\]
If $|\alpha|,|\beta|\le N$ and we choose $M>N$, we have that 
\[
p_{\alpha,\beta}(f_k)\le k^{-M+N}p_N(\varphi)\xrightarrow[k\to\infty]{}0.
\]
Then for $k$ large enough, $f_k\in U_{N,\varepsilon}$. But for $\beta=(0,...,0)$ and $\alpha\in{\mathds{N}}^n$
with $|\alpha|=M+1$, we have $p_{\alpha,\beta}(f_k)=k p_{\alpha,\beta}(\varphi) \xrightarrow[k\to\infty]{}+\infty$. This means that 
$U_{N,\varepsilon}$ is not bounded. Since $\{U_{N,\varepsilon}, N\in{\mathds{N}}, \varepsilon >0\}$ is a basis of neighbourhoods of 0,
this proves that all neighbourhoods of 0 are unbounded.
\end{proof}

\begin{exercise}
\label{ex:gaussian}
Let $\varphi$ be the function used in the proof above: $\varphi(x)=e^{-\tfrac{1}{2}\,|x|^2}$, $x\in {\mathds{R}}^n$. Prove that
${\mathscr{F}}(\varphi)=\varphi$. Recall that $\int_0^\infty e^{-t^2}\,{\rm d}t =\frac{\sqrt{\pi}}{2}$.
\end{exercise}

Next, we will see that the Schwartz space is very well tailored for the Fourier transform.

\begin{theorem}
The Fourier transform maps ${\mathscr{S}}({\mathds{R}}^n)$ to ${\mathscr{S}}({\mathds{R}}^n)$ continuously.
\end{theorem}

\begin{proof}
For all $f\in{\mathscr{S}}({\mathds{R}}^n)$, all $\alpha \in {\mathds{N}}^n$, we have that
$D_\alpha f\in {\mathscr{S}}({\mathds{R}}^n)\subset L^1({\mathds{R}}^n)$
and $x \mapsto x^\alpha f(x) \in {\mathscr{S}}({\mathds{R}}^n)\subset L^1({\mathds{R}}^n)$, so that it makes sense to 
compute their Fourier transform. Integrations by parts and the dominated convergence theorem give for all $\xi\in{\mathds{R}}^n$
\begin{equation}
\label{eq:diff-mult}
{\mathscr{F}}(D_\alpha f) (\xi) = \xi^\alpha ({\mathscr{F}} f )(\xi) \quad \mbox{ and } \quad 
{\mathscr{F}}\bigl(x \mapsto (-x)^\alpha f(x)\bigr)(\xi) = D_\alpha ({\mathscr{F}} f)(\xi).
\end{equation}
In other words, ${\mathscr{F}}$ exchanges derivatives and multiplication with a polynomial. With that in hands, the claim
that ${\mathscr{F}}$ maps ${\mathscr{S}}({\mathds{R}}^n)$ to ${\mathscr{S}}({\mathds{R}}^n)$ is straightforward (exercise!). 
For the continuity of the map, see the exercise below.
\end{proof}

\begin{exercise}
Show that for all $N\in{\mathds{N}}$ and all $\alpha,\beta\in{\mathds{N}}^n$ with $|\alpha|,|\beta|\le N$, we have that
\begin{enumerate}
\item
$\bigl\|x\mapsto (-x)^\beta D_\alpha f(x)\bigr\|_1\le C_n \,p_{N+n+1}(f)$ where $C_n$ is a constant depending only on the dimension $n$;
\item
and then for all $\xi\in {\mathds{R}}^n$, $\bigl|\xi^\alpha D_\beta ({\mathscr{F}}(f))(\xi)\bigr|\le C_n\, p_{N+n+1}(f)$.
\item
Conclude that ${\mathscr{F}}:{\mathscr{S}}({\mathds{R}}^n)\to {\mathscr{S}}({\mathds{R}}^n)$ is continuous.
\end{enumerate}
\end{exercise}

\noindent
The scalar product in $L^2({\mathds{R}}^n)$ is denoted by $\langle \cdot , \cdot \rangle$ ; it is defined by
\[
\langle f,g\rangle = \int_{{\mathds{R}}^n} f \overline{g} \ dm_n, \quad f,g \in L^2({\mathds{R}}^n).
\]
The notation $\overline{z}$ for $z \in {\mathds{C}}$ is used to denote the conjugate of $z$.

We are now in position to state (and prove) Plancherel (Parseval) theorem.

\begin{theorem}[Plancherel or Parseval-Plancherel identity]
The Fourier transform is an isometric isomorphism of ${\mathscr{S}}({\mathds{R}}^n)$, endowed with the induced topology of 
$L^2({\mathds{R}}^n)$.
\end{theorem}

\begin{proof}
Let $f,g\in{\mathscr{S}}({\mathds{R}}^n)$. By Fubini's theorem, it is immediate that
\begin{equation}
\label{eq:hatfg}
\int_{{\mathds{R}}^n}\hat{f} g\,{\rm dm}_n =\int_{{\mathds{R}}^n} f \hat{g}\,{\rm dm}_n.
\end{equation}
Let $\lambda>0$ and $\varphi$ be the function of Exercise~\ref{ex:gaussian}. We apply the previous identity with 
$f\in{\mathscr{S}}({\mathds{R}}^n)$ and $g(x)=\varphi\bigl(\tfrac{x}{\lambda}\bigr)$, $x\in {\mathds{R}}^n$. By Theorem~\ref{thm:Fourier1}
(point 4),
we have that
\begin{align*}
\int_{{\mathds{R}}^n} \hat{f}(x) \varphi\bigl(\tfrac{x}{\lambda}\bigr)\,{\rm dm}_n(x)
&=\int_{{\mathds{R}}^n} f(x) \lambda^n\varphi(\lambda x)\,{\rm dm}_n(x),\quad \mbox{\footnotesize since $\hat{\varphi}=\varphi$},\\
&=\int_{{\mathds{R}}^n} f\bigl(\tfrac{y}{\lambda}\bigr) \varphi(y)\,{\rm dm}_n(y), \quad \mbox{\footnotesize by the change of variables $y=\lambda x$}.
\end{align*}
The first term, as $\lambda$ goes to $\infty$, converges to $\int_{{\mathds{R}}^n} \hat{f}\,{\rm dm}_n$ and the last term
converges to $f(0)$. If we replace $f$ by $\tau_{-x}f$ for any $x\in {\mathds{R}}^n$, we obtain, by Theorem~\ref{thm:Fourier1} (point 1)
\[
f(x)=(\tau_{-x}f)(0)=\int_{{\mathds{R}}^n} \widehat{\tau_{-x}f}\,{\rm dm}_n = \int_{{\mathds{R}}^n} e^{ix\cdot\xi}\hat{f}(\xi)\,{\rm dm}_n(\xi).
\]
We proved that ${\mathscr{F}}$ is invertible and its inverse is given by
\begin{equation}
\label{eq:inverseF}
{\mathscr{F}}^{-1}(f)(x)= \int_{{\mathds{R}}^n} e^{ix\cdot\xi} f(\xi)\,{\rm dm}_n(\xi) = {\mathscr{F}}(f)(-x), \quad x\in{\mathds{R}}^n,
\end{equation}
which shows the first point of the theorem, the fact that the Fourier transform is an isomorphism of ${\mathscr{S}}({\mathds{R}}^n)$. It
remains to show that it is isometric for the $L^2$ norm or more precisely: $\langle {\mathscr{F}}(u),{\mathscr{F}}(v)\rangle =\langle u,v\rangle$
for all $u,v\in{\mathscr{S}}({\mathds{R}}^n)$.
We apply \eqref{eq:hatfg} with $f=u$ and $g=\overline{{\mathscr{F}}(v)}$ and we obtain
\[
\langle \hat{u} ,\hat{v} \rangle= \int_{{\mathds{R}}^n}\hat{f} g\,{\rm dm}_n=\int_{{\mathds{R}}^n} f \hat{g}\,{\rm dm}_n =\langle u,v\rangle
\]
since for all $\xi\in{\mathds{R}}^n$,
\[
\hat{g}(\xi)=\int_{{\mathds{R}}^n}e_{-\xi}\overline{{\mathscr{F}}(v)}\, dm_n
=\overline{\int_{{\mathds{R}}^n}e_{\xi}{\mathscr{F}}(v)\, dm_n}=\overline{{\mathscr{F}}^{-1}\bigl({\mathscr{F}}(v)\bigr)(\xi)}=\overline{v(\xi)}.
\]
which proves the isometry (with $u=v$, we have that $\|u\|_2=\|\hat{u}\|_2$).
\end{proof}

The space of {\tt tempered distributions} consists of linear functionals on ${\mathscr{S}}({\mathds{R}}^n)$ and is denoted by
${\mathscr{S}}'({\mathds{R}}^n)$. We can extend the Fourier transform to ${\mathscr{S}}'({\mathds{R}}^n)$ by the following relation:
for all $u\in {\mathscr{S}}'({\mathds{R}}^n)$,
\[
{\mathscr{F}}(u)(f)=u(\hat{f}), \quad f\in {\mathscr{S}}({\mathds{R}}^n). 
\]
The Fourier transform on ${\mathscr{S}}'({\mathds{R}}^n)$ is invertible and its inverse is given by 
\[
{\mathscr{F}}^{-1}(u)(f)=u({\mathscr{F}}^{-1}{f}), \quad f\in {\mathscr{S}}({\mathds{R}}^n). 
\]

\begin{exercise}
Let $u$ defined by $u(f)=f(0)$ for all $f\in{\mathscr{S}}({\mathds{R}}^n)$ (this $u$ is called the Dirac function at 0 and is denoted by $\delta_0$). 
\begin{enumerate}
\item
Show that $u\in {\mathscr{S}}'({\mathds{R}}^n)$ and calculate its Fourier transform.
\item
Let $n=1$.
Let $H$ be the Heaviside function: $H(t)=0$ if $t< 0$ and $H(t)=1$ if $t\ge 0$. Prove that $H'=\delta_0$ in ${\mathscr{S}}'({\mathds{R}})$ 
and determine ${\mathscr{F}}(H)$. 
\item
Let $a\in{\mathds{R}}^n$ and define $\delta_a(f)=f(a)$ for all $f\in{\mathscr{S}}({\mathds{R}}^n)$. What is ${\mathscr{F}}(\delta_a)$?
\end{enumerate}
\end{exercise}

\begin{remark}
Since ${\mathscr{D}}({\mathds{R}}^n):={\mathscr{C}}_c^\infty({\mathds{R}}^n)\subset {\mathscr{S}}({\mathds{R}}^n)$, tempered distributions
are also regular distributions: ${\mathscr{S}}'({\mathds{R}}^n) \subset {\mathscr{D}}'({\mathds{R}}^n)$. 
\end{remark}

\begin{proposition}
Let $u\in {\mathscr{S}}'({\mathds{R}}^n)$. 
\begin{enumerate}
\item
For all $k\in \{1,...,n\}$, $\partial_{x_k}u$ defined by $\partial_{x_k}(u)(f)=-u(\partial_{x_k}f)$ for all $f\in{\mathscr{S}}({\mathds{R}}^n)$
is a tempered distribution.
\item
For all $k\in \{1,...,n\}$, $x_k u$ defined by $(x_k u)(f)=u(x\mapsto x_kf(x))$ for all $f\in{\mathscr{S}}({\mathds{R}}^n)$
is a tempered distribution. Note that the notation $x_k u$ is slightly abusive.
\end{enumerate}
This implies that ${\mathscr{S}}'({\mathds{R}}^n)$ is invariant with respect of differentiations and multiplication by polynomials.
\end{proposition}

\begin{proof}
This is clear since ${\mathscr{S}}({\mathds{R}}^n)$ is invariant by taking derivatives or by multiplying by polynomials.
\end{proof}

\begin{exercise}
Prove \eqref{eq:diff-mult} (the action of the Fourier transform on derivatives or multiplication by polynomials)
in the case $f\in{\mathscr{S}}'({\mathds{R}}^n)$.
\end{exercise}

\subsection{The Fourier transform on $L^2$}

We have already seen that ${\mathscr{C}}_c^\infty({\mathds{R}}^n)\subset {\mathscr{S}}({\mathds{R}}^n)$ and we know that 
${\mathscr{C}}_c^\infty({\mathds{R}}^n)$ is dense in $L^2({\mathds{R}}^n)$ for the $L^2$ topology. Then ${\mathscr{S}}({\mathds{R}}^n)$
is also dense in $L^2({\mathds{R}}^n)$ for the $L^2$ topology. By Plancherel's theorem, it is immediate that we can
extend linearly and continuously the Fourier transform ${\mathscr{F}}$ to $L^2({\mathds{R}}^n)$; this extension is also an isometric
isomorphism. But be careful: the Fourier transform of an $L^2$ function which is not $L^1$ is not given by the integral \eqref{eq:defFourier}
(the integral is not convergent)!

For $f\in L^2({\mathds{R}}^n)$, we define its Fourier transform via the following procedure: let $(f_k)_{k\ge1}$ a sequence of 
functions in ${\mathscr{S}}({\mathds{R}}^n)$ approaching $f$ in $L^2$ in the sense $\|f_k-f\|_2\xrightarrow[k\to\infty]{}0$. 
Then $({\mathscr{F}}(f_k))_{k\ge 1}$ is a Cauchy sequence of functions in ${\mathscr{S}}({\mathds{R}}^n)$:
\[
\|{\mathscr{F}}(f_k)-{\mathscr{F}}(f_\ell)\|_2=\|f_k-f_\ell\|_2\xrightarrow[k,\ell\to\infty]{}0.
\] 
Since $L^2({\mathds{R}}^n)$ is complete, the sequence $({\mathscr{F}}(f_k))_{k\ge 1}$ converges to a function $g\in L^2({\mathds{R}}^n)$.
It remains to show that $g$ is independent of the choice of the approximating sequence $(f_k)_{k\ge1}$: assume that 
$(\tilde{f}_k)_{k\ge1}$ is another sequence in ${\mathscr{S}}({\mathds{R}}^n)$ approaching $f$ in the $L^2$ norm; we denote by
$h$ the limit of $({\mathscr{F}}(\tilde{f}_k))_{k\ge 1}$ in $L^2({\mathds{R}}^n)$. Then we have that
\begin{align*}
\|g-h\|_2&=\|g-{\mathscr{F}}(f_k)+{\mathscr{F}}(f_k)-{\mathscr{F}}(\tilde{f}_k)+{\mathscr{F}}(\tilde{f}_k)-h\|_2\\
&\le
\|g-{\mathscr{F}}(f_k)\|_2+\|{\mathscr{F}}(f_k)-{\mathscr{F}}(\tilde{f}_k)\|_2+\|{\mathscr{F}}(\tilde{f}_k)-h\|_2\\
&\le \|g-{\mathscr{F}}(f_k)\|_2+\|f_k-\tilde{f}_k\|_2+\|{\mathscr{F}}(\tilde{f}_k)-h\|_2, \quad 
\mbox{\footnotesize since ${\mathscr{F}}$ is isometric for the $L^2$ norm}\\
&\xrightarrow[k\to\infty]{}0 \quad \mbox{\footnotesize by the definition of $g$, $h$ and the fact that $(f_k)_{k\ge1}$ and $(\tilde{f}_k)_{k\ge1}$
both converge to $f$}.
\end{align*}
So $h=g$: the $L^2$ limit of the Fourier transforms of an approximating sequence consisting of Schwartz functions is independent 
of the choice of the sequence. Therefore, it is legitimate to define $g$ as the Fourier transform of $f$.

\begin{proposition}
The Fourier transform on $L^2({\mathds{R}}^n)$, still denoted by ${\mathscr{F}}$, is an isometric isomorphism.
\end{proposition}

\begin{proof}
Exercise!
\end{proof}

\begin{exercise}
\begin{enumerate}
\item
Find an example of an $L^2$ function on ${\mathds{R}}^n$ that doesn't belong to $L^1$ (choose your favourite $n\ge 1$).
\item
How would you calculate its Fourier transform?
\end{enumerate}
\end{exercise}

\subsection{Sobolev spaces}

\begin{definition}
Let $s\in{\mathds{R}}$.
We define the {\tt homogeneous Sobolev space} $\dot H^s({\mathds{R}}^n)$ as follows:
\[
\dot H^s({\mathds{R}}^n):=\bigl\{f\in{\mathscr{S}}'({\mathds{R}}^n); {\mathscr{F}}(f)\in L^1_{\rm loc}({\mathds{R}}^n)\mbox{ and }
\xi\mapsto |\xi|^s{\mathscr{F}}(f)(\xi) \in L^2({\mathds{R}}^n)\bigr\}.
\]
We endow $\dot H^s({\mathds{R}}^n)$ with the norm
\[
\|f\|_{\dot H^s}:=\Bigl(\int_{{\mathds{R}}^n} |\xi|^{2s}\bigl|{\mathscr{F}}(f)(\xi)\bigr|^2{\rm d}\xi\Bigr)^{\frac{1}{2}},
\quad f\in \dot H^s({\mathds{R}}^n).
\]
\end{definition}

\begin{remark}
If $s=0$, we have that $\dot H^s({\mathds{R}}^n)=\dot H^0({\mathds{R}}^n)=L^2({\mathds{R}}^n)$.
\end{remark}

\begin{exercise}
Prove that $\|\cdot\|_{\dot H^s}$ is actually a norm on $\dot H^s({\mathds{R}}^n)$. It comes from a scalar product: can you describe it?
\end{exercise}

\begin{proposition}
Let $s_1,s_2,s\in{\mathds{R}}$ with $s_1<s<s_2$ and let $\theta\in (0,1)$ such that $s=(1-\theta) s_1+\theta s_2$.
Then $\dot H^{s_1}\cap \dot H^{s_2}\subset \dot H^s$ and for all $f\in \dot H^{s_1}\cap \dot H^{s_2}$ we have that
\[
\|f\|_{\dot H^s}\le \|f\|_{\dot H^{s_1}}^{1-\theta}\ \|f\|_{\dot H^{s_2}}^\theta.
\]
\end{proposition}

\begin{proof}
This comes directly from the identity
\[
|\xi|^s|{\mathscr{F}}f(\xi)| = (|\xi|^{s_1}|{\mathscr{F}}f(\xi)|)^{1-\theta}  (|\xi|^{s_1}|{\mathscr{F}}f(\xi)|)^\theta, \quad \xi\in {\mathds{R}}^n
\]
and the H\"older inequality $\|uv\|_2\le \|u\|_p\|v\|_q$ where $1<p,q<\infty$ satisfy $\frac{1}{p}+\frac{1}{q}=\frac{1}{2}$, 
applied to $p=\frac{2}{1-\theta}$,
$u=(|\cdot|^{s_1}|{\mathscr{F}}f|)^{1-\theta}$, $q=\frac{2}{\theta}$ and $v=(|\cdot|^{s_2}|{\mathscr{F}}f|)^{\theta}$
\end{proof}

\begin{proposition}
The homogeneous Sobolev space $\dot H^s({\mathds{R}}^n)$ with the norm $\|\cdot\|_{\dot H^s}$ is complete if and only if $s<\frac{n}{2}$.
\end{proposition}

\begin{proof}
Assume that $s<\frac{n}{2}$. Let $(u_k)_{k\in{\mathds{N}}}$ be a Cauchy sequence in $\dot H^s({\mathds{R}}^n)$. Then 
$(\widehat{u_k})_{k\in{\mathds{N}}}$ is a Cauchy sequence in $L^2({\mathds{R}}^n, |\xi|^{2s}{\rm d}\xi)$, which is complete. 
This implies that there exists $f\in L^2({\mathds{R}}^n, |\xi|^{2s}{\rm d}\xi)$ such that $(\widehat{u_k})_{k\in{\mathds{N}}}$ converges to $f$. 
Since
\[
\int_{B_n(0,1)}|\xi|^{-2s}\,{\rm d}\xi = \int_{{\mathds{S}}^{n-1}}\int_0^1 r^{-2s}r^{n-1}\,{\rm d}r\,{\rm d}\omega =c_n<\infty
\]
we have that
\begin{align*}
\int_{B_n(0,1)}&|f(\xi)|\,{\rm d}\xi = \int_{B_n(0,1)}|\xi|^{-s}|\xi|^s|f(\xi)|\,{\rm d}\xi\\
&\le \Big(\int_{B_n(0,1)}|\xi|^{-2s}\,{\rm d}\xi\Big)^{\frac{1}{2}} \Big(\int_{B_n(0,1)}|\xi|^{2s}|f(\xi)|^2\,{\rm d}\xi\Big)^{\frac{1}{2}} 
\quad\mbox{\footnotesize by Cauchy-Schwarz inequality}\\
&\le c_n\ \Big(\int_{{\mathds{R}}^n}|\xi|^{2s}|f(\xi)|^2\,{\rm d}\xi\Big)^{\frac{1}{2}} \le c_n\,\|f\|_{\dot H^s}. 
\end{align*}
This proves that $f 1\hspace{-2pt}{\rm l}_{B_n(0,1)}\in L^1({\mathds{R}}^n)$ and therefore 
${\mathscr{F}}^{-1}(f 1\hspace{-2pt}{\rm l}_{B_n(0,1)})$ is a bounded continuous function by the Riemann-Lebesgue Lemma.
We also have that for all $\xi \in {\mathds{R}}^n$ with $|\xi|\ge 1$:
\begin{itemize}
\item
if $s\ge 0$, then $(1+|\xi|^2)^{\frac{s}{2}}\le 2^{\frac{s}{2}} |\xi|^s$,
\item
if $s<0$, then $\frac{1}{(1+|\xi|^2)^{-\frac{s}{2}}} \le \frac{1}{(|\xi|^2)^{-\frac{s}{2}}} =|\xi|^s$.
\end{itemize}
This shows that
$\xi\mapsto (1+|\xi|^2)^{\frac{s}{2}}\,f(\xi) 1\hspace{-2pt}{\rm l}_{B_n(0,1)^c}(\xi) \in L^2({\mathds{R}}^n)$, so that 
$f 1\hspace{-2pt}{\rm l}_{B_n(0,1)^c}$ belongs to ${\mathscr{S}}'({\mathds{R}}^n)$. Then 
$u:={\mathscr{F}}^{-1}f\in {\mathscr{S}}'({\mathds{R}}^n)$ satisfies $\hat{u}\in \dot H^s({\mathds{R}}^n)$ by the definition of $f$ and
\[
\|u_k-u\|_{\dot H^s}=\Bigl(\int_{{\mathds{R}}^n} |\xi|^{2s}|\widehat{u_k}(\xi)-\hat{u}(\xi)|^2\,{\rm d}\xi\Bigr)^{\frac{1}{2}}
=\Bigl(\int_{{\mathds{R}}^n} |\xi|^{2s}|\widehat{u_k}(\xi)-f(\xi)|^2\,{\rm d}\xi\Bigr)^{\frac{1}{2}}\xrightarrow[k\to\infty]{}0.
\]
Let now $s\ge \frac{n}{2}$. Consider $N:\dot H^s\ni u\mapsto \|\hat{u}\|_{L^1(B_n(0,1))}+\|u\|_{\dot H^s}$. We can prove that
$N$ is a norm on $\dot H^s({\mathds{R}}^n)$ and $(\dot H^s({\mathds{R}}^n), N)$ is a complete space (so a Banach space): exercise!
{\sl Hint:} proceed as in the proof in the case $s<\frac{n}{2}$.

If $(\dot H^s({\mathds{R}}^n), \|\cdot\|_{\dot H^s})$ were also complete, then there would exist a constant $c>0$ such that 
$N\le c \,\|\cdot\|_{\dot H^s}$ (Exercise: explain why!), or equivalently,  $\|\hat{\cdot}\|_{L^1(B_n(0,1))}\le c'\, \|\cdot\|_{\dot H^s}$
for a constant $c'$. 

\noindent
Let ${\mathcal{A}}:=B(0,\frac{1}{4})\setminus \overline{B}(0,\frac{1}{6})$ be an annulus so that 
$2{\mathcal{A}}=B(0,\frac{1}{2})\setminus \overline{B}(0,\frac{1}{3})$ does not intersect ${\mathcal{A}}$. Define
\[
u_N:={\mathscr{F}}^{-1}\Bigl(\sum_{k=1}^N \tfrac{1}{k}\,2^{k(s+\frac{n}{2})} {\mathds{1}}_{2^{-k}{\mathcal{A}}}\Bigr),
\quad N\ge 1.
\]
Then 
\[
\|\widehat{u_N}\|_{L^1(B_n(0,1)} = \sum_{k=1}^N \tfrac{1}{k}\,2^{k(s+\frac{n}{2})} \|{\mathds{1}}_{2^{-k}{\mathcal{A}}}\|_{L^1(B_n(0,1))}
=C\, \sum_{k=1}^N \tfrac{1}{k}\,2^{k(s-\frac{n}{2})}
\]
(the constant $C$ is clearly related to the volume of ${\mathcal{A}}$) and (calculate!)
\[
\|u_N\|_{\dot H^s} \le \tilde{C}\, \sum_{k=1}^N \frac{1}{k^2},
\]
which contradicts the existence of a constant $c'$ as above (explain why!). This last example is taken from \cite[Prop.~1.34]{BCD11}.
\end{proof}

\begin{proposition}[Sobolev embedding]
Let $0<s<\frac{n}{2}$.
The following embedding property (in dimension $n$) holds
\begin{equation}
\label{eq:sobolev-emb}
\dot H^s\hookrightarrow L^q,\quad \mbox{for }\tfrac{1}{q}=\tfrac{1}{2}-\tfrac{s}{n}.
\end{equation}
\end{proposition}

A very elegant proof of \eqref{eq:sobolev-emb}, due to Jean-Yves Chemin is reproduced here (\cite[Lemma~1.1]{Da20}). 
See also \cite[Thm~1.38]{BCD11}.

\begin{proof}
We prove that $\dot H^s\hookrightarrow L^q$ for $\frac{1}{q}=\frac{1}{2}-\frac{s}{n}$ in the case $0<s<\frac{n}{2}$.
Remark that necessarily, $q>2$. Let $f\in \dot H^s$. A simple homogeneity argument shows that we can
assume that $\|f\|_{\dot H^s}=1$ (explain!).
Let $A>0$ be fixed. We decompose $g={\mathscr{F}}(f)$ at the level $A$: $g=g_1+g_2$ where $g_1= {\mathds{1}}_{B(0,A)}{\mathscr{F}}(f)$ and 
$g_2= {\mathds{1}}_{B(0,A)^c}{\mathscr{F}}(f)$. 

It is clear that $g_1\in L^1$, since ${\mathscr{F}}(f)\in L^1_{\rm loc}$ and $g_1$ is the restriction on a ball of radius $A$.
The Fourier transform maps $L^1({\mathds{R}}^n)$ to ${\mathscr{C}}_0({\mathds{R}}^n)$ as well as its inverse, 
then $f_1:={\mathscr{F}}^{-1}(g_1)$ is bounded and
\begin{align*}
\|f_1\|_\infty\le \|g_1\|_1&= \int_{B(0,A)}|\xi|^{-s}|\xi|^s|{\mathscr{F}}(f)(\xi)|\,{\rm dm}_n(\xi)\\
&\le c \|f\|_{\dot H^s}\Bigl(\int_{B(0,A)}|\xi|^{-2s}\,{\rm dm}_n(\xi)\Bigr)^{1/2}
\le C A^{\frac{n}{2}-s},
\end{align*}
where the constant $C$ depends on $n$; the last but one inequality is due to Cauchy-Schwarz inequality. 

We now prove that $g_2$ belongs to $L^2({\mathds{R}}^n)$:
\[
\|g_2\|_2^2=\int_{{\mathds{R}}^n}\bigl|{\mathds{1}}_{B(0,A)^c}{\mathscr{F}}(f)\bigr|^2\,{\rm dm}_n
\le \int_{{\mathds{R}}^n} \bigl(\tfrac{|\xi|}{A}\bigr)^{2s}|{\mathscr{F}}(f)(\xi)|^2\,{\rm dm}_n(\xi)\le A^{-2s}\|f\|_{\dot H^s}^2.
\]
The Fourier transform being an isometry on $L^2({\mathds{R}}^n)$, $f_2={\mathscr{F}}^{-1}(g_2)$ belongs to $L^2({\mathds{R}}^n)$.

For $\lambda>0$, we choose $A_\lambda=\bigl(\frac{\lambda}{4C}\bigr)^{q/n}$:
$f=f_{1,\lambda}+f_{2,\lambda}$ corresponding to the decomposition at the level $A_\lambda$. By the choice of $A_\lambda$,
we have that $\|f_{1,\lambda}\|_\infty\le \frac{\lambda}{4}$ and then
\[
\bigl\{|f_{1,\lambda}|>\tfrac{\lambda}{2}\bigr\}=\emptyset.
\]
Now, since $\bigl\{|f|>\lambda\bigr\}\subset \bigl\{|f_{1,\lambda}|>\tfrac{\lambda}{2}\bigr\}\cup \bigl\{|f_{2,\lambda}|>\tfrac{\lambda}{2}\bigr\}$, 
it is immediate that
\[
\|f\|_q^q=q\int_0^\infty\lambda^{q-1}\bigl|\{|f_{2,\lambda}|>\tfrac{\lambda}{2}\}\bigr|\,{\rm d}\lambda,
\]
where we used the following layer cake formula:
\[
\|f\|_q^q=q\int_0^\infty\lambda^{q-1}\bigl|\{|f|>\lambda\}\bigr|\,{\rm d}\lambda.
\]
Writing $\bigl|\{|f_{2,\lambda}|>\tfrac{\lambda}{2}\}\bigr| =\int_{\{|f_{2,\lambda}|>\frac{\lambda}{2}\}}\,{\rm dm}_n(x)$ and using
Beinaym\'e-Tchebychev inequality, we find that
\[
\bigl|\{|f_{2,\lambda}|>\tfrac{\lambda}{2}\}\bigr| 
\le \int_{\{|f_{2,\lambda}|>\frac{\lambda}{2}\}}\Bigl(\tfrac{2|f_{2,\lambda}|}{\lambda}\Bigr)^2\,{\rm dm}_m(x)
\le \tfrac{4}{\lambda^2}\|f_{2,\lambda}\|_2^2.
\]
Since the Fourier transform is an isometry in $L^2$, this gives, using Fubini,
\begin{align*}
\|f\|_q^q&\lesssim \int_0^\infty \lambda^{q-3}\Bigl(\int_{|\xi|>A_\lambda}|{\mathscr{F}}(f)(\xi)|^2\,{\rm dm}_n(\xi) \Bigr)\,{\rm d}\lambda \\
&\lesssim \int_{{\mathds{R}}^n}\Bigl(\int_0^{4C|\xi|^{n/q}}\lambda^{q-3}\,{\rm d}\lambda\Bigr)|{\mathscr{F}}(f)(\xi)|^2\,{\rm dm}_n(\xi)
\lesssim \int_{{\mathds{R}}^n} |\xi|^{\frac{n(q-2)}{q}}|{\mathscr{F}}(f)(\xi)|^2\,{\rm dm}_n(\xi).
\end{align*}
Remark that $\frac{n(q-2)}{q}=2s$: the last term in the series of inequalities above is then controlled by $\|f\|_{\dot H^s}=1$. 
This completes the proof.
\end{proof}

\begin{definition}
\label{def:sobolev}
Let $s\in {\mathds{R}}$ and define the (non homogeneous) {\tt Sobolev space} $H^s({\mathds{R}}^n)$ by 
\[
H^s({\mathds{R}}^n):=\bigl\{u\in{\mathscr{S}}'({\mathds{R}}^n);\xi\mapsto (1+|\xi|^2)^{\frac{s}{2}}{\mathscr{F}}(u)(\xi)\in L^2({\mathds{R}}^n)\bigr\}
\]
endowed with the norm
\[
\|u\|_{H^s}:= \bigl\|(1+|\cdot|^2)^{\frac{s}{2}}{\mathscr{F}}(u)\|_2.
\]
\end{definition}

\begin{proposition}
For all $s_1,s_2\in{\mathds{R}}$, if $s_1\le s_2$, then $H^{s_2}({\mathds{R}}^n)\subset H^{s_1}({\mathds{R}}^n)$.
\end{proposition}

\begin{proof}
This follows directly from the fact that $(1+|\xi|^2)^{\frac{s_1-s_2}{2}}\le 1$ for all $\xi \in {\mathds{R}}^n$ (since $s_1-s_2\le 0$).
\end{proof}

\begin{exercise}
If $s\ge 0$, prove that $H^s({\mathds{R}}^n)=\dot H^s({\mathds{R}}^n)\cap L^2({\mathds{R}}^n)$.
\end{exercise}

\begin{proposition}
For all $s\in{\mathds{R}}$, $H^s({\mathds{R}}^n)$ is a Hilbert space.
\end{proposition}

\begin{proof}
It suffices to show that $H^s({\mathds{R}}^n)$ is complete, since
\[
(f,g)\mapsto \int_{{\mathds{R}}^n} (1+|\xi|^2)^s \hat f(\xi)\overline{\hat g(\xi)}\,{\rm dm}_n(\xi), \quad f,g\in H^s({\mathds{R}}^n)
\]
is a scalar product on $H^s({\mathds{R}}^n)$ (exercise: prove it!). Now, let $(f_k)_{k\in{\mathds{N}}}$ be a Cauchy sequence in
$H^s({\mathds{R}}^n)$: by definition, $\bigl(\xi \mapsto (1+|\xi|^2)^{\frac{s}{2}}\hat f_k(\xi)\bigr)_{k\in{\mathds{N}}}$ is a Cauchy 
sequence in $L^2({\mathds{R}}^n)$ which is a Hilbert space. Then it has a limit in $L^2({\mathds{R}}^n)$: let's call it $g$. It remains to 
show that $f:={\mathscr{F}}^{-1}\bigl(\xi\mapsto (1+|\xi|^2)^{-\frac{s}{2}}g(\xi)\bigr)$ (exists in ${\mathscr{S}}'({\mathds{R}}^n)$) 
is the limit of $(f_k)_{k\in{\mathds{N}}}$ in $H^s({\mathds{R}}^n)$. We have that
\begin{align*}
\|f_k-f\|_{H^s}^2&=\int_{{\mathds{R}}^n}(1+|\xi|^2)^s|\hat f_k(\xi)-\hat f(\xi)|^2\,{\rm dm}_n(\xi)\\
&= \bigl\|\xi\mapsto \bigl((1+|\xi|^2)^{\frac{s}{2}}\hat f_k(\xi)-g(\xi)\bigr)\bigr\|_2^2\xrightarrow[k\to\infty]{}0
\end{align*}
by the definition of $g$.
\end{proof}

The following embedding properties hold:

\begin{theorem}[Sobolev embedding]
\label{thm:sobolev}
\begin{itemize}
\item[$\cdot$]
Let $0<s<\frac{n}{2}$. Then $H^s({\mathds{R}}^n)\hookrightarrow L^q({\mathds{R}}^n)$ for $\tfrac{1}{q}=\tfrac{1}{2}-\tfrac{s}{n}$.
\item[$\cdot$]
Let $s>\frac{n}{2}$. Then $H^s({\mathds{R}}^n)\subset {\mathscr{C}}_0({\mathds{R}}^n)$. In particular, functions in $H^s({\mathds{R}}^n)$
are bounded (so belong to $L^\infty({\mathds{R}}^n)$).
\end{itemize}
\end{theorem}

\begin{proof}
The first point, when $0<s<\frac{n}{2}$ is a direct consequence of \eqref{eq:sobolev-emb}. To prove the case $s>\frac{n}{2}$, it is sufficient
to prove that for all $f\in H^s({\mathds{R}}^n)$, we have that $\hat{f}\in L^1({\mathds{R}}^n)$ (why?): exercise!
\end{proof}

\begin{exercise}
Let $n\ge 1$ (you can take $n=1$ if it's more comfortable) and $f:{\mathds{R}}\times {\mathds{R}}^n\to {\mathds{R}}$ be 
a function in $L^2({\mathds{R}}\times{\mathds{R}}^n)$ satisfying
$t\mapsto f(t,\cdot) \in L^2({\mathds{R}};H^n({\mathds{R}}^n))\cap H^1({\mathds{R}};L^2({\mathds{R}}^n))$. This means that for almost all
$t\in{\mathds{R}}$, the function $x\mapsto f(t,x)$ belongs to $H^n({\mathds{R}}^n)$ with $t\mapsto \|f(t,\cdot)\|_{H^n} \in L^2({\mathds{R}})$ and 
for almost all $t\in{\mathds{R}}$, the function $x\mapsto \partial_tf(t,x)$ belongs to $L^2({\mathds{R}}^n)$ with
$(t,x)\mapsto \partial_tf(t,x) \in L^2({\mathds{R}}\times{\mathds{R}}^n)$.
\begin{enumerate}
\item
Show that $(\tau,\xi)\mapsto (1+|\tau|+|\xi|^n)\hat{f}(\tau,\xi)$ belongs to $L^2({\mathds{R}}\times{\mathds{R}}^n)$.
\item
Show that  
$(\xi,\tau)\mapsto (1+|\xi|^{\frac{n}{2}})\hat{f}(\tau,\xi)$ belongs to $L^2({\mathds{R}}^n;L^1({\mathds{R}}))$.
\item
Show that 
$(\tau,\xi)\mapsto (1+|\tau|^{\frac{1}{2}})\hat{f}(\tau,\xi)$ belongs to $L^2({\mathds{R}};L^1({\mathds{R}}^n))$.
\item
Prove that $f$ belongs to ${\mathscr{C}}_0({\mathds{R}};H^{\frac{n}{2}}({\mathds{R}}^n))
\cap H^{\frac{1}{2}}({\mathds{R}};{\mathscr{C}}_0({\mathds{R}}^n))$.
\item
What can be said if we replace the non homogeneous Sobolev spaces $H^s$ with the homogeneous ones $\dot H^s$?
\end{enumerate}
\end{exercise}

\begin{proof}
We denote ${\mathscr{F}}_x$ the Fourier transform with respect to the $x$ variable and ${\mathscr{F}}_t$ the Fourier transform with 
respect to the $t$ variable.
\begin{enumerate}
\item
We know that $t\mapsto \|f(t,\cdot)\|_{H^n} \in L^2({\mathds{R}})$, which means that
\[
(t,\xi)\mapsto (1+|\xi|^2)^{\frac{n}{2}}{\mathscr{F}}_x(f(t,\cdot))(\xi) \in L^2({\mathds{R}}\times {\mathds{R}}^n),
\]
or equivalently (taking the Fourier transform in the $t$ variable)
\begin{equation}
\label{eq:L2Hn}
(\tau,\xi)\mapsto (1+|\xi|^n){\mathscr{F}}_t\bigl({\mathscr{F}}_x(f)\bigr)(\tau,\xi) = (1+|\xi|^n)\hat{f}(\tau,\xi) \in L^2({\mathds{R}}\times {\mathds{R}}^n),
\end{equation}
since $(1+|\xi|^n)$ and $(1+|\xi|^2)^{\frac{n}{2}}$ are comparable ($\xi\in{\mathds{R}}^n$). This gives the first half of the answer.
The assumption that $(t,x)\mapsto \partial_tf(t,x) \in L^2({\mathds{R}}\times{\mathds{R}}^n)$ implies that, taking the Fourier transform
in the $t$ variable:
\[
(\tau,x)\mapsto i\tau {\mathscr{F}}_t(f)(\tau,x) \in L^2({\mathds{R}}\times{\mathds{R}}^n).
\]
Taking the Fourier transform in the $x$ variable, this gives
\begin{equation}
\label{eq:H1L2}
(\tau, \xi)\mapsto |\tau|{\mathscr{F}}_x({\mathscr{F}}_t(f)(\tau,\xi)=|\tau|\hat{f}(\tau,\xi) \in L^2({\mathds{R}}\times{\mathds{R}}^n).
\end{equation}
Adding \eqref{eq:L2Hn} and \eqref{eq:H1L2}, the claim is proved.
\item
We have that
\[
(1+|\xi|^{\frac{n}{2}})\hat{f}(\tau,\xi)=
\tfrac{1+|\xi|^{\frac{n}{2}}}{1+|\tau|+|\xi|^n}\, (1+|\tau|+|\xi|^n)\hat{f}(\tau,\xi).
\]
Since $(\xi,\tau)\mapsto(1+|\tau|+|\xi|^n)\hat{f}(\tau,\xi) \in L^2({\mathds{R}}^n\times{\mathds{R}})$, it is sufficient to prove that the function
$(\xi,\tau)\mapsto \frac{1+|\xi|^{\frac{n}{2}}}{1+|\tau|+|\xi|^n}$ belongs to $L^\infty({\mathds{R}}^n; L^2({\mathds{R}}))$.
We have for all $T>0$ and all $M>0$
\begin{align*}
\int_{-T}^T\bigl|\tfrac{1+|\xi|^{\frac{n}{2}}}{1+|\tau|+|\xi|^n}\bigr|^2\,{\rm dm}_1(\tau)
\le& \sqrt{2} \int_{-T}^T\bigl|\tfrac{(1+|\xi|^n)^{\frac{1}{2}}}{1+|\tau|+|\xi|^n}\bigr|^2\,{\rm dm}_1(\tau)\\
=& \sqrt{2} \int_{-T}^T\frac{1}{\bigl(1+\frac{|\tau|}{1+|\xi|^n}\bigr)^2}\,\tfrac{1}{1+|\xi|^n}\, 
{\rm dm}_1(\tau)\\
=& \sqrt{2} \int_{-\frac{T}{1+|\xi|^n}}^{\frac{T}{1+|\xi|^n}}\tfrac{1}{(1+|\sigma|)^2}\,
{\rm dm}_1(\sigma)
\le  \frac{2}{\sqrt{\pi}},
\end{align*}
where the first inequality comes from the fact that $\frac{a+b}{\sqrt{a^2+b^2}}\le \sqrt{2}$ for all $a,b>0$ and we made the change 
of variable $\sigma=\frac{\tau}{1+|\xi|^n}$ in the last integral. The last inequality is obtained 
by taking $T\to +\infty$.
\item
This point is very similar to the previous one. We have that 
\[
(1+|\tau|^{\frac{1}{2}})\hat{f}(\tau,\xi)=
\tfrac{1+|\tau|^{\frac{1}{2}}}{1+|\tau|+|\xi|^n}\,(1+|\tau|+|\xi|^n)\hat{f}(\tau,\xi).
\]
Since $(\tau,\xi)\mapsto(1+|\tau|+|\xi|^n)\hat{f}(\tau,\xi) \in L^2({\mathds{R}}\times{\mathds{R}}^n)$, it is sufficient to prove that the function
$(\tau,\xi)\mapsto \frac{1+|\tau|^{\frac{1}{2}}}{1+|\tau|+|\xi|^n}$ belongs to $L^\infty({\mathds{R}};L^2({\mathds{R}}^n))$.
We have for all $M>0$ 
\begin{align*}
\int_{|\xi|\le M} \bigl|\tfrac{1+|\tau|^{\frac{1}{2}}}{1+|\tau|+|\xi|^n}\,\bigr|^2\,{\rm dm}_n(\xi)
\le& \sqrt{2} \int_{|\xi|\le M}\bigl|\tfrac{(1+|\tau|)^{\frac{1}{2}}}{1+|\tau|+|\xi|^n}\bigr|^2\,{\rm dm}_n(\xi)\\
=& \sqrt{2} \int_{|\xi|\le M}\frac{1}{\bigl(1+\frac{|\xi|^n}{1+|\tau|}\bigr)^2}\,\tfrac{1}{1+|\tau|}\, 
{\rm dm}_n(\xi)\\
=& \sqrt{2} \Bigl(\int_{|\xi|\le \frac{M}{(1+|\tau|)^{\frac{1}{n}}}} \tfrac{1}{(1+|\eta|^n)^2}\,
{\rm dm}_n(\eta)\Bigr)
\lesssim 1,
\end{align*}
where we made the change of variable $\eta =\frac{\xi}{(1+|\tau|)^{\frac{1}{n}}}$. The constant appearing in the last inequality
involves the quantity $\displaystyle{\int_{{\mathds{S}}^{n-1}}\Bigl(\int_0^\infty \tfrac{r^{n-1}}{(1+r^n)^2}\, {\rm d}r\Bigr)\,{\rm d}\omega}$
where we use the polar coordinates $\eta = r\omega\in {\mathds{R}}^n$, $r>0$, $\omega\in{\mathds{S}}^{n-1}$.
\item
From point 3, we have that $(\tau,\xi)\mapsto (1+|\tau|^{\frac{1}{2}})\hat{f}(\tau,\xi)$ belongs to $L^2({\mathds{R}};L^1({\mathds{R}}^n))$. 
Taking the inverse Fourier transform in the $\xi$ variable, and then in the $\tau$ variable, since ${\mathscr{F}}^{-1}$ maps
$L^1$ to ${\mathscr{C}}_0$, and thanks to the definition of the fractional Sobolev space $H^{\frac{1}{2}}$, 
we obtain that $f$ belongs to $H^{\frac{1}{2}}({\mathds{R}};{\mathscr{C}}_0({\mathds{R}}^n))$.

\noindent
Starting now from point 2, we show similarly, taking the inverse Fourier transform in the $\xi$ and the $\tau$ variables, 
that $f$ belongs to $H^{\frac{n}{2}}({\mathds{R}}^n;{\mathscr{C}}_0({\mathds{R}}))
\subset {\mathscr{C}}_0({\mathds{R}};H^{\frac{n}{2}}({\mathds{R}}^n))$.
\item
We can replace non homogeneous Sobolev spaces by the homogeneous ones. The differences will be:
\begin{itemize}
\item
in the first question: $(\tau,\xi)\mapsto (|\tau|+|\xi|^n)\hat{f}(\tau,\xi)$ belongs to $L^2({\mathds{R}}\times{\mathds{R}}^n)$ instead of
$(\tau,\xi)\mapsto (1+|\tau|+|\xi|^n)\hat{f}(\tau,\xi)$ belongs to $L^2({\mathds{R}}\times{\mathds{R}}^n)$;
\item
in the second question: $(\xi,\tau)\mapsto \frac{|\xi|^{\frac{n}{2}}}{|\tau|+|\xi|^n}$ belongs to $L^\infty({\mathds{R}}^n; L^2({\mathds{R}}))$
instead of $(\xi,\tau)\mapsto \frac{1+|\xi|^{\frac{n}{2}}}{1+|\tau|+|\xi|^n}$ belongs to $L^\infty({\mathds{R}}^n; L^2({\mathds{R}}))$;
\item
in the third question: $(\tau,\xi)\mapsto \frac{|\tau|^{\frac{1}{2}}}{|\tau|+|\xi|^n}$ belongs to $L^\infty({\mathds{R}};L^2({\mathds{R}}^n))$
instead of $(\tau,\xi)\mapsto \frac{1+|\tau|^{\frac{1}{2}}}{1+|\tau|+|\xi|^n}$ belongs to $L^\infty({\mathds{R}};L^2({\mathds{R}}^n))$.
\end{itemize}
\end{enumerate}
There is a nice counter-example due to Llewelyn Andrews to the fact that 
\[
(\tau,\xi)\mapsto  (1+|\xi|^{\frac{n}{2}})\hat{f}(\tau,\xi)
\] 
does not necessarily belong to $L^1({\mathds{R}};L^2({\mathds{R}}^n))$. To avoid any technical details,
let's work with homogeneous Sobolev spaces.

\noindent
Let $\varphi \in L^2({\mathds{R}})$ with support in $[1,\infty)$ such that $\int_1^\infty\tau^3|\varphi(\tau)|^2\,{\rm d}\tau <\infty$
and $\psi\in L^2({\mathds{R}}^n)$ such that $\int_{{\mathds{R}}^n}(1+|\eta|^{2n})|\psi(\eta)|^2\,{\rm d}\eta<\infty$.
Define $f$ via its Fourier transform: 
\[
\hat{f}(\tau,\xi)=\varphi(\tau)\psi(|\xi|\tau^{-\frac{1}{n}}), \quad \tau\in{\mathds{R}}, \xi\in {\mathds{R}}^n.
\]
On the one hand, $f\in L^2({\mathds{R}};\dot H^n({\mathds{R}}^n))$. Indeed:
\begin{align*}
\|f\|_{L^2({\mathds{R}};\dot H^n({\mathds{R}}^n))}^2&=
\int_1^\infty \int_{{\mathds{R}}^n} |\xi|^{2n}|\varphi(\tau)|^2|\psi(|\xi|\tau^{-\frac{1}{n}})|^2\,{\rm dm}_n(\xi)\,{\rm dm}_1(\tau)\\
&=\Bigl(\int_1^\infty \tau |\varphi(\tau)|^2\,{\rm dm}_1(\tau)\Bigr)\Bigl( \int_{{\mathds{R}}^n} |\eta|^{2n}|\psi(\eta)|^2\,{\rm dm}_n(\eta)\Bigr)
\end{align*}
where we made the change of variable $\eta=\xi \,\tau^{-\frac{1}{n}}$. Both of these integrals are convergent thanks to the assumptions
on $\varphi$ and $\psi$ (and the fact that for $\tau\ge 1$, we have $\tau\le \tau^3$).

\noindent
On the other hand, $f\in \dot H^1({\mathds{R}};L^2({\mathds{R}}^n))$. Indeed,
\begin{align*}
\|f\|_{\dot H^1({\mathds{R}};L^2({\mathds{R}}^n))}^2&=
\int_1^\infty \int_{{\mathds{R}}^n} |\tau|^2|\varphi(\tau)|^2|\psi(|\xi|\tau^{-\frac{1}{n}})|^2\,{\rm dm}_n(\xi)\,{\rm dm}_1(\tau)\\
&=\Bigl(\int_1^\infty \tau^3 |\varphi(\tau)|^2\,{\rm dm}_1(\tau)\Bigr)\Bigl( \int_{{\mathds{R}}^n} |\psi(\eta)|^2\,{\rm dm}_n(\eta)\Bigr)
\end{align*}
where we made the same change of variable. Again, both of the integrals on the last line are convergent thanks to the assumptions
made on $\varphi$ and $\psi$.

\noindent
Now, if we want to calculate the $L^1({\mathds{R}};\dot H^{\frac{n}{2}}({\mathds{R}}^n))$ norm of $f$, let $T>1$:
\begin{align*}
I_T(f):=&\int_1^T \Bigl(\int_{{\mathds{R}}^n}|\xi|^{n}|\varphi(\tau)|^2|\psi(|\xi|\tau^{-\frac{1}{n}})|^2\,{\rm dm}_n(\xi)\Bigr)^{\frac{1}{2}}{\rm dm}_1(\tau)\\
=&\Bigl(\int_1^T \tau|\varphi(\tau)|\,{\rm dm}_1(\tau)\Bigr)
\Bigl(\int_{{\mathds{R}}^n}|\eta|^{n}|\psi(|\eta|)|^2\,{\rm dm}_n(\eta)\Bigr)^{\frac{1}{2}}.
\end{align*}
If we choose $\varphi(\tau)=\frac{1}{\tau^2\ln(1+\tau)}{\mathds{1}}_{[1,\infty)}(\tau)$, we have that 
\[
\int_{{\mathds{R}}} \tau^3|\varphi(\tau)|^2\,{\rm d}\tau =\int_1^\infty \tfrac{1}{\tau (\ln(1+\tau))^2}\, {\rm d}\tau<\infty
\]
but
\[
\int_1^T \tau |\varphi(\tau)|\,{\rm d}\tau =\int_1^T\tfrac{1}{\tau \ln(1+\tau)}\,{\rm d}\tau\xrightarrow[T\to+\infty]{}+\infty,
\]
so that $I_T(f)\xrightarrow[T\to+\infty]{}+\infty$ and then $f$ does not belong to $L^1({\mathds{R}};\dot H^{\frac{n}{2}}({\mathds{R}}^n))$.
\end{proof}

\begin{proposition}
Let $s\in {\mathds{R}}$. The dual space of $H^s({\mathds{R}}^n)$ (continuous linear forms on $H^s({\mathds{R}}^n)$) is isometrically
isomorphic to $H^{-s}({\mathds{R}}^n)$.
\end{proposition}

\begin{proof}
Let $\varphi \in H^{-s}({\mathds{R}}^n)$: we define $L_{\varphi}: H^s({\mathds{R}}^n)\to {\mathds{C}}$ by 
\[
L_{\varphi}(f):=\int_{{\mathds{R}}^n}\,\overline{{\mathscr{F}}(\varphi)(\xi)}{\mathscr{F}}(f)(\xi)\,{\rm dm}_n(\xi), \quad f\in H^s({\mathds{R}}^n).
\]
First, let's show that this is a ``good" definition: the integral exists since
\begin{itemize}
\item[-]
$\xi\mapsto (1+|\xi|^2)^{-\frac{s}{2}}{\mathscr{F}}(\varphi)(\xi) \in L^2({\mathds{R}}^n)$;
\item[-]
$\xi\mapsto (1+|\xi|^2)^{\frac{s}{2}}{\mathscr{F}}(f)(\xi) \in L^2({\mathds{R}}^n)$.
\end{itemize}
Therefore $\xi\mapsto (1+|\xi|^2)^{-\frac{s}{2}}\overline{{\mathscr{F}}(\varphi)(\xi)}(1+|\xi|^2)^{\frac{s}{2}}{\mathscr{F}}(f)(\xi)
=\overline{{\mathscr{F}}(\varphi)(\xi)}{\mathscr{F}}(f)(\xi)$ is an integrable function, and we have that
\[
|L_{\varphi}(f)|\le \|\varphi\|_{H^{-s}}\|f\|_{H^s}, \quad \mbox{ for all } f\in H^s({\mathds{R}}^n).
\]
Moreover, it is clear that $L_{\varphi}$ is linear on $H^s({\mathds{R}}^n)$. This proves that 
\begin{align*}
L:H^{-s}({\mathds{R}}^n)& \longrightarrow \bigl(H^s({\mathds{R}}^n)\bigr)'\\
\varphi&\longmapsto L_\varphi
\end{align*}
is a bounded homomorphism. Its norm is less than or equal to 1. To prove that it is equal to 1, choose 
$f={\mathscr{F}}^{-1}\bigl(\xi\mapsto (1+|\xi|^2)^{-s}{\mathscr{F}}(\varphi)(\xi)\bigr)$: $f\in H^s({\mathds{R}}^n)$
and $\|f\|_{H^s}=\|\varphi\|_{H^{-s}}$ and we have that
\begin{align*}
|L_\varphi(f)|&=\int_{{\mathds{R}}^n}\,\overline{{\mathscr{F}}(\varphi)(\xi)}{\mathscr{F}}(f)(\xi)\,{\rm dm}_n(\xi)\\
&=\int_{{\mathds{R}}^n}\,\overline{{\mathscr{F}}(\varphi)(\xi)} (1+|\xi|^2)^{-s}{\mathscr{F}}(\varphi)(\xi)\,{\rm dm}_n(\xi)\\
&=\int_{{\mathds{R}}^n}\,(1+|\xi|^2)^{-s}\bigl|{\mathscr{F}}(\varphi)(\xi)\bigr|^2\,{\rm dm}_n(\xi)
=\|\varphi\|_{H^{-s}}^2 = \|\varphi\|_{H^{-s}} \|f\|_{H^s}.
\end{align*}
This proves that the norm of $L_{\varphi}$ as a linear form on $H^{-s}({\mathds{R}}^n)$ is equal to $\|\varphi\|_{H^{-s}}$.
This means that $L:\varphi\mapsto L_\varphi$ is an isometric homomorphism from $H^{-s}({\mathds{R}}^n)$ to 
$\bigl(H^s({\mathds{R}}^n)\bigr)'$.

We need now to show that it is an isomorphism.

\noindent
Injectivity is immediate: if $L_\varphi=0$ then $\varphi=0$ in $H^{-s}({\mathds{R}}^n)$ (exercise!). 

\noindent
It remains to prove surjectivity. Let $T\in \bigl(H^s({\mathds{R}}^n)\bigr)'$ and define the linear form $M$ on $L^2({\mathds{R}}^n)$ by
\[
M(g):=T\bigl({\mathscr{F}}^{-1}(\xi\mapsto (1+|\xi|^2)^{-\frac{s}{2}}\hat{g}(\xi))\bigr), \quad g\in L^2({\mathds{R}}^n).
\]
This definition makes sense since ${\mathscr{F}}^{-1}(\xi\mapsto (1+|\xi|^2)^{-\frac{s}{2}}\hat{g}(\xi))$ belongs to 
$H^s({\mathds{R}}^n)$ for all $g\in L^2({\mathds{R}}^n)$ (check!), and we have that
\[
\|M\|_{(L^2)'}=\|T\|_{(H^s)'}.
\]
Now, by the Riesz-Fr\'echet representation theorem, there exists $\psi\in L^2({\mathds{R}}^n)$ such that
\[
M(g)=\int_{{\mathds{R}}^n} \psi(x)\overline{g(x)}\,{\rm dm}_n(x), \quad \forall\,g\in L^2({\mathds{R}}^n).
\]
This means that for all $f\in H^s({\mathds{R}}^n)$, $g={\mathscr{F}}^{-1}(\xi\mapsto (1+|\xi|^2)^{\frac{s}{2}}{\mathscr{F}}(f)(\xi))$ belongs
to $L^2({\mathds{R}}^n)$ and we have
\begin{align*}
T(f)&=M(g)= \int_{{\mathds{R}}^n} \overline{\psi(x)} g(x)\,{\rm dm}_n(x)
=  \int_{{\mathds{R}}^n} \overline{\hat{\psi}(\xi)}\hat{g}(\xi)\,{\rm dm}_n(\xi)\\
&= \int_{{\mathds{R}}^n} (1+|\xi|^2)^{\frac{s}{2}} \overline{\hat{\psi}(\xi)} {\mathscr{F}}(f)(\xi)\,{\rm dm}_n(\xi)\\
&= \int_{{\mathds{R}}^n} \,\overline{{\mathscr{F}}(\varphi)(\xi)}{\mathscr{F}}(f)(\xi)\,{\rm dm}_n(\xi)
=L_\varphi(f),
\end{align*}
where $\varphi={\mathscr{F}}^{-1}\bigl(\xi\mapsto (1+|\xi|^2)^{\frac{s}{2}} \hat{\psi}(\xi)\bigr)$, which belongs to 
$H^{-s}({\mathds{R}}^n)$ (check!).
\end{proof}

\begin{remark}
This duality result holds also for homogeneous Sobolev spaces but only in the range where they are Hilbert spaces, so
$\bigl(\dot H^s({\mathds{R}}^n)\bigr)'$ is isometrically isomorphic to $\dot H^{-s}({\mathds{R}}^n)$ for $-\frac{n}{2}<s<\frac{n}{2}$.
\end{remark}

\subsection{The Laplacian on Sobolev spaces}

The Laplace operator $\Delta$ is defined on ${\mathscr{S}}'({\mathds{R}}^n)$ as follows:
\[
\Delta u := \partial_1^2u+...+\partial_n^2u, \quad u\in {\mathscr{S}}'({\mathds{R}}^n).
\]
On the Fourier side, this gives:
\[
{\mathscr{F}}(-\Delta u)(\xi)=\bigl(|\xi_1|^2+...+|\xi_n|^2\bigr){\mathscr{F}}(u)(\xi) = |\xi|^2{\mathscr{F}}(u)(\xi), \quad \xi\in{\mathds{R}}^n.
\]

\begin{definition}
The {\tt fractional powers} of the Laplace operator are given by:
\[
(-\Delta)^su = {\mathscr{F}}^{-1}\bigl(\xi\mapsto |\xi|^{2s}{\mathscr{F}}(u)(\xi)\bigr),\quad u\in {\mathscr{S}}'({\mathds{R}}^n),\ s\in{\mathds{R}}.
\]
\end{definition}

\begin{proposition}
Let $s\in{\mathds{R}}$. Then $(-\Delta)^s$ maps $\dot H^\alpha({\mathds{R}}^n)$ to $\dot H^{\alpha-2s}({\mathds{R}}^n)$ boundedly for all
$\alpha\in{\mathds{R}}$. In particular, $(-\Delta)^s$ is invertible on $\dot H^\alpha({\mathds{R}}^n)$ and its inverse is $(-\Delta)^{-s}$.
\end{proposition}

\begin{proof}
This is immediate when working on the Fourier side.
\end{proof}

The following lemma will be useful for Navier-Stokes equations.

\begin{lemma}
\label{lem:multSobolev}
Let $n\ge 2$. There exists $C>0$ such that for all $f,g\in \dot H^{\frac{n-1}{2}}$ we have that $fg\in\dot H^{n/2-1}$ and
the following estimate holds
\[
\|fg\|_{\dot H^{n/2-1}}\le C\ \|f\|_{\dot H^{\frac{n-1}{2}}}\|g\|_{\dot H^{\frac{n-1}{2}}}.
\]
\end{lemma}

\begin{proof}
This can be proved using weighted Young's inequality for convolution.
We propose here an elementary proof for the particular case we are interested in. We will distinguish two cases: $n$ even and $n$ odd.

\noindent
{\tt Case $n$ even:}
Denote by $k$ the natural number $\frac{n}{2}-1$. Then Leibniz rule gives for all $f,g\in \dot H^{\frac{n-1}{2}}$:
\[
D^k(fg)=\sum_{j=0}^k c_{j,k} D^jfD^{k-j}g.
\]
Since $f\in \dot H^{\frac{n-1}{2}}$, $D^jf\in \dot H^{\frac{n-1}{2}-j}\hookrightarrow L^p$ with $\frac{1}{p}=\frac{j}{n}+\frac{1}{2n}$.
For $g\in \dot H^{\frac{n-1}{2}}$, we have that $D^{k-j}g\in \dot H^{j+\frac{1}{2}}\hookrightarrow L^q$ with $\frac{1}{q}=\frac{1}{2}-\frac{1}{p}$.
This gives
\[
\|fg\|_{\dot H^{n/2-1}} \sim \|D^k(fg)\|_2\lesssim \|f\|_{\dot H^{\frac{n-1}{2}}} \|g\|_{\dot H^{\frac{n-1}{2}}}
\]
as claimed.

\noindent
{\tt Case $n$ odd:}
We denote now by $k$ the natural number $\frac{n-1}{2}$. As before, the Leibniz rule gives
\[
D^k(fg)=\sum_{j=0}^kc_{j,k} D^jfD^{k-j}g.
\]
As in the previous case, $D^jf\in \dot H^{\frac{n-1}{2}-j}\hookrightarrow L^p$ with $\frac{1}{p}=\frac{j}{n}+\frac{1}{2n}$.
For $g\in \dot H^{\frac{n-1}{2}}$, we have that $D^{k-j}g\in \dot H^j \hookrightarrow L^q$ with $\frac{1}{q}=\frac{1}{2}-\frac{j}{n}$.
This gives
\[
\|D^k(fg)\|_r\lesssim \|f\|_{\dot H^{\frac{n-1}{2}}} \|g\|_{\dot H^{\frac{n-1}{2}}}
\]
where $\frac{1}{r}=\frac{1}{2}+\frac{1}{2n}$. By Sobolev embedding and duality, we see that $\dot H^{-1/2}\hookrightarrow L^r$, 
from which we infer
\[
\|(-\Delta)^{-1/4}D^k(fg)\|_2\sim \|D^k(fg)\|_{\dot H^{-1/2}}\lesssim \|D^k(fg)\|_r.
\]
This gives
\[
\|fg\|_{\dot H^{n/2-1}}\sim \|(-\Delta)^{-1/4}D^k(fg)\|_2 \lesssim  \|f\|_{\dot H^{\frac{n-1}{2}}} \|g\|_{\dot H^{\frac{n-1}{2}}}
\]
as claimed.
\end{proof}

Next result is the well-known boundedness of {\tt Riesz transforms} ${\mathcal{R}}_k=\partial_k(-\Delta)^{-\frac{1}{2}}$, $k=1,...,n$.

\begin{remark}
In dimension 1 ($n=1$), there is only one Riesz transform: it is called the {\tt Hilbert transform}.
\end{remark}

\begin{proposition}
For all $k\in \{1,...,n\}$, the Riesz transforms ${\mathcal{R}}_k$ are bounded operators on $\dot H^s({\mathds{R}}^n)$ for all $s\in{\mathds{R}}$.
\end{proposition}

\begin{proof}
Let $s\in{\mathds{R}}$ and $k\in\{1,...,n\}$. For all $f\in \dot H^s({\mathds{R}}^n)$, we have that 
\[
{\mathscr{F}}({\mathcal{R}}_k(f))(\xi)=\tfrac{i\xi_k}{|\xi|}\hat{f}(\xi), \quad  \xi\in{\mathds{R}}^n.
\]  
So that
\begin{align*}
\|{\mathcal{R}}_k(f))\|_{\dot H^s}=&\Bigl(\int_{{\mathds{R}}^n}|\xi|^{2s}|{\mathscr{F}}({\mathcal{R}}_k(f))(\xi)|^2\,{\rm d}m_n(\xi)\Bigr)^{\frac{1}{2}}\\
=&\Bigl(\int_{{\mathds{R}}^n}|\xi|^{2s}|\hat{f}(\xi)|^2\bigl|\tfrac{i\xi_k}{|\xi|}\bigr|^2\,{\rm d}m_n(\xi)\Bigr)^{\frac{1}{2}}
\le\|f\|_{\dot H^s}
\end{align*}
since $\bigl|\tfrac{i\xi_k}{|\xi|}\bigr|\le 1$ for all $\xi\in{\mathds{R}}^n\setminus\{0\}$.
\end{proof}

\begin{theorem}
For all $k\in \{1,...,n\}$, the Riesz transforms ${\mathcal{R}}_k$ are bounded operators on $L^p({\mathds{R}}^n)$ for all $1<p<\infty$.
\end{theorem}

\begin{proof}[Sketch of the proof]
First, we just proved that the Riesz transforms are bounded on $L^2({\mathds{R}}^n)$ (previous proposition with $s=0$).

\noindent
Then, using {\tt Calder\'on-Zygmund decomposition}, we prove that for all $k\in\{1,...,n\}$, ${\mathcal{R}}_k$ is of weak type $(1,1)$, that is
\[
\sup_{\lambda>0}\bigl(\lambda\bigl|\{|{\mathcal{R}}_k f|>\lambda\}\bigr|\bigr)\lesssim \|f\|_1, \quad f\in L^1({\mathds{R}}^n).
\]

\noindent
We then argue by interpolation ({\tt Marcinkiewicz interpolation} theorem) to prove that for all $k\in\{1,...,n\}$, ${\mathcal{R}}_k$
is bounded on $L^p({\mathds{R}}^n)$, $1<p\le 2$.

\noindent
We conclude by duality to prove the boundedness on $L^q({\mathds{R}}^n)$, $2\le q <\infty$.
\end{proof}

\section{The heat equation}

\subsection{The stationary case}

We consider the following equation
\begin{equation}
\label{chaleurstationnaire}
u - \Delta u = f, \quad a.e. \mbox{ in } {\mathds{R}}^n
\end{equation}
for $f\in L^2({\mathds{R}}^n)$. We are looking for a solution $u\in H^2({\mathds{R}}^n)$ satisfying \eqref{chaleurstationnaire}.

\begin{theorem}
For all $f\in L^2({\mathds{R}}^n)$, there exists a unique solution $u\in H^2({\mathds{R}}^n)$ of \eqref{chaleurstationnaire} given by
\begin{equation} 
\label{solstat}
u ={\mathscr{F}}^{-1} \left( \xi\mapsto \tfrac{1}{1 +|\xi|^2} \  \hat{f}(\xi) \right).
\end{equation}
\end{theorem}

\begin{proof}
Uniqueness : Assume that there are two solutions $u,v \in H^2({\mathds{R}}^n)$ of \eqref{chaleurstationnaire}.
Let then $w=u-v$ : $w \in H^2({\mathds{R}}^n)$ and satisfies $w=\Delta w$. Applying Fourier transform to this equation, we obtain
\[
\hat{w}(\xi) = -|\xi|^2 \hat{w}(\xi), \quad a.e.\ \xi\in{\mathds{R}}^n.
\]
Indeed, since $w \in H^2({\mathds{R}}^n)$, the partial derivatives of order 2, $\partial_j\partial_kw$, are in $L^2({\mathds{R}}^n)$ 
for all $j,k=1,\dots, n$ and we have that  
\[
{\mathscr{F}}(\partial_j\partial_k w)(\xi) = (i\xi_j)(i\xi_k) \hat{w}(\xi), \quad a.e.
\]
Therefore, for almost all $\xi \in {\mathds{R}}^n$, we have that $(1+|\xi|^2)\hat{w}(\xi)=0$, so that $\hat{w} =0$ as a function of 
$L^2({\mathds{R}}^n)$. Thanks to the fact that the Fourier transform is an isomorphism on $L^2({\mathds{R}}^n)$, we conclude that $w$ is
zero and then $u=v$.

\noindent
Existence : we show that the function given by \eqref{solstat} is solution of \eqref{chaleurstationnaire}. Let $u$ be the function
given by \eqref{solstat}. We have then 
\[
\hat{u} (\xi) = \tfrac{1}{1 +|\xi|^2} \  \hat{f}(\xi), \quad a.e. \ t\in {\mathds{R}}^n.
\]
Since $f \in L^2({\mathds{R}}^n)$, $ \hat{f}\in L^2({\mathds{R}}^n)$ and by definition of $H^2({\mathds{R}}^n)$, it is
clear that $u\in  H^2({\mathds{R}}^n)$. Moreover, we have  
\[
{\mathscr{F}} (u -\Delta u)(\xi) = \hat{u}(\xi) + \sum_{k=1}^n |\xi_k|^2 \hat{u}(\xi) = (1 +|\xi|^2)\hat{u}(\xi) = \hat{f}(\xi), \quad a.e.\ \xi\in {\mathds{R}}^n,
\]
by the definition of $u$. By taking the inverse Fourier transform, we obtain that $u$ is a solution of \eqref{chaleurstationnaire}. 
Moreover, $\|u\|_2 \le \|f\|_2$.
\end{proof}

\begin{remark}
One could also consider the equation $\lambda u - \Delta u = f \in L^2({\mathds{R}}^n)$ which can be solved exactly as
\eqref{chaleurstationnaire} as soon as $\lambda \in {\mathds{C}} \setminus ]-\infty , 0]$. 
\end{remark}

\begin{remark}
The same proof is valid if $f\in H^s({\mathds{R}}^n)$ for $s\in {\mathds{R}}$. We find then a solution $u \in H^{s+2}({\mathds{R}}^n)$.
\end{remark}

\subsection{The evolution case}

Let's consider now the evolution heat equation
\begin{equation}
\label{chaleur}
\left\{
\begin{array}{rcll}
\partial_t u - \Delta u &=& 0 & \mbox{ in } (0,\infty) \times {\mathds{R}}^n\\
u(0,x) & = & u_0(x) &  \mbox{ in } {\mathds{R}}^n,
\end{array}
\right.
\end{equation}
for an initial data $u_0 \in L^2({\mathds{R}}^n)$. We are looking for a solution $u$ continuous in time ($t$) with $\partial_tu$ also
continuous in time, and $u$ smooth in space ($x$).

\begin{theorem}
For all $u_0 \in L^2({\mathds{R}}^n)$, there exists a unique solution $u$ of \eqref{chaleur} satisfying for all fixed $x \in {\mathds{R}}^n$,
the map $t \mapsto u(t,x)$ is in ${\mathscr{C}}([0,\infty[ ) \cap {\mathscr{C}}^1(]0,\infty[)$ and for fixed $t > 0$, the map 
$x \mapsto u(t,x)$ is in ${\mathscr{S}}({\mathds{R}}^n)$. This solution is given by
\begin{equation} 
\label{noyau}
u(t,x) = \tfrac{1}{(4\pi t)^{\frac n2}} \int_{{\mathds{R}}^n} e^{-\frac{|x-y|^2}{4t}} u_0(y) dy, \quad t>0, \ x \in {\mathds{R}}^n.
\end{equation}
\end{theorem}

\begin{definition}
The function $(0,\infty)\times{\mathds{R}}^n\ni (t,x)\mapsto p_t(x)=\frac{1}{(4\pi t)^{\frac n2}}\,e^{-\frac{|x|^2}{4t}}$ is called
the {\tt heat kernel} associated with the heat equation in ${\mathds{R}}^n$.
\end{definition}

\begin{proof}
As for the problem \eqref{chaleurstationnaire}, we apply the Fourier transform in the $x$ variable to the equation \eqref{chaleur}. 
Uniqueness is proven as follows. Assume that there are two solutions $u$ and $v$ satisfying the conditions of the theorem. Then
we set $w=u-v$ and we have that
\[
\partial_t w = \Delta w, \quad w(0,x) = 0.  
\]
Applying the Fourier transform in the $x$ variable to this equation, we obtain, for fixed $\xi\in{\mathds{R}}^n$
\begin{equation} 
\label{phi}
\varphi_\xi'(t) = -|\xi|^2 \varphi_\xi (t), \ t>0, \quad \varphi_\xi(0) = 0,
\end{equation}
where $\varphi_\xi (t) = {\mathscr{F}} (x \mapsto w(t,x))(\xi)$ for $t>0$. Therefore, \eqref{phi} is an ODE in $t$ that
can be solved as follows:
\[
\varphi_\xi(t) = \varphi_\xi(0) e^{-t|\xi|^2} = 0.
\]
This shows that $w(t,x) = {\mathscr{F}}^{-1} (\xi \mapsto \varphi_\xi(t))(x) = 0$.

\noindent
Existence : we show that the function $u$ given by \eqref{noyau} is a solution of \eqref{chaleur} as stated in the theorem. 
It is clear that $u$ satisfies the regularity hypotheses given in the theorem. The continuity at $t=0$ is worth looking into. Let's start 
with $t>0$:
\[
u(t,x) = q_t * u_0 (x), \quad x \in {\mathds{R}}^n,
\]
where $q_t (x) =(2t)^{- \frac{n}{2}} \ e^{-\frac{|x|^2}{4t}}$, if we define the convolution with the normalised Lebesgue measure ${\rm dm}_n$. 
Therefore, we have 
\[
{\mathscr{F}}(x \mapsto u(t,x))(\xi) = \hat{q}_t(\xi) \hat{u}_0(\xi), \quad \xi\in{\mathds{R}}^n.
\]
It is easy to compute $\hat{q}_t(\xi)$ thanks to Exercise~\ref{ex:gaussian} and Theorem~\ref{thm:Fourier1}:
for all $\xi \in {\mathds{R}}^n$ and $t>0$, $\hat{q}_t(\xi) = e^{-t|\xi|^2}$. Therefore
\begin{align*}
{\mathscr{F}} \bigl(\partial_t (q_t * u_0 ) \bigr)(\xi) &=  \partial_t \bigl(\hat{q}_t(\xi) \hat{u}_0(\xi) \bigr) \ =\  -|\xi|^2 \hat{q}_t(\xi) \hat{u}_0(\xi) \\
&=  -|\xi|^2 {\mathscr{F}} (q_t * u_0 ) (\xi) \ = \ {\mathscr{F}} (\Delta (q_t * u_0 ))(\xi).
\end{align*}
Since the Fourier transform is invertible in $L^2({\mathds{R}}^n)$, we obtain $\partial_t u = \Delta u$ for almost every
$(t,x)\in(0,\infty)\times{\mathds{R}}^n$. It remains to prove that $u$ is continuous at $t=0$, $i.e.$,
\[
\lim_{t\to 0^+} \| u(t,\cdot) - u_0\|_2 = 0,
\]
or, equivalently, since the Fourier transform is an isometry
\[
\lim_{t\to 0^+} \| (\hat{q}_t -1) \hat{u}_0\|_2 = 0.
\]
We have, for all $\xi\in{\mathds{R}}^n$,
\[
|e^{-t|\xi|^2} - 1| \xrightarrow[t\to 0^+]{} 0
\]
and
\[
|(\hat{q}_t -1) \hat{u}_0|\le 2 |\hat{u}_0|\in L^2({\mathds{R}}^n).
\]
This implies that, thanks to the dominated convergence theorem, 
\[
\| u(t,\cdot) - u_0\|_2 = \| (\hat{q}_t -1) \hat{u}_0\|_2 
=\Bigl(\int_{{\mathds{R}}^n} |e^{-t|\xi|^2} - 1|^2 |\hat{u}_0(\xi)|^2\,{\rm dm}_n(\xi) \Bigr)^{\frac{1}{2}}\xrightarrow[t\to 0^+]{} 0. \qedhere
\]
\end{proof}

\begin{definition}[Heat semigroup]
For $u_0\in \dot H^s({\mathds{R}}^n)$, we define $e^{t\Delta}u_0$ for $t\ge 0$ by 
\[
e^{t\Delta}u_0=q_t * u_0 ={\mathscr{F}}^{-1}\bigl(\xi\mapsto e^{-t|\xi|^2}\hat{u}_0(\xi)\bigr), \quad t>0\quad \mbox{ and }\quad e^{0\Delta}u_0=u_0.
\]
The family of operators $(e^{t\Delta})_{t\ge0}$ is called the {\tt heat semigroup}.
\end{definition}

\noindent
We can also consider the heat equation with a forcing term: \eqref{chaleur} becomes
\begin{equation}
\label{eq:chaleur}
\left\{
\begin{array}{rcll}
\partial_t u - \Delta u &=& f & \mbox{ in } (0,\infty) \times {\mathds{R}}^n\\
u(0,x) & = & u_0(x) &  \mbox{ in } {\mathds{R}}^n,
\end{array}
\right.
\end{equation}
for smooth (or integrable) enough $f : (0,\infty)\times {\mathds{R}}^n\to {\mathds{R}}$. Recall that for an ordinary 
differential equation, the solution of 
\[
\varphi'(t)+\omega \,\varphi(t) = g(t), \ t>0, \quad \varphi(0)=0,
\]
for a given $\omega\in{\mathds{R}}$ and a continuous function $g:[0,\infty)\to {\mathds{R}}$, is given by the variation of constant formula
\[
\varphi(t)=\int_0^t e^{-\omega(t-s)}g(s)\,{\rm d}s, \quad t\ge 0.
\]

\begin{theorem}
\label{thm:classicalsol}
Let $f:[0,\infty)\times{\mathds{R}}^n\to {\mathds{R}}$ be continuous in the first variable, with values in $L^2({\mathds{R}}^n)$. Assume
moreover that $\partial_tf\in L^1(0,\infty;L^2({\mathds{R}}^n))$ so that $f(t)=f(0)+\int_0^t \partial_sf(s)\,{\rm d}s$, $t\ge 0$.
Then there is a unique function $u\in {\mathscr{C}}([0,\infty);L^2({\mathds{R}}^n))\cap {\mathscr{C}}^1([0,\infty);L^2({\mathds{R}}^n))
\cap {\mathscr{C}}([0,\infty);H^2({\mathds{R}}^n))$ solution of 
\[
\partial_tu-\Delta u =f \mbox{ in }(0,\infty)\times{\mathds{R}}^n, \quad u(0,\cdot)=0.
\]
Moreover, for all $t>0$, $u$ is given by
\begin{equation}
\label{eq:solchaleur}
u(t,\cdot)=\int_0^te^{(t-s)\Delta}f(s,\cdot)\,{\rm d}s
={\mathscr{F}}^{-1}\Bigl(\xi\mapsto \int_0^te^{-(t-s)|\xi|^2}{\mathscr{F}}(f(s,\cdot))(\xi)\,{\rm d}s\Bigr).
\end{equation}
\end{theorem}

\begin{proof}
Let $u$ be given by \eqref{eq:solchaleur}. Then $u$ belongs to ${\mathscr{C}}([0,\infty);L^2({\mathds{R}}^n))$: this is immediate 
by composition and the dominated convergence theorem, since 
$\displaystyle{\sup_{\stackrel{\xi\in{\mathds{R}}^n}{0\le s\le t}}|e^{-(t-s)|\xi|^2}|\le 1}$.

\noindent
It follows then that $u(0,\cdot)=0$.

\noindent
To prove that $u\in {\mathscr{C}}([0,\infty);\dot H^2({\mathds{R}}^n))$, it suffices to prove that 
\[
(t,\xi)\mapsto \int_0^t |\xi|^2 e^{-(t-s)|\xi|^2}{\mathscr{F}}(f(s,\cdot))(\xi)\,{\rm d}s
\] 
belongs to ${\mathscr{C}}([0,\infty);L^2({\mathds{R}}^n))$. Using an integration by parts (in $s$) and
the equality $f(t)=f(0)+\int_0^t \partial_sf(s)\,{\rm d}s$,
we have that
\begin{align*}
&\int_0^t |\xi|^2 e^{-(t-s)|\xi|^2}{\mathscr{F}}(f(s,\cdot))(\xi)\,{\rm d}s\\
=& {\mathscr{F}}(f(t,\cdot))(\xi) - {\mathscr{F}}(f(0,\cdot))(\xi) -\int_0^te^{-(t-s)|\xi|^2}{\mathscr{F}}(\partial_sf(s,\cdot))(\xi)\,{\rm d}s
\end{align*}
The first two terms are clearly continuous in time with values in $L^2({\mathds{R}}^n)$. The integral term is also continuous by the 
dominated convergence theorem.

\noindent
Finally, it is immediate that $u$ satisfies $\partial_t u-\Delta u = f$ (clear on the Fourier side), so that 
$\partial_t u\in {\mathscr{C}}([0,\infty);L^2({\mathds{R}}^n))$ and then $u\in  {\mathscr{C}}^1([0,\infty);L^2({\mathds{R}}^n))$.
\end{proof}

\begin{definition}
A function $u\in {\mathscr{C}}^1([0,\infty);L^2({\mathds{R}}^n))\cap {\mathscr{C}}([0,\infty);H^2({\mathds{R}}^n))$ satisfying
\[
\partial_t u-\Delta u =f \quad \mbox{and} \quad u(0)=u_0
\]
for $f\in {\mathscr{C}}([0,\infty);L^2({\mathds{R}}^n))$ and $u_0\in L^2({\mathds{R}}^n)$
is called a {\tt classical solution} of \eqref{eq:chaleur}.
\end{definition}

\begin{exercise}
Let $s\in{\mathds{R}}$.
Prove that if $f$ satisfies the assumptions of Theorem~\ref{thm:classicalsol} with $L^2({\mathds{R}}^n)$ replaced by 
$\dot H^s({\mathds{R}}^n)$, then the solution $u$ given by the formula \eqref{eq:solchaleur} belongs to the space
${\mathscr{C}}([0,\infty);\dot H^{s+2}({\mathds{R}}^n))\cap {\mathscr{C}}^1([0,\infty);\dot H^s({\mathds{R}}^n))$ and
satisfies $\partial_tu-\Delta u=f$, $u(0, \cdot)=0$.
\end{exercise}

\subsection{Maximal regularity}

In this paragraph, we will focus on {\tt mild solutions} of \eqref{eq:chaleur} with $u_0=0$, as in the previous paragraph, but now,
we lower the conditions on $f$. The question is the following:
\begin{center}
\begin{minipage}[t]{12cm}
for $1<p<\infty$ and $f\in L^p(0,\infty;L^2({\mathds{R}}^n))$, does $u$ given by \eqref{eq:solchaleur} belong to 
$L^p(0,\infty;\dot H^2({\mathds{R}}^n))$ with $\partial_t u \in L^p(0,\infty; L^2({\mathds{R}}^n))$?
\end{minipage}
\end{center}
In other words, we are interested in whether the terms ($\partial_tu$ and $\Delta u$) in \eqref{eq:chaleur} with $u_0=0$ 
belong to $L^p(0,\infty; L^2({\mathds{R}}^n))$. We will start with the case $p=2$

\begin{proposition}[Maximal regularity -- $L^2$]
\label{prop:MRL2}
For all $f\in L^2(0,\infty; L^2({\mathds{R}}^n))$, there exists a unique $u\in L^2(0,\infty; \dot H^2({\mathds{R}}^n))$ with
$\partial_t u\in L^2(0,\infty; L^2({\mathds{R}}^n))$ satisfying $\partial_t u-\Delta u=f$ for almost all $(t,x)\in (0,\infty)\times{\mathds{R}}^n$
and $u(0)=0$. Moreover, $u$ is given by \eqref{eq:solchaleur}.
\end{proposition}

\begin{proof}
Thanks to what we have already proved, it suffices to prove that for all forcing term $f\in L^2(0,\infty; L^2({\mathds{R}}^n))$, 
the function
\[
\Phi:(t,\xi)\mapsto \int_0^t |\xi|^2 e^{-(t-s)|\xi|^2}{\mathscr{F}}_x(f(s,\cdot))(\xi)\,{\rm d}s
\]
belongs to $L^2(0,\infty;L^2({\mathds{R}}^2))$. For $\xi\in{\mathds{R}}^n$, let $g_{\xi}(t):= |\xi|^2e^{-t|\xi|^2} \mathds{1}_{(0,\infty)}(t)$,
$t\in{\mathds{R}}$ and $h_{\xi}(t):=\mathds{1}_{(0,\infty)}(t){\mathscr{F}}(f(t,\cdot))(\xi)$, $t\in{\mathds{R}}$.
Then we have $\Phi(t,\xi)=(2\pi)^{-\frac{1}{2}}g_{\xi}*h_{\xi}(t)$ for all $t\ge 0$ and $\Phi(t,\xi)=0$ if $t<0$, where the convolution is to 
be understood in the time variable. Taking the Fourier transform in $t$, denoted by ${\mathscr{F}}_t$, we obtain
\[
{\mathscr{F}}_t(\Phi(\cdot,\xi))(\tau) = {\mathscr{F}}_t(g_{\xi})(\tau){\mathscr{F}}_t(h_{\xi})(\tau).
\]
The computation of ${\mathscr{F}}_t(g_{\xi})$ gives directly ${\mathscr{F}}_t(g_{\xi})(\tau)=(2\pi)^{-\frac{1}{2}}\frac{|\xi|^2}{i\tau+|\xi|^2}$, which 
belongs to $L^\infty({\mathds{R}}\times{\mathds{R}}^n)$.
We also have that ${\mathscr{F}}_t(h_{\xi})(\tau)={\mathscr{F}}_{t,x}\bigl((t,x)\mapsto f(t,x)\,{\mathds{1}}_{[0,\infty)}(t)\bigr)(\tau,\xi)$, so that 
$(\tau,\xi)\mapsto {\mathscr{F}}_t(h_{\xi})(\tau)$ belongs to $L^2({\mathds{R}}\times{\mathds{R}}^n)$. This proves that 
\[
(t,\xi)\mapsto \Phi (t,\xi)={\mathscr{F}}_{\tau}^{-1}\bigl(\tau\mapsto{\mathscr{F}}_t(g_{\xi})(\tau){\mathscr{F}}_t(h_{\xi})(\tau)\bigr)(t)
\]
belongs to $L^2({\mathds{R}}\times{\mathds{R}}^2))=L^2({\mathds{R}};L^2({\mathds{R}}^2))$. Since $\Phi(t,\xi)=0$ if $t<0$, we obtain 
the claim.
\end{proof}

\begin{exercise}
Let $s\in {\mathds{R}}$. Prove that
for all $f\in L^2(0,\infty; \dot H^s({\mathds{R}}^n))$, there exists a unique $u\in L^2(0,\infty; \dot H^{s+2}({\mathds{R}}^n))$ with
$\partial_t u\in L^2(0,\infty; \dot H^s({\mathds{R}}^n))$ satisfying $\partial_t u-\Delta u=f$ for almost all $(t,x)\in (0,\infty)\times{\mathds{R}}^n$
and $u(0)=0$. Moreover, $u$ is given by \eqref{eq:solchaleur}.
\end{exercise}

\begin{proposition}[Mixed derivatives]
\label{prop:mixedMaxReg}
Let $s\in{\mathds{R}}$.
For all $f\in L^2(0,\infty; H^s({\mathds{R}}^n))$, there exists a unique $u\in \dot H^{\alpha}(0,\infty ;\dot H^{s+2(1-\alpha)}({\mathds{R}}^n))$ 
for all $\alpha\in [0,1]$ satisfying $\partial_t u-\Delta u=f$ 
for almost all $(t,x)\in (0,\infty)\times{\mathds{R}}^n$
and $u(0)=0$. As before, $u$ is given by \eqref{eq:solchaleur}.
\end{proposition}

\begin{proof}
It was proved in Proposition~\ref{prop:MRL2} that $u$ given by \eqref{eq:solchaleur} is solution of $\partial_tu-\Delta u=f$, $u(0)=0$.
The cases $\alpha=0$ and $\alpha =1$ were proved in Proposition~\ref{prop:MRL2}. Now, let $\alpha\in(0,1)$. Taking the Fourier
transform in $t$ and $x$ of the expression \eqref{eq:solchaleur}, we obtain
\[
{\mathscr{F}}_{t,x}(u)(\tau,\xi)=\tfrac{1}{\sqrt{2\pi}(|\xi|^2+i\tau)}\,{\mathscr{F}}_{t,x}(f)(\tau,\xi), \quad \tau\in{\mathds{R}}, \xi \in {\mathds{R}}^n.
\]
Therefore, for all $\alpha\in(0,1)$, we have that
\[
\bigl\|u\bigr\|_{\dot H^\alpha(0,\infty;\dot H^{s+2(1-\alpha)}({\mathds{R}}^n))}
=\Bigl\|(\tau,\xi)\mapsto \tfrac{|\tau|^\alpha|\xi|^{2(1-\alpha)}}{\sqrt{2\pi}(|\xi|^2+i\tau)}\,
|\xi|^{\frac{s}{2}}{\mathscr{F}}_{t,x}(f)(\tau,\xi)\Bigr\|_{L^2({\mathds{R}}\times{\mathds{R}}^n)}.
\]
By the assumption $f\in L^2(0,\infty; \dot H^s({\mathds{R}}^n))$, the function $(\tau,\xi)\mapsto |\xi|^{\frac{s}{2}}{\mathscr{F}}_{t,x}(f)(\tau,\xi)$
belongs to $L^2({\mathds{R}}\times{\mathds{R}}^n)$. Then, it suffices to prove that
$(\tau,\xi)\mapsto \frac{|\tau|^\alpha|\xi|^{2(1-\alpha)}}{\sqrt{2\pi}(|\xi|^2+i\tau)}$ is a bounded function on ${\mathds{R}}\times{\mathds{R}}^n$.
This is immediate when writing the expression with the change of variable $\sigma=\frac{\tau}{|\xi|^2}$:
\[
\sup_{(\tau,\xi)\in{\mathds{R}}\times{\mathds{R}}^n}\Bigl|\frac{|\tau|^\alpha|\xi|^{2(1-\alpha)}}{\sqrt{2\pi}(|\xi|^2+i\tau)}\Bigr|
=\sup_{\sigma\in{\mathds{R}}}\Bigl|\frac{|\sigma|^\alpha}{\sqrt{2\pi}(1+i\sigma)}\Bigr| 
=\tfrac{1}{\sqrt{2\pi}}\,\alpha^{\frac{\alpha}{2}}(1-\alpha)^{1-\frac{\alpha}{2}}.
\]
This last value is obtained by studying the function $y\mapsto \frac{|y|^\alpha}{\sqrt{1+y^2}}$ on ${\mathds{R}}$ (its maximum is attained
at $y=\sqrt{\frac{\alpha}{1-\alpha}}$).
\end{proof}

\begin{theorem}[Maximal regularity -- $L^p$]
\label{thm:MRLp}
For all $f\in L^p(0,\infty; L^2({\mathds{R}}^n))$, there exists a unique $u\in L^p(0,\infty; \dot H^2({\mathds{R}}^n))$ with
$\partial_t u\in L^p(0,\infty; L^2({\mathds{R}}^n))$ satisfying $\partial_t u-\Delta u=f$ for almost all $(t,x)\in (0,\infty)\times{\mathds{R}}^n$
and $u(0)=0$. Moreover, $u$ is given by \eqref{eq:solchaleur}.
\end{theorem}

\begin{proof}
We need to prove that the operator
\begin{equation}
\label{eq:opMR}
{\mathscr{M}} : 
f \longmapsto \Bigl(t \mapsto \int_0^t (-\Delta)e^{(t-s)\Delta}f(s)\,{\rm d}s\Bigr)
\end{equation}
is bounded from $L^p(0,\infty; L^2({\mathds{R}}^n))$ to $L^p(0,\infty; L^2({\mathds{R}}^n))$. 
Recall that for $t\ge 0$,
\[ 
\int_0^t (-\Delta)e^{(t-s)\Delta}f(s)\,{\rm d}s=
{\mathscr{F}}_{\xi}^{-1}\bigl(\xi\mapsto \int_0^t |\xi|^2 e^{-(t-s)|\xi|^2}{\mathscr{F}}_x(f(s,\cdot))(\xi)\,{\rm d}s\bigr).
\]
The previous proposition proves the case $p=2$.
To simplify the notations, we denote by $X$ the space $L^2({\mathds{R}}^n)$.

\noindent
We start with showing that ${\mathscr{M}}$ maps $L^1(0,\infty;X)$ to $L^1_w(0,\infty;X)$ where
the weak-$L^1$ space $L^1_w$ consists in measurable functions $v$ such that 
\[
\|v\|_{L^1_w(0,\infty;X)}:=\sup_{\lambda>0}\lambda \bigl|\{t>0;\|v(t)\|_X>\lambda\}\bigr| <\infty.
\]
Let $v={\mathscr{M}}(f)$ and assume that $f\in L^2(0,\infty;X)\cap L^1(0,\infty;X)$. We decompose $f$ via the Calder\'on-Zygmund
method into a ``good" part $g$ and a ``bad" part $b=\sum b_k$: $f=g+\sum b_k$ at a level $\lambda>0$, where
\begin{itemize}
\item
$b_k = \bigl(f-\fint_{I_k}f(s,\cdot)\,{\rm d}s\bigr){\mathds{1}}_{I_k}$, $k\in {\mathds{N}}$, where $I_k$ are disjoint intervals;
\item
$\sum_k |I_k|\le \lambda^{-1}\|f\|_{L^1(0,\infty;X)}$;
\item
$\|g\|_{L^\infty(0,\infty;X)}\le 2\lambda$.
\end{itemize}
This construction implies that $\|g\|_{L^1(0,\infty;X)}\le \|f\|_{L^1(0,\infty;X)}$ and $\|b\|_{L^1(0,\infty;X)}\le 2 \|f\|_{L^1(0,\infty;X)}$.

\noindent
We have that ${\mathscr{M}}(f)={\mathscr{M}}(g)+ {\mathscr{M}}(b)$ and then
\begin{align}
&\bigl\{t>0; \|{\mathscr{M}}(f)(t)\|_X> \lambda\bigr\} \nonumber\\
\subset &\bigl\{t>0; \|{\mathscr{M}}(g)(t)\|_X> \tfrac{\lambda}{2}\bigr\}
\cup \bigl\{t>0; \|{\mathscr{M}}(b)(t)\|_X> \tfrac{\lambda}{2}\bigr\}. \label{eq:0}
\end{align}
Since $g\in L^\infty(0,\infty;X)\cap L^1(0,\infty;X)$, by H\"older inequality we infer that $g\in L^2(0,\infty;X)$ and
$\|g\|_{L^2(0,\infty;X)}\le \|g\|_{L^1(0,\infty;X)}^{1/2}\|g\|_{L^\infty(0,\infty;X)}^{1/2}\le \sqrt{2\lambda \|f\|_{L^1(0,\infty;X)}}$. Since
${\mathscr{M}}$ is bounded on $L^2(0,\infty;X)$ by the previous proposition, we have that
\begin{equation}
\label{eq:1}
\Bigl|\bigl\{t>0; \|{\mathscr{M}}(g)(t)\|_X> \tfrac{\lambda}{2}\bigr\}\Bigr|\le \frac{\|{\mathscr{M}}(g)\|^2_{L^2(0,\infty;X)}}{(\lambda/2)^2}
\le \tfrac{8}{\lambda}\, \|{\mathscr{M}}\|_{L^2(X)\to L^2(X)}^2 \|f\|_{L^1(0,\infty;X)}.
\end{equation}
It remains to estimate $\Bigl|\bigl\{t>0; \|{\mathscr{M}}(b)(t)\|_X> \tfrac{\lambda}{2}\bigr\}\Bigr|$. We decompose the set as follows:
\[
\bigl\{t>0; \|{\mathscr{M}}(b)(t)\|_X> \tfrac{\lambda}{2}\bigr\}\subset 
E \cup \bigl\{t\in (0,\infty)\setminus E; \|{\mathscr{M}}(b)(t)\|_X> \tfrac{\lambda}{2}\bigr\}
\]
where $E=\cup_{k\in{\mathds{N}}} \tilde{I}_k$ where $\tilde{I}_k$ is the double of $I_k$. We have that
\begin{equation}
\label{eq:2}
|E| \le \sum_{k\in{\mathds{N}}}|\tilde{I}_k| \le 2 \sum_{k\in{\mathds{N}}}|I_k|\le 2\,\tfrac{1}{\lambda}\,\|f\|_{L^1(0,\infty;X)}.
\end{equation}
To estimate the size of $\bigl\{t\in (0,\infty)\setminus E; \|{\mathscr{M}}(b)(t)\|_X> \tfrac{\lambda}{2}\bigr\}$, let first denote by $s_k$ 
the center of the interval $I_k$ and thanks to the property that $\int_0^\infty b_k(s,\cdot)\,{\rm d}s=0$, we have
\begin{align*}
&\int_{(0,\infty)\setminus E}\|{\mathscr{M}}(b)(t)\|_X\, {\rm d}t\\
\le & \sum_{k\in{\mathds{N}}}\int_{(0,\infty)\setminus E} 
\Bigl\|\xi\mapsto \int_{I_k} g_{\xi}(t-s){\mathscr{F}}_x\bigl(b_k(s,\cdot)\bigr)(\xi)\,{\rm d}s\Bigr\|_X\,{\rm d}t\\
\le& \sum_{k\in{\mathds{N}}}\int_{(0,\infty)\setminus \tilde{I}_k} 
\Bigl\|\xi\mapsto \int_{I_k} \bigl(g_{\xi}(t-s)-g_\xi(t-s_k)\bigr){\mathscr{F}}_x\bigl(b_k(s,\cdot)\bigr)(\xi)\,{\rm d}s\Bigr\|_X\,{\rm d}t,
\end{align*}
where $g_{\xi}(s)=|\xi|^2e^{-s|\xi|^2}{\mathds{1}}_{(0,\infty)}(s)$ as in the previous proof.
For all $\xi\in{\mathds{R}}^n$, for all $k\in {\mathds{N}}$ and all $s\in I_k$, we have that
\begin{align*}
\int_{(0,\infty)\setminus \tilde{I}_k}|g_{\xi}(t-s)-g_{\xi}(t-s_k)|\,{\rm d}t
\le& 
\int_{t-s_k>2|s-s_k|}|\xi|^2\bigl|e^{-(t-s)|\xi|^2}-e^{-(t-s_k)|\xi|^2}\bigr| \,{\rm d}t \\
\le&
\int_{t-s_k>2|s-s_k|} \Bigl|\int_{s_k}^s|\xi|^4 e^{-(t-\sigma)|\xi|^2}\,{\rm d}\sigma\Bigr|\,{\rm d}t\\
\le&
\tfrac{4}{e^2}\int_{t-s_k>2|s-s_k|} \bigl|\tfrac{1}{t-s}-\tfrac{1}{t-s_k}\bigr| \,{\rm d}t\\
\le&
\tfrac{4}{e^2}\int_{t'>2|s'|} \bigl|\tfrac{1}{t'-s'}-\tfrac{1}{t'}\bigr| \,{\rm d}t' \ \le \ \tfrac{4\ln 2}{e^2}
\end{align*}
where we made the change of variable $t'=t-s_k$ and $s'=s-s_k$ and the last inequality comes from the exact computation of
$\int_{t'>2|s'|} \bigl|\tfrac{1}{t'-s'}-\tfrac{1}{t'}\bigr| \,{\rm d}t'$ which is equal to $0$ if $s'=0$, $\ln 2$ if $s'>0$ and $\ln\frac{3}{2}$ if $s'<0$.
This implies that
\[
\int_{(0,\infty)\setminus E}\|{\mathscr{M}}(b)(t)\|_X\, {\rm d}t
\le \tfrac{4\ln 2}{e^2}\,\sum_{k\in{\mathds{N}}}\int_{I_k}\|{\mathscr{F}}_x\bigl(b_k(s,\cdot)\|_X\,{\rm d}s
\lesssim \|b\|_{L^1(0,\infty;X)}\lesssim \|f\|_{L^1(0,\infty;X)}.
\]
We are now in position to estimate the size of $\bigl\{t\in (0,\infty)\setminus E; \|{\mathscr{M}}(b)(t)\|_X> \tfrac{\lambda}{2}\bigr\}$:
\begin{equation}
\label{eq:3}
\Bigl|\bigl\{t\in (0,\infty)\setminus E; \|{\mathscr{M}}(b)(t)\|_X> \tfrac{\lambda}{2}\bigr\}\Bigr|\le
\frac{\|{\mathscr{M}}(b)\|_{L^1((0,\infty\setminus E:X)}}{\frac{\lambda}{2}}\lesssim\tfrac{1}{\lambda}\,\|f\|_{L^1(0,\infty;X)}.
\end{equation}
Puting \eqref{eq:1}, \eqref{eq:2} and \eqref{eq:3} in \eqref{eq:0}, and taking the $\sup$ over all $\lambda>0$, we obtain
\[
\|{\mathscr{M}}(f)\|_{L^1_w(0,\infty;X)}
=\sup_{\lambda>0}\lambda \Bigl|\bigl\{t\in (0,\infty); \|{\mathscr{M}}(f)(t)\|_X> \lambda\bigr\}\Bigr|
\lesssim \|f\|_{L^1(0,\infty;X)}.
\]
We started with $f\in L^1(0,\infty;X)\cap L^2(0,\infty;X)$; by density, this is also true for all $f\in L^1(0,\infty;X)$ (by density, we can
extend ${\mathscr{M}}$ on $L^1(0,\infty;X)$).

\noindent
By Proposition~\ref{prop:MRL2}, we know that ${\mathscr{M}}$ is a bounded linear operator from $L^2(0,\infty;X)$ to 
$L^2(0,\infty;X)$. We just proved that ${\mathscr{M}}$ can be extended to a bounded linear operator from $L^1(0,\infty;X)$
to $L^1_w(0,\infty;X)$: we say that ${\mathscr{M}}$ is of {\tt weak type} $(1,1)$. By Marcinkiewicz interpolation theorem
(Theorem~\ref{thm:marcinkiewicz} below), 
we obtain that ${\mathscr{M}}$ is a bounded linear operator
from $L^p(0,\infty;X)$ to $L^p(0,\infty;X)$ for all $1<p<2$. It remains to prove that we have the same result for $p>2$. The dual 
operator of ${\mathscr{M}}$ is given by
\begin{equation}
\label{eq:M'}
{\mathscr{M}}'(g)(s,\cdot)=
{\mathscr{F}}_{\xi}^{-1}\Bigl(\xi\mapsto \int_s^\infty |\xi|^2e^{-(t-s)|\xi|^2}{\mathscr{F}}_x(g(t,\cdot))(\xi)\,{\rm d}t \Bigr), \quad
s>0.
\end{equation}
By duality, the operator ${\mathscr{M}}'$ is bounded in $L^2(0,\infty;X)$ (we identify the dual space of the Hilbert space $X$ with
itself). Using the same procedure as for ${\mathscr{M}}$ (Calder\'on-Zygmund decomposition), the same arguments allow us to 
prove that ${\mathscr{M}}'$ is of weak type $(1,1)$ (bounded from $L^1(0,\infty;X)$ to $L^1_w(0,\infty;X)$). And again by interpolation,
we know that ${\mathscr{M}}'$ is a bounded linear operator from $L^p(0,\infty;X)$ to $L^p(0,\infty;X)$ for all $1<p<2$. By duality, we
obtain that ${\mathscr{M}}$ is a bounded linear operator from $L^p(0,\infty;X)$ to $L^p(0,\infty;X)$ for all $p>2$. And that concludes the\
proof.
\end{proof}

\begin{theorem}[Marcinkiewicz interpolation theorem]
\label{thm:marcinkiewicz}
Let $T$ be a linear operator that is of weak type $(1,1)$ and strong type $(2,2)$: that means that $T$ is bounded from $L^1({\mathds{R}};X)$
to $L^1_w({\mathds{R}};X)$ and bounded from $L^2({\mathds{R}};X)$ to itself. Then $T:L^p({\mathds{R}};X)\to L^p({\mathds{R}};X)$
is bounded for all $1<p\le 2$.
\end{theorem}

\begin{proof}
This theorem is proved in the book ``Interpolation Spaces -- An Introduction" by J\"oran Bergh and J\"orgen L\"ofstr\"om published in
1976 in Grundlehren der mathematischen Wissenschaften (223), Springer Verlag: Theorem~1.3.1, page~9, \cite{BL76}.
\end{proof}

\begin{exercise}
Prove that the operator ${\mathscr{M}}'$ defined by \eqref{eq:M'} is of weak type $(1,1)$.
\end{exercise}

\subsection{The nonlinear heat equation}

In this section (see also \cite{Mo24}), we are interested in solutions of \eqref{eq:chaleur} when $f=f(u)$ depends on the solution $u$:
\begin{equation}
\label{eq:nonlinearheat}
\left\{
\begin{array}{rcl}
\partial_tu-\Delta u&=&f(u)\\
u(0,\cdot)&=&u_0.
\end{array}
\right.
\end{equation}
We will focus on nonlinearities $f$ that are of power type, $i.e.$, $f(u)=|u|^{\nu-1}u$ or $f(u)=|u|^\nu$ for $\nu>1$. In that case,
the equation \eqref{eq:nonlinearheat} is {\tt homogeneous} in the sense that there exists $\alpha\in{\mathds{R}}$ such that
if $u$ is a solution of \eqref{eq:nonlinearheat}, then $u_{\lambda}:(t,x)\mapsto \lambda^\alpha u(\lambda^2 t,\lambda x)$ is also a solution
of \eqref{eq:nonlinearheat} with $u_0$ replaced by $u_{0,\lambda}:x\mapsto \lambda^\alpha u_0(\lambda x)$ for all $\lambda >0$.
In our case $\nu>1$, the homogeneity $\alpha$ must satisfy $\alpha+2=\nu\cdot \alpha$ so that $\alpha=\frac{2}{\nu-1}$.

\begin{definition}
\label{def:critical}
In the case where \eqref{eq:nonlinearheat} is homogeneous, a {\tt critical space} in the Lebesgue scale
$L^q(0,\infty;L^p({\mathds{R}}^n))$ ($1<p,q<\infty$) is a space for which
$\|u_\lambda\|_{L^q(0,\infty;L^p({\mathds{R}}^d))}=\|u\|_{L^q(0,\infty;L^p({\mathds{R}}^d))}$.
\end{definition}

\begin{lemma}
In the case where the non linear heat equation is homogeneous of degree $\alpha$, the associated critical space
is $L^q(0,\infty;L^p({\mathds{R}}^n))$ where $\frac{2}{q}+\frac{n}{p}=\alpha$.
\end{lemma}

\begin{exercise}
\begin{enumerate}
\item
Prove the lemma above!
\item
Determine the critical spaces $L^q(0,\infty;\dot H^s({\mathds{R}}^n))$ for \eqref{eq:nonlinearheat} with homogeneity $\alpha$, $i.e.$,
the relation between $q$ and $s$ to have $\|u_\lambda\|_{L^q(0,\infty;\dot H^s({\mathds{R}}^n))}=\|u\|_{L^q(0,\infty;\dot H^s({\mathds{R}}^n))}$
where $u_\lambda(t,x)=\lambda^\alpha u(\lambda^2 t,\lambda x)$.
\end{enumerate}
\end{exercise}

The method to prove the existence of global mild solutions for small initial data or local mild solutions if the initial condition has no
size restriction relies on the following fixed point theorem.

\begin{theorem}[Fixed Point Theorem]
\label{thm:picard}
Let $Y$ be a Banach space, $a\in Y$ and $F:Y\to Y$ satisfying $F(0)=0$ and there exists $\epsilon>0$ and $M>0$ such
that 
\begin{equation}
\label{eq:estPicard}
\|F(u)-F(v)\|_Y\le M\|u-v\|_Y\bigl(\|u\|_Y^\epsilon+\|v\|_Y^\epsilon\bigr), \quad u,v\in Y.
\end{equation}
Then for each $\delta <\frac{1}{2(2M)^{\frac{1}{\epsilon}}}$, for all $a\in\overline{B}_Y(0,\delta)$, the ball in $Y$ with center $0$ and
radius $\delta$, the map $\Phi:u\mapsto a+F(u)$ admits a fixed point, unique in $\overline{B}_Y(0,2\delta)$.
\end{theorem}

\begin{proof}
The choice of $\delta$ implies that $(2\delta)^\epsilon<\frac{1}{2M}$.

\noindent
First, we prove that if $a\in \overline{B}_Y(0,\delta)$, then $\overline{B}_Y(0,2\delta)$ is stable under the map $\Phi$.  
The estimate on $F$ and the fact that $F(0)=0$ give for $u\in \overline{B}_Y(0,2\delta)$ the estimate
$\|F(u)\|_Y\le M\|u\|^{1+\epsilon}\le M (2\delta)^{1+\epsilon}\le \delta$. With $\|a\|_Y\le \delta$, we
obtain $\|\Phi(u)\|_Y\le 2\delta$.

\noindent
Next, $\Phi:\overline{B}_Y(0,2\delta)\to \overline{B}_Y(0,2\delta)$ is a contraction. We have that
\[
\|\Phi(u)-\Phi(v)\|_Y\le M(\|u\|_Y^\epsilon+\|v\|_Y^\epsilon)\|u-v\|_Y\le 2M(2\delta)^\epsilon\|u-v\|_Y
\] 
with $2M(2\delta)^\epsilon<1$ as established in the previous step.
We conclude by Picard contraction principle that $\Phi$ has a unique fixed point in $\overline{B}_Y(0,2\delta)$.
\end{proof}

\subsubsection{Existence of mild solutions}

We are now in position to prove the existence of solutions of the nonlinear heat equation for $\nu\in(1,\infty)$
\begin{equation}
\label{eq:NLHE}\tag{NLHE}
\left\{
\begin{array}{rcl}
\partial_tu-\Delta u&=&|u|^{\nu-1}u\\
u(0,\cdot)&=&u_0.
\end{array}
\right.
\end{equation}
\begin{theorem}
\label{thm:existenceNLHE}
Let $n\ge \frac{4}{\nu-1}$ and let $s=\frac{n}{2}-\frac{2}{\nu-1}$. Then there exists $\eta>0$ such that for all $u_0\in\dot H^s({\mathds{R}}^n)$
with $\|u_0\|_{\dot H^s}\le \eta$, there exists $u\in {\mathscr{C}}([0,\infty);\dot H^s({\mathds{R}}^n))$ mild solution of \eqref{eq:NLHE}, $i.e.$,
$u$ satisfies
\begin{align*}
&u(t,\cdot)=e^{t\Delta}u_0 + \int_0^t e^{(t-s)\Delta}\bigl(|u(s,\cdot)|^{\nu-1}u(s,\cdot)\bigr)\,{\rm d}s\\
&={\mathscr{F}}_\xi^{-1}\bigl(\xi\mapsto e^{-t|\xi|^2}{\mathscr{F}}_x(u_0)(\xi)\bigr)
+\int_0^t {\mathscr{F}}_\xi^{-1}\Bigl(\xi \mapsto e^{-(t-s)|\xi|^2}{\mathscr{F}}_x\bigl(|u(s,\cdot)|^{\nu-1}u(s,\cdot)\bigr)(\xi)\Bigr)\,{\rm d}s.
\end{align*}
\end{theorem}

\begin{proof}
We will prove this theorem in the special case $n=4$ and $\nu=3$: this is what is closer to the scaling of the Navier-Stokes equations that 
we will investigate in the next section. In that case, $s=1$ and we know that $\dot H^1({\mathds{R}}^4)\hookrightarrow L^4({\mathds{R}}^4)$
by the Sobolev embedding (Theorem~\ref{thm:sobolev}). It is also immediate that \eqref{eq:NLHE} with $\nu=3$ and $n=4$ is homogeneous
with homogeneity $\alpha=1$ and that ${\mathscr{C}}_b([0,\infty);\dot H^1({\mathds{R}}^4))$ is critical for this homogeneity:
for all $u\in {\mathscr{C}}_b([0,\infty);\dot H^1({\mathds{R}}^4))$, for all $\lambda>0$,
\begin{equation}
\label{eq:dotH1critical}
\sup_{t>0}\|(t,x)\mapsto \lambda u(\lambda^2t,\lambda x)\|_{\dot H^1({\mathds{R}}^4)} =
\sup_{t>0}\|u(t,\cdot)\|_{\dot H^1({\mathds{R}}^4)}.
\end{equation}
Here: ${\mathscr{C}}_b$ stands for the space of bounded continuous functions. Since we want to find solutions 
$u\in {\mathscr{C}}_b([0,\infty);\dot H^1({\mathds{R}}^4))$, necessarily, $u_0=u(0,\cdot) \in \dot H^1({\mathds{R}}^4)$.

\noindent
Now, let $a(t):={\mathscr{F}}_\xi^{-1}\bigl(\xi\mapsto e^{-t|\xi|^2}{\mathscr{F}}_x(u_0)\bigr)$, $t\ge 0$ and
\begin{equation}
\label{eq:defF(u)}
F(u)(t)=\int_0^t {\mathscr{F}}_\xi^{-1}\Bigl(\xi \mapsto e^{-(t-s)|\xi|^2}{\mathscr{F}}_x\bigl(|u(s,\cdot)|^2 u(s,\cdot)\bigr)\Bigr)\,{\rm d}s, \quad
t\ge0.
\end{equation}
We need to find a space $Y\subset {\mathscr{C}}_b([0,\infty);\dot H^1({\mathds{R}}^4))$ for which 
the assumptions of Theorem~\ref{thm:picard} hold. In that case, we can find $u\in Y$ such that $u=a+F(u)$ if $\|a\|_Y$ 
is small enough, which is the definition of a mild solution of \eqref{eq:NLHE}. We define
\begin{equation}
\label{eq:Y}
Y:=\bigl\{u\in {\mathscr{C}}_b([0,\infty);\dot H^1({\mathds{R}}^4)); 
t\mapsto t^{\frac{1}{4}}u(t)\in {\mathscr{C}}_b([0,\infty);\dot H^{\frac{3}{2}}({\mathds{R}}^4))\bigr\},
\end{equation}
endowed with the norm
\[
\|u\|_Y:=\sup_{t\ge 0}\bigl\|u(t,\cdot)\bigr\|_{\dot H^1({\mathds{R}}^4)} +
\sup_{t\ge 0} t^{\frac{1}{4}}\bigl\|u(t,\cdot)\bigr\|_{\dot H^{\frac{3}{2}}({\mathds{R}}^4)}, \quad u\in Y.
\]
\begin{enumerate}
\item
$a\in Y$: indeed, let us prove first that $a\in {\mathscr{C}}_b([0,\infty);\dot H^1({\mathds{R}}^4))$: 
\[
|\xi| {\mathscr{F}}_xa(t,\cdot)(\xi)= e^{-t|\xi|^2}\bigl(|\xi|{\mathscr{F}}_xu_0(\xi)\bigr), \quad \xi \in {\mathds{R}}^4,
\] 
and therefore, for all $t\ge 0$
\[
\bigl\|a(t,\cdot)\bigr\|_{\dot H^1({\mathds{R}}^4)}\le \bigl\|\xi\mapsto |\xi|{\mathscr{F}}_xu_0(\xi)\bigr\|_{L^2({\mathds{R}}^4)}
=\|u_0\|_{\dot H^1({\mathds{R}}^4)}.
\]
Continuity in $t$ follows immediately from the continuity of $t\mapsto e^{-t|\xi|^2}$ for all $\xi \in {\mathds{R}}^4$.
To prove that $t\mapsto t^{\frac{1}{4}}a(t)\in {\mathscr{C}}_b([0,\infty);\dot H^{\frac{3}{2}}({\mathds{R}}^4))$, it suffices
to check that 
\[
\sup_{t\ge 0} t^{\frac{1}{4}}\bigl\|a(t,\cdot)\bigr\|_{\dot H^{\frac{3}{2}}({\mathds{R}}^4)}<\infty.
\]
This follows from the fact that $\sup_{t\ge 0} (t|\xi|^2)^{\frac{1}{4}}e^{-t|\xi|^2}=(4e)^{-\frac{1}{4}}<\infty$: write
\[
\bigl\|t^{\frac{1}{4}}a(t,\cdot)\bigr\|_{\dot H^{\frac{3}{2}}({\mathds{R}}^4)}\le \sup_{t\ge 0} \bigl((t|\xi|^2)^{\frac{1}{4}}e^{-t|\xi|^2}\bigr)\,
\bigl\|\xi\mapsto |\xi|{\mathscr{F}}_xu_0(\xi)\bigr\|_{L^2({\mathds{R}}^4)}
\lesssim \|u_0\|_{\dot H^1({\mathds{R}}^4)}.
\]
\item
Next, we prove that for $u\in Y$, $F(u)\in {\mathscr{C}}_b([0,\infty);\dot H^1({\mathds{R}}^4))$.
By the Sobolev embedding $\dot H^{\frac{3}{2}}({\mathds{R}}^4)\hookrightarrow L^8({\mathds{R}}^4)$, we have that
$t\mapsto t^{\frac{1}{2}}|u(t)|^2$ belongs to ${\mathscr{C}}_b([0,\infty);L^4({\mathds{R}}^4))$. If we multiply this by
$u\in Y\subset {\mathscr{C}}_b([0,\infty);\dot H^1({\mathds{R}}^4))\hookrightarrow {\mathscr{C}}_b([0,\infty);L^4({\mathds{R}}^4))$, 
we obtain that $g:t\mapsto t^{\frac{1}{2}}|u(t)|^2u(t)$ belongs to ${\mathscr{C}}_b([0,\infty);L^2({\mathds{R}}^4))$ with the estimate
\[
\sup_{t>0}\bigl\|t^{\frac{1}{2}}|u(t)|^2u(t)\bigr\|_{L^2({\mathds{R}}^4)}\le \|u\|_Y^3.
\]
We now use this estimate on $g$ to estimate $F(u)$ in ${\mathscr{C}}_b([0,\infty);\dot H^1({\mathds{R}}^4))$: for $t\ge 0$,
\begin{align}
F(u)(t)=&{\mathscr{F}}_\xi^{-1}\Bigl(\xi\mapsto \int_0^t e^{-(t-s)|\xi|^2} s^{-\frac{1}{2}} {\mathscr{F}}_xg(s,\cdot)(\xi)\,{\rm d}s\Bigr)
\label{eq:estF(u)}\\
=&{\mathscr{F}}_\xi^{-1}\Bigl(\xi\mapsto \int_0^t \sqrt{(t-s)|\xi|^2}\, e^{-(t-s)|\xi|^2} s^{-\frac{1}{2}} (t-s)^{-\frac{1}{2}}
|\xi|^{-1}{\mathscr{F}}_xg(s,\cdot)(\xi)\,{\rm d}s\Bigr).\nonumber
\end{align}
We have then for all $t\ge0$
\begin{align}
\|F(u)(t)\|_{\dot H^1({\mathds{R}}^4)}\le& \sup_{\sigma >0} \bigl(\sqrt{\sigma |\xi|^2}\, e^{-\sigma|\xi|^2}\bigr) 
\Bigl(\int_0^t \frac{1}{\sqrt{t-s}\sqrt{s}}\,{\rm d}s\Bigr)\sup_{s>0}\|g(s,\cdot)\|_{L^2({\mathds{R}}^4)} \nonumber\\
\lesssim & \Bigl(\int_0^1\frac{1}{\sqrt{1-\tau}\sqrt{\tau}}\,{\rm d}\tau\Bigr)\|u\|_Y^3. \label{eq:estF(u)H1}
\end{align}
\item
The next step is to prove that for $u\in Y$, $t\mapsto t^{\frac{1}{4}}F(u)(t)\in {\mathscr{C}}_b([0,\infty);\dot H^{\frac{3}{2}}({\mathds{R}}^4))$.
Using the formula \eqref{eq:estF(u)}, 
\[
F(u)(t)={\mathscr{F}}_\xi^{-1}\Bigl(\xi\mapsto \int_0^t \bigl((t-s)|\xi|^2\bigr)^{\frac{3}{4}}\, e^{-(t-s)|\xi|^2} s^{-\frac{1}{2}} (t-s)^{-\frac{3}{4}}
|\xi|^{-\frac{3}{2}}{\mathscr{F}}_xg(s,\cdot)(\xi)\,{\rm d}s\Bigr)
\]
and proceeding as in \eqref{eq:estF(u)H1} since 
\[
\sup_{\sigma >0} \bigl((\sigma |\xi|^2)^{\frac{3}{4}}\, e^{-\sigma|\xi|^2}\bigr) <\infty,
\]
we have for all $t>0$
\begin{align*}
\|F(u)(t)\|_{\dot H^{\frac{3}{2}}({\mathds{R}}^4)}\le& \sup_{\sigma >0} \bigl((\sigma |\xi|^2)^{\frac{3}{4}}\, e^{-\sigma|\xi|^2}\bigr) 
\Bigl(\int_0^t \frac{1}{(t-s)^{\frac{3}{4}}\sqrt{s}}\,{\rm d}s\Bigr)\sup_{s>0}\|g(s,\cdot)\|_{L^2({\mathds{R}}^4)}\\
\lesssim & t^{-\frac{1}{4}}\Bigl(\int_0^1\frac{1}{(1-\tau)^{\frac{3}{4}}\sqrt{\tau}}\,{\rm d}\tau\Bigr)\|u\|_Y^3, 
\end{align*}
which proves the claim.
\item
Finally, we need to prove the estimate on $\|F(u)-F(v)\|_Y$ as in Theorem~\ref{thm:picard}. It follows immediately from the points 2 and 3
above and the fact that 
\begin{equation}
\label{eq:fLip}
\bigl||u|^2u-|v|^2v\bigr|\le |u|^2|u-v|+\bigl|u|^2-|v|^2\bigr||v|\le \tfrac{3}{2}\,|u-v|(|u|^2+|v|^2)
\end{equation}
which proves \eqref{eq:estPicard} with $\epsilon=2$.
\item
We apply now Theorem~\ref{thm:picard} and prove that there is $\delta>0$ such that if $\|u_0\|_{\dot H^1({\mathds{R}}^4)}\le \delta$,
then there exists a unique mild solution of \eqref{eq:NLHE} with $n=4$ and $\nu=3$
in the space $Y\subset {\mathscr{C}}_b([0,\infty);\dot H^1({\mathds{R}}^4))$ defined by \eqref{eq:Y}.\qedhere
\end{enumerate}
\end{proof}

\begin{exercise}
Prove \eqref{eq:dotH1critical}
\end{exercise}

\subsubsection{Uniqueness of mild solutions}

The existence result we just proved shows also uniqueness, but in the space $Y$. The question is now: do we have uniqueness in 
$X:={\mathscr{C}}_b([0,\infty);\dot H^1({\mathds{R}}^4))$? This is where we will use the maximal regularity result (Theorem~\ref{thm:MRLp}).

\begin{theorem}
\label{thm:NLHEuniqueness}
Let $0<T\le \infty$.
Assume that $u$ and $v$ are two mild solutions of \eqref{eq:NLHE} ($i.e.$ satisfying $u=a+F(u)$ and $v=a+F(v)$ where $a$ and $F$ have 
been defined in the proof of Theorem~\ref{thm:existenceNLHE}) with $n=4$ and $\nu=3$ for the same initial value 
$u_0\in \dot H^1({\mathds{R}}^4)$ in ${\mathscr{C}}_b([0,T);\dot H^1({\mathds{R}}^4))$. Then $u=v$ on $[0,T)$.
\end{theorem}

\begin{proof}
Let $u,v$ two mild solutions of \eqref{eq:NLHE} with $n=4$ and $\nu=3$: $u=a+F(u)$ and $v=a+F(v)$ so that
\begin{align*}
|u(t)-v(t)|=&|F(u)(t)-F(v)(t)|\\
=&\Bigl|\int_0^te^{(t-s)\Delta}\bigl(|u(s)|^2u(s)-|v(s)|^2v(s)\bigr) \,{\rm d}s\Bigr|.
\end{align*}
To prove that $u=v$ on $[0,\tau]$ for a $\tau >0$, it suffices to prove that $u-v$ satisfies 
$\|u-v\|_{L^p(0,\tau;\dot H^1({\mathds{R}}^4))}\le \frac{1}{2}\|u-v\|_{L^p(0,\tau;\dot H^1({\mathds{R}}^4))}$ for a $p\in(1,\infty)$.
Let $\varepsilon>0$. Denote by $a_\varepsilon$ the function 
$a_\varepsilon(t):=e^{(t+\varepsilon)\Delta}u_0={\mathscr{F}}^{-1}_\xi\bigl(\xi\mapsto e^{-(t+\varepsilon) |\xi|^2}{\mathscr{F}}_x(u_0)(\xi)\bigr)$.
It is easy to see that $a_\varepsilon\in {\mathscr{C}}_b([0,T);\dot H^{\frac{3}{2}}({\mathds{R}}^4))$ with norm
$\sup_{t\in[0,T)}\|a_\varepsilon(t)\|_{\dot H^{\frac{3}{2}}}\le \varepsilon^{-\frac{1}{4}}\|u_0\|_{\dot H^1}$ and
\[
\sup_{t\ge 0}\|a(t)-a_\varepsilon(t)\|_{\dot H^1}\le \|u_0-u_{0,\varepsilon}\|_{\dot H^1}\xrightarrow[\varepsilon\to 0]{}0
\] 
where $u_{0,\varepsilon}={\mathscr{F}}^{-1}_\xi\bigl(\xi\mapsto e^{-\varepsilon |\xi|^2}{\mathscr{F}}_x(u_0)(\xi)\bigr)$.
We also have
\[
\sup_{0\le t\le \tau}\|u(t)-a(t)\|_{\dot H^1}\xrightarrow[\tau\to 0]{}0\quad \mbox{and}\quad 
\sup_{0\le t\le \tau}\|v(t)-a(t)\|_{\dot H^1}\xrightarrow[\tau\to 0]{}0.
\] 
Moreover,
\begin{align}
|u|^2u-|v|^2v=&|u|^2(u-v) +v (u-v)(u+v)\nonumber\\
&\hspace*{-2cm}=(u-v)\bigl(u(u-a)+(u+v)(v-a)+a(2u+v)\bigr)\nonumber\\
&\hspace*{-2cm}=(u-v)\bigl(u(u-a)+(u+v)(v-a)+(2u+v)(a-a_\varepsilon)+(2u+v)a_\varepsilon\bigr) \label{eq:u3-v3}
\end{align}
We define $\Phi:L^p(0,T;\dot H^1({\mathds{R}}^4))\times X\times X\to L^p(0,T;\dot H^1({\mathds{R}}^4))$ (see below) by
\[
\Phi(u,v,w)(t) := {\mathscr{F}}_\xi^{-1}\Bigl(\xi\mapsto \int_0^te^{-(t-s)|\xi|^2}{\mathscr{F}}_x\bigl(u(s)v(s)w(s)\bigr)(\xi) \,{\rm d}s\Bigr),
\quad t\ge0
\]
so that $F(u)=\Phi(u,u,u)$ for all $u\in X$. It is immediate thanks to \eqref{eq:u3-v3} that
\begin{align}
F(u)-F(v) = \,&\Phi(u-v,u,u-a)+\Phi(u-v,u+v,v-a) \nonumber\\
&+\Phi(u-v,2u+v,a-a_\varepsilon)+\Phi(u-v, 2u+v, a_\varepsilon). \label{eq:F(u)-F(v)}
\end{align}
The proof that $\Phi$ maps $L^p(0,T;\dot H^1({\mathds{R}}^4))\times X\times X$ to $L^p(0,T;\dot H^1({\mathds{R}}^4))$ comes
from its expression using the maximal regularity operator ${\mathscr{M}}$ from \eqref{eq:opMR}. 
  
\noindent
Indeed, for 
$u\in L^p(0,T;\dot H^1({\mathds{R}}^4))$ and $v,w\in X$, we have that $uvw$ belongs to the space
$L^p(0,T; L^{\frac{4}{3}}({\mathds{R}}^4))$ which embeds into
$L^p(0,T; \dot H^{-1}({\mathds{R}}^4))$ and therefore,
\[
\Phi(u,v,w) ={\mathscr{M}} \psi(u,v,w), \quad \psi(u,v,w)={\mathscr{F}}_\xi^{-1}\bigl(\xi\mapsto \tfrac{1}{|\xi|^2}{\mathscr{F}}_x(uvw)(\xi)\bigr).
\]
By Sobolev embeddings ($\dot H^1\hookrightarrow L^4$ and $L^{\frac{4}{3}}\hookrightarrow \dot H^{-1}$ in dimension 4), we have
\[
\bigl\|\psi(u,v,w)\bigr\|_{L^p(0,T;\dot H^1({\mathds{R}}^4))}\lesssim \|u\|_{L^p(0,T;\dot H^1({\mathds{R}}^4))}\|v\|_{X}\|w\|_{X}.
\]
We also have, denoting by $W$ the space ${\mathscr{C}}_b([0,T);\dot H^{\frac{3}{2}}({\mathds{R}}^4))$, 
\[
uvw\in L^p(0,T;L^{\frac{8}{5}}({\mathds{R}}^4))\hookrightarrow L^p(0,T;\dot H^{-\frac{1}{2}}({\mathds{R}}^4))
\]
with
\[
\|{\mathscr{F}}_\xi^{-1}\bigl(\xi\mapsto \tfrac{1}{|\xi|^{\frac{3}{2}}}{\mathscr{F}}_x(uvw)(\xi)\bigr)\|_{L^p(0,T;\dot H^{1}({\mathds{R}}^4))}
\lesssim \|u\|_{L^p(0,T;\dot H^1({\mathds{R}}^4))}\|v\|_X\|w\|_W
\]
and then
\[
\Phi(u,v,w)(t)={\mathscr{F}}_\xi^{-1}\Bigl(\xi\mapsto \int_0^t |\xi|^{\frac{3}{2}}e^{-(t-s)|\xi|^2}
\tfrac{1}{|\xi|^{\frac{3}{2}}}{\mathscr{F}}_x\bigl(u(s)v(s)w(s)\bigr)(\xi) \,{\rm d}s\Bigr), \quad t\in (0,T)
\]
can be estimated in $L^p(0,\tau;\dot H^1({\mathds{R}}^4))$ for every $\tau \in (0,T)$ as follows:
\[
\bigl\|\Phi(u,v,w)\bigr\|_{L^p(0,\tau;\dot H^1({\mathds{R}}^4))}\le C \tau^{\frac{1}{4}}\|u\|_{L^p(0,\tau;\dot H^1({\mathds{R}}^4))}\|v\|_X\|w\|_W.
\]
Using now the decomposition \eqref{eq:F(u)-F(v)}, we have for all $\tau\in (0,T)$
\begin{align*}
\|u-v\|_{L^p(0,\tau;\dot H^1({\mathds{R}}^4))}\le  &C \,\|u-v\|_{L^p(0,\tau;\dot H^1({\mathds{R}}^4))}\cdot\\
&\Bigl( \|u\|_X\|u-a\|_{L^\infty(0,\tau;\dot H^1({\mathds{R}}^4))}  +(\|u\|_X+\|v\|_X)\|v-a\|_{L^\infty(0,\tau;\dot H^1({\mathds{R}}^4))} \\
&+ (2\|u\|_X+\|v\|_X)\|a-a_\varepsilon\|_X+\tau^{\frac{1}{4}}(2\|u\|_X+\|v\|_X)\|a_\varepsilon\|_W \Bigr)
\end{align*}
We first choose $\varepsilon>0$ such that $\|a-a_\varepsilon\|_X\le \frac{1}{6C(2\|u\|_X+\|v\|_X)}$. Then let $\tau_0\in(0,T)$ such that
$\tau_0^{\frac{1}{4}}\|a_\varepsilon\|_W\le \frac{1}{6C(2\|u\|_X+\|v\|_X)}$. Finally, let $\tau\in (0,\tau_0]$ such that 
$\|u-a\|_{L^\infty(0,\tau;\dot H^1({\mathds{R}}^4))}\le \frac{1}{6C(2\|u\|_X+\|v\|_X)}$ and 
$\|v-a\|_{L^\infty(0,\tau;\dot H^1({\mathds{R}}^4))}\le \frac{1}{6C(2\|u\|_X+\|v\|_X)}$. With these choices, we have that
\[
\|u-v\|_{L^p(0,\tau;\dot H^1({\mathds{R}}^4))}\le \tfrac{1}{2}\,\|u-v\|_{L^p(0,\tau;\dot H^1({\mathds{R}}^4))},
\]
so that $u=v$ on $[0,\tau)$.

\noindent
Now, the set $E:=\{0\le t <T; u(t)=v(t)\}$ is non empty ($0\in E$), closed (due to the continuity of $u-v$) and open (that's just what we proved,
up to a translation changing $u_0$ at $t=0$ with $u(t_0)$ at $t=t_0$ if necessary). Since $[0,T)$ is connected, this means that $E=[0,T)$.
\end{proof}

\section{The Navier-Stokes equations}

The (simplified) equations that we consider in this section are the $n$-dimensional Navier-Stokes system ($n\ge 3$)
without external force for which we set the constant density and the viscosity to be equal to $1$
\begin{equation}
\label{eq:NS}\tag{NS}
\begin{array}{rclcl}
\partial_tu-\Delta u+\nabla\pi+(u\cdot\nabla)u&=&0&\mbox{ in }&(0,T)\times{\mathds{R}}^n
\\[4pt]
{\rm div}\,u&=&0&\mbox{ in }&(0,T)\times{\mathds{R}}^n
\end{array}
\end{equation}
where $0<T\le\infty$, the velocity 
$u:(0,T)\times{\mathds{R}}^n\to{\mathds{R}}^n$ and the pressure 
$\pi:(0,T)\times{\mathds{R}}^n\to{\mathds{R}}$ being the unknown. Here, $\partial_t$ denotes the 
partial derivative with respect to time $t\in(0,T)$, 
$\Delta=\sum_{j=1}^n\partial_j^2$ denotes the Laplacian in ${\mathds{R}}^n$ and 
$u\cdot\nabla:=\sum_{j=1}^n u_j\partial_j$ whenever this makes sense.

\begin{notation}
Recall that we denote by $e^{t\Delta}$, for $t\ge 0$ the operator acting on $\dot H^s({\mathds{R}}^n)$
($s\in{\mathds{R}}$) as follows:
\[
e^{t\Delta}g:={\mathscr{F}}_\xi^{-1}\bigl(\xi\mapsto e^{-t|\xi|^2}{\mathscr{F}}_x(g)(\xi)\bigr),\quad g\in\dot H^s({\mathds{R}}^n).
\]
\end{notation}

\begin{exercise}
\label{ex:heatanalytic}
Let $s\in{\mathds{R}}$.
\begin{enumerate}
\item
Prove that for all $g\in \dot H^s({\mathds{R}}^n)$, $t>0$ and all $\alpha\ge 0$, we have that $(-\Delta)^{\alpha}e^{t\Delta}g$ belongs
to $\dot H^s({\mathds{R}}^n)$ and 
\[
\bigl\|(-\Delta)^{\alpha}e^{t\Delta}g\bigr\|_{\dot H^s}\lesssim t^{-\alpha}\|g\|_{\dot H^s}.
\]
\item
Prove that for all $g\in \dot H^s({\mathds{R}}^n)$, $t>0$ and all $\alpha\ge 0$, we have that $e^{t\Delta}g$ belongs
to $\dot H^{s+2\alpha}({\mathds{R}}^n)$ and 
\[
\bigl\|e^{t\Delta}g\bigr\|_{\dot H^{s+2\alpha}}\lesssim t^{-\alpha}\|g\|_{\dot H^s}.
\]
\end{enumerate}
\end{exercise}

\begin{lemma}
The system \eqref{eq:NS} (with $T=\infty$) is homogeneous of order $1$ for $u$ and $2$ for the pressure $\pi$. More precisely, if $(u,\pi)$ is 
a solution of \eqref{eq:NS} for the initial data $u_0$, then for all $\lambda >0$, $(u_\lambda,\pi_\lambda)$ defined by
\[
u_\lambda(t,x)=\lambda u(\lambda^2t,\lambda x), \quad \pi_\lambda(t,x)=\lambda^2\pi(\lambda^2t,\lambda x), \quad t>0, \ x\in{\mathds{R}}^n
\]
is a solution of \eqref{eq:NS} for the initial data ${\mathds{R}}^n \ni x\mapsto u_{0,\lambda}(x):=\lambda u_0(\lambda x)$.
\end{lemma}

\begin{proof}
Exercise!
\end{proof}

\begin{proposition}
\label{prop:criticalNS}
The space ${\mathscr{C}}_b([0,\infty);\dot H^s({\mathds{R}}^n;{\mathds{R}}^n))$ is critical for \eqref{eq:NS} if $s=\tfrac{n}{2}-1$.
\end{proposition}

\begin{proof}
According to Definition~\ref{def:critical}, we need to find $s\in{\mathds{R}}$ for which 
\[
\sup_{t>0}\|u(t,\cdot)\|_{\dot H^s}=\sup_{t>0}\|u_\lambda(t,\cdot)\|_{\dot H^s}, \quad \mbox{for all }\lambda>0
\mbox{ and all }u\in \dot H^s({\mathds{R}}^n,{\mathds{R}}^n).
\]
We have, thanks to Theorem~\ref{thm:Fourier1}, for all $\lambda>0$ and all $\xi\in{\mathds{R}}^n$, and all $t>0$,
${\mathscr{F}}_x(u_\lambda(t,\cdot)(\xi)=\lambda^{-n+1}{\mathscr{F}}_x\bigl(u_\lambda(\lambda t,\cdot)\bigr)\bigl(\tfrac{\xi}{\lambda}\bigr)$,
so that 
\[
\bigl\|u_\lambda(t,\cdot)\bigr\|_{\dot H^s}=\lambda^{-\frac{n}{2}+1+s}\bigl\|u(\lambda t,\cdot)\bigr\|_{\dot H^s}, \quad t>0.
\]
This implies that the two norms are equal for all $\lambda>0$ (after taking the $\sup_{t>0}$) if, and only if $-\frac{n}{2}+1+s=0$, or 
as claimed, $s=\frac{n}{2}-1$.
\end{proof}

\subsection{Mild solutions}
\label{subsec:mild}

Formally, we can eliminate the pressure term in \eqref{eq:NS} by taking the divergence of the
first equation. Since ${\rm div}\,u=0$, 
\[
(u\cdot\nabla)u=\sum_{j=1}^nu_j\partial_ju=\sum_{j=1}^n\partial_j(u_ju)=:\nabla\cdot(u\otimes u)
\]
and we obtain
\[
\Delta \pi=-{\rm div}\,\bigl(\nabla\cdot(u\otimes u)\bigr)=-\sum_{j,k=1}^n\partial_ju_k\partial_ku_j,
\]
so that 
\[
\nabla\pi=\nabla(-\Delta)^{-1}{\rm div}\,\bigl(\nabla(u\otimes u)\bigr)
=\nabla {\mathscr{F}}_\xi^{-1}\bigl(\xi\mapsto\sum_{j,k=1}^n\tfrac{i\xi_ji\xi_k}{|\xi|^2}{\mathscr{F}}_x(u_ju_k)(\xi)\bigr).
\] 
The system \eqref{eq:NS} with initial condition $u_0:{\mathds{R}}^n\to{\mathds{R}}^n$ 
satisfying ${\rm div}\,u_0=0$ then becomes
\begin{equation}
\label{eq:NS1}
\partial_tu-\Delta u=-{\mathbb{P}}\bigl(\nabla\cdot(u\otimes u)\bigr),\quad u(0)=u_0,
\end{equation}
where 
${\mathbb{P}}:={\rm Id}+\nabla(-\Delta)^{-1}{\rm div}\,={\rm Id}+{\mathcal{R}}\otimes{\mathcal{R}}$ 
is the {\tt Leray projection} operator (or {\tt Helmholtz projection}) defined on tempered 
distributions ${\mathscr{S}}'({\mathds{R}}^n,{\mathds{R}}^n)$ (recall that ${\mathcal{R}}$ denotes the vector
of Riesz transforms; ${\mathcal{R}}_j=\partial_j(-\Delta)^{-1/2}$). Its Fourier symbol is given by 
$\bigl(\delta_{jk}-\frac{\xi_j\xi_k}{|\xi^2|}\bigr)_{1\le j,k\le n}$ (where $\delta_{jk}$ denotes the
{\tt Kronecker} symbol, is equal to~$1$ if $j=k$ and $0$ otherwise). From that, we obtain
the following result.

\begin{proposition}
\label{prop:LerayR3}
The Leray projection operator maps tempered distributions to divergence free tempered distributions. 
It maps also Sobolev $\dot H^s$ vector fields to $\dot H^s$ divergence free vector fields for all
$s\in{\mathds{R}}$.
\end{proposition}

\begin{proof}
The fact that ${\mathbb{P}}$ maps $\dot H^s$ to $\dot H^s$ boundedly comes from the fact
that Riesz transforms are bounded in $\dot H^s$.
\end{proof}

For the right hand-side of \eqref{eq:NS1},
$-{\mathbb{P}}\bigl(\nabla\cdot(u\otimes u)\bigr)$, to make sense, it is sufficient that $u\otimes u$ is
a tempered distribution in the space variable, which is achieved if for almost every $t\in (0,T)$,
there is a $p\ge 2$ such that $u(t,\cdot)\in L^p$. In that case,
\eqref{eq:NS1} has merely the form of the heat equation with initial condition $u_0$ and
forcing term $-{\mathbb{P}}\bigl(\nabla\cdot(u\otimes u)\bigr)$. We are now in position to
give the definition of the solutions we will consider. 

\begin{definition}
\label{def:defmild}
A {\tt mild solution} of \eqref{eq:NS} on $(0,T)\times{\mathds{R}}^n$ with initial condition $u_0$ 
satisfying ${\rm div}\,u_0=0$
is a vector field $u:[0,T)\times{\mathds{R}}^n\to{\mathds{R}}^n$ solution of the integral equation
\begin{equation}
\label{eq:intNS}
u(t)=e^{t\Delta}u_0-\int_0^te^{(t-s)\Delta}{\mathbb{P}}\bigl(\nabla\cdot(u\otimes u)(s)\bigr)\,{\rm d}s,
\quad 0\le t\le T,
\end{equation}
obtained as a fixed point (in an appropriate function space) of the transform
\[
v\mapsto \Bigl(t\mapsto 
e^{t\Delta}u_0-\int_0^te^{(t-s)\Delta}{\mathbb{P}}\bigl(\nabla\cdot(v\otimes v)(s)\bigr)\,{\rm d}s\Bigr).
\]
\end{definition}

We will apply our fixed point theorem (Theorem~\ref{thm:picard}) with $a(t)=e^{t\Delta}u_0$ and
$F(u)(t)=B(u,u)(t)$ where $B$ is defined by 
\begin{equation}
\label{eq:defB}
B:(u,v)\mapsto \Bigl(t\mapsto 
-\int_0^te^{(t-s)\Delta}{\mathbb{P}}\bigl(\nabla\cdot(u\otimes v)(s)\bigr)\,{\rm d}s\Bigr)
\end{equation}
in a Banach space ${\mathscr{E}}_T$ of vector fields defined on $(0,T)\times{\mathds{R}}^n$ to be
determined. 

\subsection{Existence in $\dot H^{n/2-1}$}
\label{subsec:mildSobolev}

We will construct mild solutions of \eqref{eq:NS} in $(0,T)$ ($0<T\le \infty$) in the spirit of Fujita and Kato (1964) 
\cite{FK64} in the
critical space 
\[
{\mathscr{E}}_T=
\bigl\{u\in {\mathscr{C}}((0,T);\dot H^{\frac{n-1}{2}}); 
\|u\|_{{\mathscr{E}}_T}:=\sup_{t\in(0,T)}\|t^{1/4}u(t)\|_{\dot H^{\frac{n-1}{2}}}<\infty\bigr\}  
\]
where
\[
\dot H^s:=\bigl\{f\in {\mathscr{S}}';\xi\mapsto |\xi|^s\hat f(\xi)\in  L^2\bigr\},
\]
endowed with the semi norm $\|f\|_{\dot H^s}=\|\xi\mapsto |\xi|^s\hat f(\xi)\|_{L^2}$.
Thanks to Lemma~\ref{prop:criticalNS}, the space ${\mathscr{C}}([0,\infty);\dot H^{\frac{n}{2}-1})$
is a critical space for \eqref{eq:NS}.

\begin{lemma}
The space ${\mathscr{E}}_\infty$ is critical for \eqref{eq:NS}.
\end{lemma}

\begin{proof}
We have to prove that for $\lambda>0$ and $u\in {\mathscr{E}}_\infty$ the following equality holds: 
$\|u_\lambda\|_{{\mathscr{E}}_\infty}=\|u\|_{{\mathscr{E}}_\infty}$.
Remark first that, taking the Fourier transform of $u_\lambda$ in the $x$ variable, we obtain:
\[
{\mathscr{F}}(u_\lambda)(t,\cdot))(\xi)=\lambda^{-n+1}{\mathscr{F}} (u(\lambda^2t, \cdot))\bigl(\tfrac{\xi}{\lambda}\bigr),
 \quad t>0, \xi\in{\mathds{R}}^n.
\]
Therefore, 
\[
\|u_\lambda(t,\cdot)\|_{H^{\frac{n-1}{2}}}^2 =\lambda \|u(\lambda^2t, \cdot)\|_{H^{\frac{n-1}{2}}}^2.
\]
This gives
\[
\|u_\lambda\|_{{\mathscr{E}}_\infty}=\sup_{t>0}t^{1/4}\|u_\lambda(t,\cdot)\|_{H^{\frac{n-1}{2}}}
=\sup_{\tau>0}\tau^{1/4}\|u(\tau,\cdot)\|_{H^{\frac{n-1}{2}}}=\|u\|_{{\mathscr{E}}_\infty},
\]
using the notation $\tau=\lambda^2t$.
\end{proof}

\begin{theorem}[Existence of mild solutions]
\label{thm:mildHs}
For all $u_0\in\dot H^{n/2-1}$ with ${\rm div}\,u_0=0$, there exist $T^*>0$ and 
$u\in{\mathscr{E}}_{T^*}$ mild solution of \eqref{eq:NS} on $(0,T^*)\times{\mathds{R}}^n$ 
with initial condition $u_0$. Moreover, $u\in{\mathscr{C}}([0,T^*);\dot H^{n/2-1})$.
\end{theorem}

\begin{proof}
Let $a:t\mapsto e^{t\Delta}u_0$ and $B$ the bilinear transform defined by \eqref{eq:defB}
\[
B(u,v)(t):=-\int_0^t e^{(t-s)\Delta}{\mathbb{P}}\bigl(\nabla\cdot(u\otimes v)\bigr)(s)\,{\rm d}s,
\quad t\ge 0.
\]
We will show that for all $T>0$, $a\in {\mathscr{E}}_{T}$ and $B$ is bounded from
${\mathscr{E}}_{T}\times{\mathscr{E}}_{T}$ to ${\mathscr{E}}_{T}$, uniformly in $T>0$. 
It will then remain to prove that there exists $T^*>0$ such that $\|a\|_{{\mathscr{E}}_{T^*}}$
is small enough, so that we can apply Picard's contraction principle.

\medskip

\noindent
{\tt Step 1:} $a\in {\mathscr{E}}_{\infty}$. 

Remark that $a$ satisfies $\partial_t a-\Delta a=0$,
$a(0)=u_0$. Taking the Fourier transform in the space variable and solving a ordinary differential
equation, we arrive at ${\mathscr{F}}_x(a(t,\cdot)(\xi)=e^{-t|\xi|^2}{\mathscr{F}}_x(u_0)(\xi)$, $t\ge 0$, $\xi\in{\mathbb{R}}^n$. 
The following estimate holds for all $t>0$
\begin{align}
\label{eq:est_a}
\|t^{1/4}a(t,\cdot)\|_{\dot H^{\frac{n-1}{2}}}
&\le \|\xi\mapsto t^{1/4}|\xi|^{\frac{n-1}{2}}e^{-t|\xi|^2}|{\mathscr{F}}_x(u_0)(\xi)|\|_{L^2}
\nonumber\\
&=\bigl\|\xi\mapsto (t|\xi|^2)^{1/4}e^{-t|\xi|^2}\bigl(|\xi|^{n/2-1}|{\mathscr{F}}_x(u_0)(\xi)|\bigr)\bigr\|_{L^2}
\nonumber\\
&\le \sup_{s>0}\bigl(s^{1/4}e^{-s}\bigr)\,\|u_0\|_{\dot H^{n/2-1}},
\end{align}
where we have set $s=t|\xi|^2$ in the last inequality.

Clearly, for all $T>0$, ${\mathscr{E}}_T\subset {\mathscr{E}}_\infty$ (via restriction). Therefore,
$a\in{\mathscr{E}}_T$ for all $T>0$.

\medskip

\noindent
{\tt Step 2:}
Let $\varepsilon>0$. For every $u_0\in \dot H^{n/2-1}$, there exists $T^*>0$ such that 
$\|a\|_{{\mathscr{E}}_{T^*}}< \varepsilon$.

Indeed, the Schwartz space ${\mathscr{S}}$ is dense in $\dot H^{n/2-1}$. We can then find, 
for any $\delta >0$, $u_\delta\in{\mathscr{S}}$ so that 
$\|u_\delta-u_0\|_{\dot H^{n/2-1}}<\delta$. Let 
$a_\delta(t):=e^{t\Delta}u_\delta$, $t\ge 0$.
The estimate \eqref{eq:est_a} applied to $a-a_\delta$ shows that for all $T>0$,
\[
\|a-a_\delta\|_{{\mathscr{E}}_T}\le C\delta,
\]
the constant $C$ being independent from $T>0$.
Now, since $u_\delta\in{\mathscr{S}}\subset \dot H^{\frac{n-1}{2}}$ and thanks to the fact 
that the heat semigroup is a contraction on $\dot H^{\frac{n-1}{2}}$, we have 
$\|a_\delta(t,\cdot)\|_{\dot H^{\frac{n-1}{2}}}\le \|u_\delta\|_{\dot H^{\frac{n-1}{2}}}$. Therefore, for all $T>0$, we obtain
\[
\|a_\delta\|_{{\mathscr{E}}_T}=\sup_{t\in(0,T)}\bigl(t^{1/4}\|a_{\delta}(t,\cdot)\|_{\dot H^{\frac{n-1}{2}}}\bigr)
\le T^{1/4} \|u_\delta\|_{\dot H^{\frac{n-1}{2}}}.
\]
Choosing $\delta>0$ small enough so that $C\delta<\varepsilon/2$ and $T^*>0$ such
that $T^{*1/4} \|u_\delta\|_{\dot H^{\frac{n-1}{2}}}<\varepsilon/2$, we proved that
\[
\|a\|_{{\mathscr{E}}_{T^*}}\le \|a-a_\delta\|_{{\mathscr{E}}_{T^*}}+\|a_\delta\|_{{\mathscr{E}}_{T^*}}
\le C\delta+T^{*1/4} \|u_\delta\|_{\dot H^{\frac{n-1}{2}}}<\varepsilon,
\]
which is the claim.

\medskip

\noindent
{\tt Step 3:} $B$ is bounded from ${\mathscr{E}}_{T}\times{\mathscr{E}}_{T}$ to 
${\mathscr{E}}_{T}$, uniformly in $T>0$.

We use the result of Lemma~\ref{lem:multSobolev}.
Let $u,v\in{\mathscr{E}}_{T}$. Then 
$s\mapsto s^{1/2}u(s)\otimes v(s)\in{\mathscr{C}}([0,T);\dot H^{n/2-1})$. It follows that
$s\mapsto s^{1/2}{\mathbb{P}}\bigl(\nabla\cdot (u\otimes v)\bigr)(s)\in {\mathscr{C}}([0,T);\dot H^{n/2-2})$ 
since ${\mathbb{P}}$ is bounded on $\dot H^{n/2-2}$.
Therefore, by the mapping properties of the (inverse of the) Laplacian, we arrive at
\begin{equation}
\label{eq:cont-uv}
s\mapsto s^{1/2}(-\Delta)^{-3/4}{\mathbb{P}}\bigl(\nabla\cdot (u\otimes v)\bigr)(s)\in {\mathscr{C}}([0,T);\dot H^{\frac{n-1}{2}}).
\end{equation}
Now, taking into account the property proved in Exercise~\ref{ex:heatanalytic} with $s=0$, we have
that $\|(-\Delta)^\alpha e^{\sigma\Delta}\|_{L^2\to L^2}\le C_\alpha\,\sigma^{-\alpha}$. Thanks to
the boundedness of Riesz transforms, we obtain for $0<s<t<T$,
\begin{align}
\label{eq:est-Bbis}
\bigl\|e^{(t-s)\Delta}{\mathbb{P}}\bigl(\nabla\cdot(u\otimes v)\bigr)(s)\bigr\|_{\dot H^{\frac{n-1}{2}}}
\le &\bigl\|\bigl((-\Delta)^{3/4}e^{(t-s)\Delta}
(-\Delta)^{-3/4}{\mathbb{P}}\bigl(\nabla\cdot(u\otimes v)\bigr)(s)\bigr\|_{\dot H^{\frac{n-1}{2}}}
\nonumber\\
\le &C\,(t-s)^{-3/4}s^{-1/2}\|u\|_{{\mathscr{E}}_{T}}\|v\|_{{\mathscr{E}}_{T}},
\end{align}
for all $u,v\in {\mathscr{E}}_{T}$, the constant $C$ does not depend on $T$, but merely on
the properties of the Laplacian in ${\mathbb{R}}^n$. This shows the following estimate
\begin{align*}
t^{1/4}\|B(u,v)(t)\|_{\dot H^{\frac{n-1}{2}}}
&\le C\,\Bigl(t^{1/4}\int_0^t (t-s)^{-3/4}s^{-1/2}\,{\rm d}s\Bigr)
\|u\|_{{\mathscr{E}}_{T}}\|v\|_{{\mathscr{E}}_{T}}\\
&\le C\,\Bigl(\int_0^1(1-\sigma)^{-3/4}\sigma^{-1/2}\,{\rm d}\sigma\Bigr)
\|u\|_{{\mathscr{E}}_{T}}\|v\|_{{\mathscr{E}}_{T}},
\end{align*}
where we have set $\sigma=\frac{s}{t}$ in the last inequality. This proves the claim, the 
continuity of $t\mapsto t^{1/4} B(u,v)(t)$ from $[0,T)$ to $\dot H^1$ coming from the 
continuity property in \eqref{eq:cont-uv} preserved by convolution with the $L^1$ kernel
$\sigma\mapsto (-\Delta)^{3/4}e^{\sigma\Delta}$.

\medskip

\noindent
{\tt Step 4:} 
We can now apply Picard's contraction principle to prove that for all $u_0\in \dot H^{1/2}$,
there exists $T^*>0$ and $u\in {\mathscr{E}}_{T^*}$ mild solution of \eqref{eq:NS}; $T^*$ depends
on $u_0$ and is determined as in Step~2 with 
$\varepsilon=\frac{1}{4\| B\|_{{\mathscr{E}}_{T^*}\times{\mathscr{E}}_{T^*}\to{\mathscr{E}}_{T^*}}}$.
Remark that if $\|u_0\|_{\dot H^{\frac{n-1}{2}}}$ is small enough, then $T^*=\infty$, i.e., there is a global 
mild solution of \eqref{eq:NS}.

\medskip

\noindent
{\tt Step 5:} 
The mild solution $u\in {\mathscr{E}}_{T^*}$ from Step~4 is continuous from $[0,T^*)$ 
to $\dot H^{n/2-1}$.

The solution $u$ satisfies $u=a+ B(u,u)$ on $[0,T^*)$.
Using \eqref{eq:cont-uv} for $u=v$ as in \eqref{eq:est-Bbis}, we have for $0<s<t<T^*$
\begin{align*}
\bigl\|e^{(t-s)\Delta}{\mathbb{P}}\bigl(\nabla\cdot(u\otimes u)\bigr)(s)\bigr\|_{\dot H^{n/2-1}}
=& \bigl\|(-\Delta)^{1/2}e^{(t-s)\Delta}(-\Delta)^{-1/2}{\mathbb{P}}\bigl(\nabla\cdot(u\otimes u)\bigr)(s)\bigr\|_{\dot H^{n/2-1}}\\
\le &\bigl\|(-\Delta)^{1/2}e^{(t-s)\Delta}\bigr\|_{\dot H^{n/2-1}\to \dot H^{n/2-1}}
\bigl\|(u\otimes u)\bigr)(s)\bigr\|_{\dot H^{n/2-1}}\\
\le &C\,(t-s)^{-1/2}s^{-1/2}\|u\|_{{\mathscr{E}}_{T^*}}^2.
\end{align*}
This proves that
\begin{align*}
\|B(u,u)(t)\|_{\dot H^{1/2}}
\le& C\,\Bigl(\int_0^t(t-s)^{-1/2}s^{-1/2}\,{\rm d}s\Bigr)\|u\|_{{\mathscr{E}}_{T^*}}^2
\\
\le& C\,\Bigl(\int_0^1(1-\sigma)^{-1/2}\sigma^{-1/2}\,{\rm d}\sigma\Bigr)\|u\|_{{\mathscr{E}}_{T^*}}^2.
\end{align*}
Moreover, $t\mapsto B(u,u)(t)$ is continuous from $[0,T^*)$ to $\dot H^{n/2-1}$: this
comes from the continuity property \eqref{eq:cont-uv} for $u=v$ and $T=T^*$ preserved by
the convolution with the $L^1$ kernel $\sigma\mapsto (-\Delta)^{1/2}e^{\sigma\Delta}$.
It remains to prove that $a\in{\mathscr{C}}([0,T^*),\dot H^{n/2-1})$, which is a direct consequence
of the strong continuity of the heat semigroup on $\dot H^{n/2-1}$.
\end{proof}

\subsection{Uniqueness of mild solutions ($n\ge 3$)}
\label{subsec:uniqueness}

One may ask now whether any mild solution of \eqref{eq:NS} in ${\mathscr{C}}([0,T^*),Y)$ with
$Y=\dot H^{n/2-1}$ is unique. The uniqueness coming from the Picard's 
contraction principle occurs in the space ${\mathscr{E}}_{T^*}$ which is smaller than
${\mathscr{C}}([0,T^*),Y)$. This question has been a longstanding problem since the
existence of mild solutions has been proved (see the very documented review on the paper by
Furioli, Lemari\'e-Rieusset and Terraneo, published in 2000,
written by Marco Cannone in MathSciNet). We propose here a proof via 
maximal regularity (coming from \cite{Mo99}), following the same lines as what we did for the nonlinear heat equation
(see Theorem~\ref{thm:NLHEuniqueness}). 

\begin{lemma}
\label{lem:propB}
\begin{enumerate}
\item
There exists a constant $C>0$ such that for all $T>0$, for all $u\in L^4(0,T;\dot H^{\frac{n}{2}-1})$ and
all $v\in L^\infty(0,T;\dot H^{\frac{n}{2}-1})$, we have
\begin{equation}
\label{eq:BLnLn}
\|B(u,v)\|_{L^4(0,T;\dot H^{\frac{n}{2}-1})}+\|B(v,u)\|_{L^4(0,T;\dot H^{\frac{n}{2}-1})}\le 
C\|u\|_{L^4(0,T;\dot H^{\frac{n}{2}-1})}\|v\|_{L^\infty(0,T;\dot H^{\frac{n}{2}-1})}.
\end{equation}
\item
There exists a constant $C>0$ such that for all $T>0$, for all $u\in L^4(0,T;\dot H^{\frac{n}{2}-1})$ and
all $v\in L^4(0,T;\dot H^{\frac{n-1}{2}})$, we have
\begin{equation}
\label{eq:BLnL2n}
\|B(u,v)\|_{L^4(0,T;\dot H^{\frac{n}{2}-1})}+\|B(v,u)\|_{L^4(0,T;\dot H^{\frac{n}{2}-1})}
\le C\|u\|_{L^4(0,T;\dot H^{\frac{n}{2}-1})}\|v\|_{L^4(0,T;\dot H^{\frac{n-1}{2}})}.
\end{equation}
\end{enumerate}
\end{lemma}

\begin{proof}
To prove \eqref{eq:BLnLn} and \eqref{eq:BLnL2n}, we will use maximal regularity as in 
Theorem~\ref{thm:MRLp}. We will also need a variant of Lemma~\ref{lem:multSobolev} as follows:

Let $n\ge 2$. There exists $C>0$ such that for all $f,g\in \dot H^{\frac{n}{2}-1}$ we have that $fg\in\dot H^{\frac{n}{2}-2}$ and
the following estimate holds
\[
\|fg\|_{\dot H^{\frac{n}{2}-2}}\le C\ \|f\|_{\dot H^{\frac{n}{2}-1}}\|g\|_{\dot H^{\frac{n}{2}-1}}
\]
(the cases $n=3$ and $n=4$ are easy, using Sobolev embeddings; for $n\ge 5$, we apply the same method as in the proof
of Lemma~\ref{lem:multSobolev}).
\begin{enumerate}
\item
For $u\in L^4(0,T;\dot H^{\frac{n}{2}-1})$ and $v\in L^\infty(0,T;\dot H^{\frac{n}{2}-1})$, we have
\[
u\otimes v, v\otimes u\in L^4(0,T;\dot H^{\frac{n}{2}-2})
\]
with the accompanying estimate
\[
\bigl\|u\otimes v\bigr\|_{L^4(0,T;\dot H^{\frac{n}{2}-2})}+\bigl\|v\otimes u\bigr\|_{L^4(0,T;\dot H^{\frac{n}{2}-2})}
\le C \|u\|_{L^4(0,T;\dot H^{\frac{n}{2}-1})}\|v\|_{L^\infty(0,T;\dot H^{\frac{n}{2}-1})}
\]
where $C>0$ is a constant independent of $T>0$. Therefore, thanks to the mapping properties
of the (inverse of the) Laplacian and Sobolev embeddings, we have
\[
f:=(-\Delta)^{-1}{\mathbb{P}}\bigl(\nabla\cdot(u\otimes v)\bigr)\in L^4(0,T;\dot H^{\frac{n}{2}-1})
\]
and similarly
\[
g:=(-\Delta)^{-1}{\mathbb{P}}\bigl(\nabla\cdot(v\otimes u)\bigr)
\in L^4(0,T;\dot H^{\frac{n}{2}-1}).
\]
Now, using the maximal regularity property of the heat semigroup in $L^4(\dot H^{\frac{n}{2}-1})$,
we have that
\[
t\mapsto \int_0^t(-\Delta) e^{(t-s)\Delta}f(s)\,{\rm d}s=B(u,v)(t)\in L^4(0,T;\dot H^{\frac{n}{2}-1})
\]
and
\[
t\mapsto \int_0^t(-\Delta) e^{(t-s)\Delta}g(s)\,{\rm d}s=B(v,u)(t)\in L^4(0,T;\dot H^{\frac{n}{2}-1})
\]
and \eqref{eq:BLnLn} follows.
\item
For $u\in L^4(0,T;\dot H^{\frac{n}{2}-1})$ and $v\in L^4(0,T;\dot H^{\frac{n-1}{2}})$ we have that
\[
u\otimes v, v\otimes u\in L^2(0,T;\dot H^{\frac{n-3}{2}})
\] 
with the accompanying estimate
\[
\bigl\|u\otimes v\bigr\|_{L^2(0,T;\dot H^{\frac{n-3}{2}})}+\bigl\|v\otimes u\bigr\|_{L^2(0,T;\dot H^{\frac{n-3}{2}})}
\le C \|u\|_{L^4(0,T;\dot H^{\frac{n}{2}-1})}\|v\|_{L^4(0,T;\dot H^{\frac{n-1}{2}})}
\]
where $C>0$ is a constant independent of $T>0$. Therefore, we have that
\[
f:={\mathbb{P}}\bigl(\nabla\cdot(u\otimes v)\bigr),
g:={\mathbb{P}}\bigl(\nabla\cdot(v\otimes u)\bigr) \in L^2(0,T;\dot H^{\frac{n-5}{2}}).
\]
Now, using Proposition~\ref{prop:mixedMaxReg} for $\alpha=\frac{1}{4}$ and $s=\frac{n-5}{2}$,
we have that
\[
t\mapsto \int_0^t e^{(t-s)\Delta} f(s)\,{\rm d}s=B(u,v)(t)
\in \dot H^{1/4}(0,T;\dot H^{\frac{n-5}{2}+\frac{3}{2}})
\hookrightarrow L^4(0,T;\dot H^{\frac{n}{2}-1})
\]
by the Sobolev embeddings  $H^{1/4}\hookrightarrow L^4$ in dimension $1$ and
$\frac{n-5}{2}+\frac{3}{2}=\frac{n}{2}-1$.
And similarly with $g$ instead of $f$. The estimate \eqref{eq:BLnL2n} follows.\qedhere
\end{enumerate}
\end{proof}

\begin{theorem}[Uniqueness of mild solutions]
\label{thm:uniquenessLn}
Assume that there exist two mild solutions $u$ and $v$ in ${\mathscr{C}}([0,T),\dot H^{\frac{n}{2}-1})$
with the same initial condition $u_0\in \dot H^{\frac{n}{2}-1}$. Then $u=v$ on $[0,T)$.
\end{theorem}

\begin{proof}
By definition of mild solutions, $u$ and $v$ satisfy
$u=a+B(u,u)$ and $v=a+B(v,v)$, where $a(t)=e^{t\Delta}u_0$, $t\ge 0$,  
and $B$ was defined by \eqref{eq:defB}.
Let now $w:=u-v$; since $B$ is bilinear, $w$ satisfies
\[
w=B(w,u)+B(v,w)=B(w,u-a)+B(v-a,w)+B(w,a)+B(a,w).
\]
The idea is to estimate $w$ in $L^4(0,\tau;\dot H^{\frac{n}{2}-1})$ for $0<\tau\le T$.
Using \eqref{eq:BLnLn} and \eqref{eq:BLnL2n} we have
\begin{align}
\label{eq:est-w}
&\|w\|_{L^4(0,\tau;\dot H^{\frac{n}{2}-1})}
\\
\le& C\|w\|_{L^4(0,\tau;\dot H^{\frac{n}{2}-1})}\bigl(\|u-a\|_{L^\infty(0,\tau;\dot H^{\frac{n}{2}-1})}+\|v-a\|_{L^\infty(0,\tau;\dot H^{\frac{n}{2}-1})}
+\|a\|_{L^4(0,\tau;\dot H^{\frac{n-1}{2}})}\bigr).
\nonumber
\end{align}
We next prove that 
\[
\|u-a\|_{L^\infty(0,\tau;\dot H^{\frac{n}{2}-1})}+\|v-a\|_{L^\infty(0,\tau;\dot H^{\frac{n}{2}-1})}
+\|a\|_{L^4(0,\tau;\dot H^{\frac{n-1}{2}})}\xrightarrow[\tau\to 0]{} 0.
\]
Since $u,v,a\in{\mathscr{C}}([0,T),\dot H^{\frac{n}{2}-1})$ and $u(0)=v(0)=u_0=a(0)$, it is clear that
\[
\|u-a\|_{L^\infty(0,\tau;\dot H^{\frac{n}{2}-1})}+\|v-a\|_{L^\infty(0,\tau;\dot H^{\frac{n}{2}-1})}
\xrightarrow[\tau\to 0]{} 0
\]
The fact that $\|a\|_{L^4(0,\tau;\dot H^{\frac{n-1}{2}})}\xrightarrow[\tau\to 0]{} 0$ follows from the fact
that $a\in L^4(0,\infty;\dot H^{\frac{n-1}{2}})$. Indeed, since $u_0\in \dot H^{\frac{n}{2}-1}$, the properties of the heat semigroup
show that $a\in L^\infty(0,\infty;\dot H^{\frac{n}{2}-1})$ and $a\in L^2(0,\infty;\dot H^{\frac{n}{2}})$ (see next lemma).

\begin{lemma}
\label{lem:L2H1}
Let $s\in{\mathds{R}}$ and $u_0\in\dot H^s({\mathds{R}}^n)$. Then the function $a:t\mapsto e^{t\Delta}u_0$ belongs to 
$L^2(0,\infty;\dot H^{s+1}({\mathds{R}}^n)$.
\end{lemma}

\begin{proof}
Taking the Fourier transform of $a(t)$ in the $x$ variable, we have for all $t>0$ by definition
\[
{\mathscr{F}}_x\bigl(a(t)\bigr)(\xi)=e^{-t|\xi|^2}{\mathscr{F}}_x(u_0)(\xi),\quad \xi\in{\mathds{R}}^n.
\]
Therefore,
\begin{align*}
&\bigl\|a\bigr\|_{L^2(0,\infty;\dot H^{s+1})}^2= \int_0^\infty \|a(t)\|_{\dot H^{s+1}}^2\, {\rm d}t\\
=&
\int_0^\infty \int_{{\mathds{R}}^n}\bigl||\xi|^{s+1}e^{-t|\xi|^2}{\mathscr{F}}_x(u_0)(\xi)\bigr|^2\, {\rm dm}_n(\xi)\, {\rm d}t\\
=&\int_{{\mathds{R}}^n}\Bigl(\int_0^\infty e^{-2t|\xi|^2}, {\rm d}t \Bigr) \bigl||\xi|^{s+1}{\mathscr{F}}_x(u_0)(\xi)\bigr|^2\, {\rm dm}_n(\xi),
\quad \mbox{\footnotesize by Fubini for positive functions}\\
=&\int_{{\mathds{R}}^n}\tfrac{1}{2|\xi|^2}\,|\xi|^{2(s+1)}|{\mathscr{F}}_x(u_0)(\xi)|^2\, {\rm dm}_n(\xi),
\quad \mbox{\footnotesize computing the integral in $t$}\\
=&\tfrac{1}{2}\,\int_{{\mathds{R}}^n}|\xi|^{2s}|{\mathscr{F}}_x(u_0)(\xi)|^2\, {\rm dm}_n(\xi)\\
=&\tfrac{1}{2}\,\|u_0\|_{\dot H^s}^2. \qedhere
\end{align*}
\end{proof}

\noindent
{\it Proof of Theorem~\ref{thm:uniquenessLn}, continued.}
To prove that $a\in L^4(0,\infty;\dot H^{\frac{n-1}{2}})$, write
\begin{align*}
\|a\|_{L^4(0,\infty;\dot H^{\frac{n-1}{2}})}^4=& \int_0^\infty \|a(t)\|_{\dot H^{\frac{n-1}{2}}}^4\,{\rm d}t\\
=&\int_0^\infty \Bigl(\int_{{\mathds{R}}^n}e^{-2t|\xi|^2} |\xi|^{n-1}|{\mathscr{F}}_x(u_0)(\xi)|^2\,{\rm dm}_n(\xi)\Bigr)^2{\rm d}t\\
\le & \int_0^\infty\Bigl(\int_{{\mathds{R}}^n}e^{-4t|\xi|^2}|\xi|^n|{\mathscr{F}}_x(u_0)(\xi)|^2\,{\rm dm}_n(\xi)\Bigr)\\
& \qquad \cdot \Bigl(\int_{{\mathds{R}}^n}|\xi|^{n-2}|{\mathscr{F}}_x(u_0)(\xi)|^2\,{\rm dm}_n(\xi)\Bigr)\,{\rm d}t
\quad\mbox{\footnotesize by Cauchy-Schwarz inequality}\\
= & \bigl\|t\mapsto a(2t)\bigr\|_{L^2(0,\infty;\dot H^{\frac{n}{2}})}^2 \|u_0\|_{\dot H^{\frac{n}{2}-1}}^2 \\
=& \tfrac{1}{4}\,\|u_0\|_{\dot H^{\frac{n}{2}-1}}^4\qquad\qquad \mbox{\footnotesize thanks to Lemma~\ref{lem:L2H1} with $s=\tfrac{n}{2}-1$.}
\end{align*}
This implies that $\|a\|_{L^4(0,\tau;\dot H^{\frac{n-1}{2}})}\xrightarrow[\tau\to 0]{}0$.
Therefore, we can find $\tau_0\in(0,T]$ such that
\[
\|u-a\|_{L^\infty(0,\tau_0;\dot H^{\frac{n}{2}-1})}+\|v-a\|_{L^\infty(0,\tau_0;\dot H^{\frac{n}{2}-1})}
+\|a\|_{L^4(0,\tau_0;\dot H^{\frac{n-1}{2}})}\le\textstyle{\frac{1}{2C}}
\]
and \eqref{eq:est-w} becomes 
\[
\|w\|_{L^4(0,\tau_0;\dot H^{\frac{n}{2}-1})}\le\textstyle{\frac{1}{2}}\,\|w\|_{L^4(0,\tau_0;\dot H^{\frac{n}{2}-1})},
\]
which shows that $w=0$ $a.e.$ on $[0,\tau_0)$ and since $w$ is continuous, $w=0$ on $[0,\tau_0)$. This proves, up to translation,
that the set $E:=\{t\in [0,T);w(t)=0\}$ is open. By continuity of $w$, this set $E$ is also closed. Since it is non empty ($0\in E$), 
by connectedness of $[0,T)$, we have that $E=[0,T)$.
\end{proof}


\section{Mid-semester exam}

A printed version of the script (Chapter 1) was allowed.

\subsection{Questions}

\begin{enumerate}
\item
{\bf Sobolev embedding.} Let $n \ge 1$ and $s>\frac{n}{2}$. Let $f\in H^s({\mathds{R}}^n)$. 
\begin{enumerate}
\item
Recall in terms of ${\mathscr{F}}(f)$ what it means that $f\in H^s({\mathds{R}}^n)$.
\item
Prove that ${\mathscr{F}}(f) \in L^1({\mathds{R}}^n)$.
\item
Prove that $f\in {\mathscr{C}}_0({\mathds{R}}^n)$.
\end{enumerate}
\item
{\bf Fourier transform.}
Let $n\ge 1$. Let $f\in L^1({\mathds{R}}^n)\setminus L^2({\mathds{R}}^n)$. 
\begin{enumerate}
\item
Give an example of such a function $f$.
\item
What is the result in the lecture asserting that $g:={\mathscr{F}}(f)\in {\mathscr{C}}_0({\mathds{R}}^n)$?
\item
Why doesn't $g$ belong to $L^2({\mathds{R}}^n)$?
\item
Does $g$ belong to $L^1({\mathds{R}}^n)$?
\item
Do we have ${\mathscr{F}}^{-1}(g)=f$? In what sense?
\end{enumerate}
\end{enumerate}

\subsection{Answers}

\begin{enumerate}
\item
{\bf Sobolev embedding.} Let $n \ge 1$ and $s>\frac{n}{2}$. Let $f\in H^s({\mathds{R}}^n)$.
\begin{enumerate}
\item
According to Definition~\ref{def:sobolev}, $f\in H^s({\mathds{R}}^n)$ means that $f\in {\mathscr{S}}'({\mathds{R}}^n)$ 
is such that the function $\xi\mapsto \bigl(1+|\xi|^2\bigr)^{\frac{s}{2}}{\mathscr{F}}(f)(\xi)$
belongs to $L^2({\mathds{R}}^n)$.
\item
Since $s>\frac{n}{2}$, we have that $\xi\mapsto \frac{1}{(1+|\xi|^2)^{\frac{s}{2}}}$ belongs to $L^2({\mathds{R}}^n)$. Then
we have that 
\[
{\mathscr{F}}(f):\xi\mapsto \frac{1}{(1+|\xi|^2)^{\frac{s}{2}}} \, \bigl(1+|\xi|^2\bigr)^{\frac{s}{2}}{\mathscr{F}}(f)(\xi)
\]
belongs to $L^1({\mathds{R}}^n)$ as a product of two functions of $L^2({\mathds{R}}^n)$.
\item
Since $g={\mathscr{F}}(f)$ belongs to $L^1({\mathds{R}}^n$, we have that ${\mathscr{F}}(g)$ belongs to ${\mathscr{C}}_0({\mathds{R}}^n)$
by the Riemann-Lebesgue lemma (Theorem~\ref{thm:RL}). Moreover, as in \eqref{eq:inverseF}, we have that ${\mathscr{F}}(g)(-\xi)=f(\xi)$
for all $\xi\in{\mathds{R}}^n$. This proves that $f\in {\mathscr{C}}_0({\mathds{R}}^n)$.
\end{enumerate}
\item
{\bf Fourier transform.}
Let $n\ge 1$. Let $f\in L^1({\mathds{R}}^n)\setminus L^2({\mathds{R}}^n)$. 
\begin{enumerate}
\item
For instance, we can define $f$ on ${\mathds{R}}^n$ as follows: $f(x)=|x|^{-\frac{n}{2}}{\mathds{1}}_{(0,1)}(|x|)$, $x\in{\mathds{R}}^n$:
the singularity at $0$ is integrable 
\[
\int_{{\mathds{R}}^n}|f(x)|\,{\rm d}x=
\int_{|x|<1}|x|^{-\frac{n}{2}}\,{\rm d}x =\int_{{\mathds{S}}^{n-1}}\Bigl(\int_0^1 r^{n-1-\frac{n}{2}}\,{\rm d}r\Bigr)\,{\rm d}\omega <\infty
\]
whereas 
\[
\int_{{\mathds{R}}^n}|f(x)|^2\,{\rm d}x=
\int_{|x|<1}|x|^{-n}\,{\rm d}x =\int_{{\mathds{S}}^{n-1}}\Bigl(\int_0^1 r^{n-1-n}\,{\rm d}r\Bigr)\,{\rm d}\omega =\infty.
\]
\item
Since $f\in L^1({\mathds{R}}^n)$, Theorem~\ref{thm:RL} (Riemann-Lebesgue Lemma) asserts that $g={\mathscr{F}}(f)$
belongs to ${\mathscr{C}}_0({\mathds{R}}^n)$.
\item
If $g$ would belong to $L^2({\mathds{R}}^n)$, so would ${\mathscr{F}}^{-1}(g)$ since ${\mathscr{F}}$ is invertible on $L^2({\mathds{R}}^n)$.
And therefore, $f$ would belong to $L^2({\mathds{R}}^n)$ which is not the case by our choice of $f$.
\item
If $g$ would belong to $L^1({\mathds{R}}^n)$, by H\"older inequality ($\|g\|_r\le \|g\|_p^{1-\theta}\|g\|_q^\theta$ where $1\le p,q,r\le \infty$, 
$\theta\in[0,1]$ and
$\frac{1}{r}=\frac{1-\theta}{p}+\frac{\theta}{q}$ with $r=2$, $p=1$, $q=\infty$ and $\theta=\frac{1}{2}$), $g$ would also belong to 
$L^2({\mathds{R}}^n)$ which contradicts the previous result.
\item
We proved that $g$ does not belong to either $L^1({\mathds{R}}^n)$ nor $L^2({\mathds{R}}^n)$. But as a Fourier transform of 
an $L^1$ function, $g$ belongs to ${\mathscr{C}}_0({\mathds{R}}^n)\subset {\mathscr{S}}'({\mathds{R}}^n)$ and ${\mathscr{F}}$
is invertible on ${\mathscr{S}}'({\mathds{R}}^n)$. We then have that $f={\mathscr{F}}^{-1}(g)$ in ${\mathscr{S}}'({\mathds{R}}^n)$.
\end{enumerate}
\end{enumerate}

\begin{remark}
The fact that ${\mathscr{F}}$ maps $L^1({\mathds{R}}^n)$ to ${\mathscr{C}}_0({\mathds{R}}^n)$ doesn't mean that it is a bijection
between these two spaces. It only means that the range of the Fourier transform on $L^1({\mathds{R}}^n)$ is a subset of
${\mathscr{C}}_0({\mathds{R}}^n)$. Take for instance, in dimension 1, the odd function $\varphi$ defined as follows:
$\varphi(\xi)=\frac{1}{\ln \xi}$ for $\xi\ge e$ and $\varphi(\xi)=\frac{\xi}{e}$ for $0\le \xi\le e$. It is clear that $\varphi\in {\mathscr{C}}_0({\mathds{R}})$
(draw a picture!), but if $\varphi$ were the Fourier transform of an $L^1$ function, say $f$, we would have for all $0<\varepsilon<M<\infty$:
\begin{align*}
\int_\varepsilon^M\frac{\varphi(\xi)}{\xi}\,{\rm d}\xi=&
\int_\varepsilon^M\frac{\hat{f}(\xi)}{\xi}\,{\rm d}\xi
= \int_\varepsilon^M \tfrac{1}{\xi} \Bigl(\tfrac{1}{\sqrt{2\pi}}\int_{\mathds{R}}f(x)e^{-ix\xi}\,{\rm d}x\Bigr)\,{\rm d}\xi\\
=&\tfrac{i}{\sqrt{2\pi}} \int_\varepsilon^M\ \tfrac{1}{\xi} \Bigl(\int_{\mathds{R}} f(x) \sin(x\xi)\,{\rm d}x \Bigr)\,{\rm d}\xi 
\quad\mbox{\scriptsize{using the fact that $\varphi$ is odd: $\varphi(\xi)=\tfrac{1}{2}\bigl(\varphi(\xi)-\varphi(-\xi)\bigr)$}}\\
=& \tfrac{i}{\sqrt{2\pi}} \int_{\mathds{R}} f(x) \Bigl(\int_{\varepsilon}^{M}\tfrac{\sin(x\xi)}{\xi}\,{\rm d}\xi \Bigr)\,{\rm d}x 
\quad\mbox{\scriptsize{using Fubini}}\\
=& \tfrac{i}{\sqrt{2\pi}} \int_{\mathds{R}} f(x) \Bigl(\int_{x\varepsilon}^{xM}\tfrac{\sin(\eta)}{\eta}\,{\rm d}\eta \Bigr)\,{\rm d}x 
\quad\mbox{\scriptsize{with the change of variable $\eta=x\xi$}}.
\end{align*}
Now, for all $0<\delta<K<\infty$, it is not difficult to show that 
\[
\Bigl|\int_{\delta}^{K}\tfrac{\sin(\eta)}{\eta}\,{\rm d}\eta \Bigr|\le 4 \quad \mbox{and} \quad
\Bigl|\int_{-\delta}^{-K}\tfrac{\sin(\eta)}{\eta}\,{\rm d}\eta \Bigr|\le 4.
\]
The second inequality comes from the first one, since $\eta\mapsto \frac{\sin(\eta)}{\eta}$ is an even function. To prove the first 
estimate, if $\delta<1<K$ (the other cases follow directly from that one), we have
\begin{align*}
\Bigl|\int_{\delta}^{K}\tfrac{\sin(\eta)}{\eta}\,{\rm d}\eta \Bigr|
\le &\Bigl|\int_{\delta}^1\tfrac{\sin(\eta)}{\eta}\,{\rm d}\eta \Bigr| +\Bigl| \int_1^K\tfrac{\sin(\eta)}{\eta}\,{\rm d}\eta \Bigr|\\
\le &1 +\cos 1 +\bigl|\tfrac{\cos K}{K}\bigr| +\int_1^K\tfrac{1}{\eta^2}\,{\rm d}\eta \le 4
\end{align*}
where we estimated $\frac{\sin (\eta)}{\eta}$ by 1 for $\eta\in[\delta,1]$ and we integrated by parts the integral between $1$ and
$K$ (with the estimate $|\cos(\eta)|\le 1$ in the last part). Therefore, we would obtain for all $0<\varepsilon<M<\infty$
\[
\Bigl|\int_\varepsilon^M\frac{\varphi(\xi)}{\xi}\,{\rm d}\xi \Bigr|\le \tfrac{4}{\sqrt{2\pi}}\|f\|_1.
\]
This is clearly false for our function $\varphi$
since $\int_e^M\frac{\varphi(\xi)}{\xi}\,{\rm d}\xi=\int_e^M\frac{1}{\xi\ln(\xi)}\,{\rm d}\xi$ is unbounded as $M$ goes to $\infty$.
So there is no $L^1$ function of which $\varphi$ is the Fourier transform although $\varphi\in {\mathscr{C}}_0({\mathds{R}})$.
This proves that the image of $L^1({\mathds{R}})$ by the Fourier transform is a strict subspace of ${\mathscr{C}}_0({\mathds{R}})$.
\end{remark}

\section{Final exams}

A printed version of the script was allowed.

\subsection{First subject}

This was much too long for a 2 hours exam. The students were not supposed to finish to get the maximal grade.

\begin{exercise}
As in the script, we use the notation 
 \[
 e^{t\Delta}g={\mathscr{F}}_\xi^{-1}\bigl(\xi\mapsto e^{-t|\xi|^2}{\mathscr{F}}_x(g)(\xi)\bigr)
 \] 
 for $t>0$ and $g\in \dot H^s({\mathds{R}}^n)$ ($s\in {\mathds{R}}$).
\begin{enumerate}
\item 
Prove that for all $g\in \dot H^s({\mathds{R}}^n)$ and all $\alpha>0$, the function $(-\Delta)^\alpha e^{t\Delta}g$ belongs
to $\dot H^s({\mathds{R}}^n)$ for all $t>0$ and
\[
\|(-\Delta)^\alpha e^{t\Delta}g \|_{\dot H^s}\le \alpha^\alpha e^{-\alpha} \, t^{-\alpha}\,\|g\|_{\dot H^s}.
\]
\item 
Let $g\in L^2({\mathds{R}}^n)$. Prove that $(t,x)\mapsto \bigl(e^{t\Delta}g\bigr)(x)$ 
belongs to the space $L^2(0,\infty;\dot H^1({\mathds{R}}^n))$.

{\small\sl Hint: On the Fourier side in the $x$ variable, use Fubini Theorem.}
\item 
Prove that $(t,x)\mapsto \partial_t\bigl(t\mapsto e^{t\Delta}g\bigr)(x)$ belongs to $L^2(0,\infty;\dot H^{-1}({\mathds{R}}^n))$.
\end{enumerate}
\end{exercise}

\begin{exercise}
The purpose of this exercise is to prove that mild solutions of the quadratic heat equation in dimension $5$ are unique. So from now on, $n=5$.
\begin{enumerate}
\item 
Using Sobolev embeddings, prove that if $f\in \dot H^{\frac{1}{2}}({\mathds{R}}^5)$ and $g\in \dot H^2({\mathds{R}}^5)$, then the product
$fg$ belongs to $L^2({\mathds{R}}^5)$ and that the following estimate holds:
\begin{equation}
\label{eq:uvL^2}
\|fg\|_{L^2}\le c\,\|f\|_{\dot H^{\frac{1}{2}}}\|g\|_{\dot H^2},
\end{equation}
where $c$ is a positive constant independent from $f$ and $g$.
\item
\begin{enumerate}
\item 
Prove that if $f,g\in \dot H^{\frac{1}{2}}({\mathds{R}}^5)$, then $fg \in L^{\frac{5}{4}}({\mathds{R}}^5)$,
using Sobolev embeddings.
\item 
By duality (if $\dot H^s\hookrightarrow L^p$ then $L^{p'}\hookrightarrow \dot H^{-s}$)
and thanks to the properties of the Laplace operator, prove that
$(-\Delta)^{-1}(fg)\in \dot H^{\frac{1}{2}}({\mathds{R}}^5)$
and that the following estimate holds:
\begin{equation}
\label{eq:uvH-3/2}
\|(-\Delta)^{-1}(fg)\|_{\dot H^{\frac{1}{2}}}\le \gamma \, \|f\|_{\dot H^{\frac{1}{2}}}\|g\|_{\dot H^{\frac{1}{2}}},
\end{equation}
where $\gamma$ is a positive constant independent from $f$ and $g$.
\end{enumerate}
\item
We denote by $B$ the bilinear form defined by 
\[
(u,v) \mapsto \Bigl(t\mapsto B(u,v)(t):=\int_0^te^{(t-s)\Delta}\bigl(u(s)v(s)\bigr)\,{\rm d}s\Bigr).
\]
\begin{enumerate}
\item
Prove that $B$ is bounded from $L^2(0,\infty;\dot H^{\frac{1}{2}}({\mathds{R}}^5))\times {\mathscr{C}}_b([0,\infty);\dot H^2({\mathds{R}}^5))$
to the space $L^2(0,\infty;\dot H^{\frac{1}{2}}({\mathds{R}}^5))$ and that the following estimate holds for all 
$u\in L^2(0,\infty;\dot H^{\frac{1}{2}}({\mathds{R}}^5))$ and
$v\in {\mathscr{C}}_b([0,\infty);\dot H^2({\mathds{R}}^5))$:
\[
\|B(u,v)\|_{L^2(0,\tau;\dot H^{\frac{1}{2}}({\mathds{R}}^5))}\le 
C\, \tau^{\frac{3}{4}}\,\|u\|_{L^2(0,\tau;\dot H^{\frac{1}{2}}({\mathds{R}}^5))} \sup_{0<t<\tau}\|v(t)\|_{\dot H^2}, \quad \mbox{for all }\tau>0
\]
where $C$ is a constant independent from $\tau$, $u$ and $v$.

{\small\sl Hint: use equation \eqref{eq:uvL^2} and Young inequality.}

\item 
Prove that $B$ is bounded from $L^2(0,\infty;\dot H^{\frac{1}{2}}({\mathds{R}}^5))\times {\mathscr{C}}_b([0,\infty);\dot H^{\frac{1}{2}}({\mathds{R}}^5))$
to the space $L^2(0,\infty;\dot H^{\frac{1}{2}}({\mathds{R}}^5))$ and that for all $u,v\in L^2(0,\infty;\dot H^{\frac{1}{2}}({\mathds{R}}^5))$
the following estimate holds:
\[
\|B(u,v)\|_{L^2(0,\tau;\dot H^{\frac{1}{2}}({\mathds{R}}^5))}\le 
C'\, \|u\|_{L^2(0,\tau;\dot H^{\frac{1}{2}}({\mathds{R}}^5))} \sup_{0<t<\tau}\|v(t)\|_{\dot H^{\frac{1}{2}}}, \quad \mbox{for all }\tau>0
\]
where $C'$ is a constant independent from $\tau$, $u$ and $v$.

{\small\sl Hint: use equation \eqref{eq:uvH-3/2} and the maximal regularity operator ${\mathscr{M}}$ from the script.}
\end{enumerate}

\item 
Assume that $u_0\in \dot H^{\frac{1}{2}}({\mathds{R}}^5)$ and denote by $a$ the function $a(t):=e^{t\Delta}u_0$, $t\ge 0$.
\begin{enumerate}
\item 
Prove that $a\in {\mathscr{C}}_b([0,\infty);\dot H^{\frac{1}{2}}({\mathds{R}}^5))$.
\item 
Let $\varepsilon>0$ and define $a_\varepsilon:t\mapsto e^{(t+\varepsilon)\Delta}u_0$.
\begin{enumerate}
\item 
Prove that $a_\varepsilon\in {\mathscr{C}}_b([0,\infty);\dot H^{2}({\mathds{R}}^5))$.
\item 
Prove that $\sup_{t>0}\|a_\varepsilon(t)-a(t)\|_{\dot H^{\frac{1}{2}}}\xrightarrow[\varepsilon\to 0]{}0$.
\end{enumerate}
\end{enumerate}
\item 
Let $u_0\in \dot H^{\frac{1}{2}}({\mathds{R}}^5)$ and $a(t)=e^{t\Delta}u_0$, $t\ge 0$.
Assume that $u$ and $v$ belong to ${\mathscr{C}}_b([0,\infty);\dot H^{\frac{1}{2}}({\mathds{R}}^5))$ and satisfy
$u=a+B(u,u)$ and $v=a+B(v,v)$.
\begin{enumerate}
\item 
Prove that $\sup_{0<t<\tau}\|u(t)+v(t)-2a(t)\|_{\dot H^{\frac{1}{2}}}\xrightarrow[\tau\to 0]{}0$.
\item
Let $E=\bigl\{t\in [0,\infty); u(t)=v(t)\bigr\}$.
\begin{enumerate}
\item 
Show that $E\neq \emptyset$ and that $E$ is closed.
\item 
Give an argument that proves $E=[0,\infty)$ if one can show that $E$ is open.
\end{enumerate}
\item 
Let $w:=u-v$. Prove that 
\[
w=B(w,u+v)=B(w,u+v-2a)+2B(w,a-a_\varepsilon)+2B(w,a_\varepsilon),
\]
and show that
\begin{align*}
\|w\|_{L^2(0,\tau;\dot H^{\frac{1}{2}}({\mathds{R}}^5))}\le &\|w\|_{L^2(0,\tau;\dot H^{\frac{1}{2}}({\mathds{R}}^5))}
\Bigl(C'\sup_{0<t<\tau}\|u(t)+v(t)-2a(t)\|_{\dot H^{\frac{1}{2}}}\\
&\hspace*{.5cm}+2C'\sup_{t>0}\|a_\varepsilon(t)-a(t)\|_{\dot H^{\frac{1}{2}}}
+2C\tau^{\frac{3}{4}}\sup_{0<t<\tau}\|a_\varepsilon\|_{\dot H^2}\Bigr).
\end{align*}
\item 
Prove that there exists $\tau_0>0$ such that $[0,\tau_0)\subset E$.
\end{enumerate}
\item 
How can you conclude?
\end{enumerate}
\end{exercise}
 
\subsection{Second subject}

This is a replacement subject for a student who couldn't make it on the day of the exam.

\begin{exercise}
As in the script, we use the notation 
 \[
 e^{t\Delta}u_0={\mathscr{F}}_\xi^{-1}\bigl(\xi\mapsto e^{-t|\xi|^2}{\mathscr{F}}_x(u_0)(\xi)\bigr)
 \] 
 for $t>0$ and $u_0\in \dot H^s({\mathds{R}}^n)$ ($s\in {\mathds{R}}$).
\begin{enumerate}
\item 
Prove that for all $u_0\in \dot H^s({\mathds{R}}^n)$ and all $\alpha>0$, the function $(-\Delta)^\alpha e^{t\Delta}u_0$ belongs
to $\dot H^s({\mathds{R}}^n)$ for all $t>0$ and
\[
\|(-\Delta)^\alpha e^{t\Delta}u_0 \|_{\dot H^s}\le \alpha^\alpha e^{-\alpha} \, t^{-\alpha}\,\|u_0\|_{\dot H^s}.
\]
\item
Let $u_0\in L^2({\mathds{R}}^n)$. 
\begin{enumerate}
\item 
Prove that $(t,x)\mapsto \bigl(e^{t\Delta}u_0\bigr)(x)$ 
belongs to the space $L^2(0,\infty;\dot H^1({\mathds{R}}^n))$ and that
\[
\bigl\|t\mapsto e^{t\Delta}u_0\bigr\|_{L^2(0,\infty;\dot H^1({\mathds{R}}^n))}=\tfrac{1}{\sqrt{2}}\,\|u_0\|_2
\]
{\small\sl Hint: On the Fourier side in the $x$ variable, use Fubini Theorem.}
\item 
Prove that $(t,x)\mapsto \partial_t\bigl(t\mapsto e^{t\Delta}u_0\bigr)(x)$ belongs to $L^2(0,\infty;\dot H^{-1}({\mathds{R}}^n))$.
\item 
Prove that $(t,x)\mapsto \bigl(e^{t\Delta}u_0\bigr)(x)$ belongs to the space $L^4(0,\infty;\dot H^{\frac{1}{2}}({\mathds{R}}^n))$.

\noindent 
Prove the estimate $\|t\mapsto e^{t\Delta}u_0\|_{L^4(0,\infty;\dot H^{\frac{1}{2}}({\mathds{R}}^n))}\le \frac{1}{\sqrt{2}}\,\|u_0\|_2$.
\end{enumerate}
\end{enumerate}
\end{exercise}

\begin{exercise}
In this exercise, we will investigate the equation $\partial_tu-\Delta u=(-\Delta)^{\frac{1}{4}}(u^2)$, for $t>0$ and $x\in{\mathds{R}}^3$
with an initial condition $u(0,\cdot)=u_0\in L^2({\mathds{R}}^3)$. For the definition of $(-\Delta)^{\frac{1}{4}}$, see the script.
\begin{enumerate}
\item 
Let $s\ge 0$ and $v\in H^s({\mathds{R}}^3)$. For $\lambda >0$, denote by $v_\lambda$ the function $x\mapsto v(\lambda x)$.
Prove that $(-\Delta)^{\frac{1}{4}}v_\lambda(x)=\sqrt{\lambda} \, (-\Delta)^{\frac{1}{4}}v (\lambda x)$ for $x\in {\mathds{R}}^3$.
\item 
Assume that $u$ is a solution of $\partial_tu-\Delta u=(-\Delta)^{\frac{1}{4}}(u^2)$ on $(0,\infty)\times {\mathds{R}}^3$. Find $\alpha\in{\mathds{R}}$
such that $u_{\lambda,\alpha}:(t,x)\mapsto \lambda^\alpha u(\lambda^2t,\lambda x)$ is also solution of the same equation for all $\lambda>0$.
\item 
Find $s\ge 0$ such that $\|u_{\lambda,\alpha}(t,\cdot)\|_{\dot H^s}=\|u(t,\cdot)\|_{\dot H^s}$ for all $t\in {\mathds{R}}$, all
$\lambda >0$ and $\alpha$ found in the previous question.
\item
Find $\sigma\ge 0$ such that 
$\|u_{\lambda,\alpha}\|_{L^2(0,\infty;\dot H^{\sigma}({\mathds{R}}^3))}=\|u\|_{L^2(0,\infty;\dot H^{\sigma}({\mathds{R}}^3))}$ for all 
$\lambda >0$ and $\alpha$ as before.
\item (This was not part of the exam!)
Prove existence and uniqueness of mild solutions of the equation: global existence for initial condition $u_0$ small in the $L^2$ norm and 
local existence otherwise; uniqueness in the space ${\mathscr{C}}([0,T),L^2({\mathds{R}}^3))$ where $T>0$ is the existence time of the solution.
\end{enumerate}
\end{exercise}


\begin{bibdiv}
\begin{biblist}

\end{biblist}
\end{bibdiv}


\begin{thebibliography}{11}
\addcontentsline{toc}{section}{References}

\bib{BCD11}{book}{
 author = {Bahouri, Hajer},
 author = {Chemin, Jean-Yves},
 author = {Danchin, Rapha{\"e}l},
 title = {Fourier analysis and nonlinear partial differential equations},
 series = {Grundlehren Math. Wiss.},
 volume = {343},
 year = {2011},
 publisher = {Berlin: Springer},
}

\bib{BL76}{book}{
 author = {Bergh, J{\"o}ran},
 author = {L{\"o}fstr{\"o}m, J{\"o}rgen},
 title = {Interpolation spaces. {An} introduction},
 series = {Grundlehren Math. Wiss.},
 volume = {223},
 year = {1976},
 publisher = {Springer, Cham},
}

\bib{Da20}{unpublished}{
author = {Danchin, Rapha{\"e}l},
title = {Fourier analysis methods for models of nonhomogeneous fluids},
year = {2020},
note = {https://perso.math.u-pem.fr/danchin.raphael/cours/coursM2-20.pdf},
}

\bib{FK64}{article}{
 author = {Fujita, Hiroshi},
 author = {Kato, Tosio},
 title = {On the {Navier}-{Stokes} initial value problem. {I}},
 journal = {Arch. Ration. Mech. Anal.},
 volume = {16},
 pages = {269--315},
 year = {1964},
}

\bib{Mo99}{article}{
 author = {Monniaux, Sylvie},
 title = {Unicit{\'e} des solutions ``mild'' de l'{\'e}quation de {Navier}-{Stokes} et r{\'e}gularit{\'e} maximale {{\(L^p\)}}},
 journal = {C. R. Acad. Sci., Paris, S{\'e}r. I, Math.},
 volume = {328},
 number = {8},
 pages = {663--668},
 year = {1999},
}

\bib{Mo09}{incollection}{
 author = {Monniaux, Sylvie},
 title = {Maximal regularity and applications to {PDEs}},
 booktitle = {Analytical and numerical aspects of partial differential equations. Notes of a lecture series. Lectures held at the Technische Universit\"at Berlin, Berlin, Germany, 2007--2009},
 pages = {247--287},
 year = {2009},
 publisher = {Berlin: Walter de Gruyter},
}

\bib{Mo24}{incollection}{
 author = {Monniaux, Sylvie},
 title = {Maximal regularity as a tool for partial differential equations},
 booktitle = {Modern problems in PDEs and applications. Extended abstracts of the 2023 GAP Center summer school, Ghent, Belgium, August 23 -- September 2, 2023},
 pages = {81--93},
 year = {2024},
 publisher = {Cham: Birkh{\"a}user},
}
\end{thebibliography}
\end{document}